\documentclass[10pt, twoside]{article}
\usepackage{amsmath,amsthm,amssymb}
\usepackage{times}
\usepackage{enumerate}
\usepackage{color}
\usepackage[colorlinks=true]{hyperref}

\def\titlerunning#1{\gdef\titrun{#1}}
\makeatletter
\def\author#1{\gdef\autrun{\def\and{\unskip, }#1}\gdef\@author{#1}}
\def\address#1{{\def\and{\\\hspace*{18pt}}\renewcommand{\thefootnote}{}%
\footnote {#1}}%
\markboth{\autrun}{\titrun}}
\makeatother
\def\email#1{e-mail: #1}
\def\subjclass#1{{\renewcommand{\thefootnote}{}%
\footnote{\emph{Mathematics Subject Classification (2010):} #1}}}
\def\keywords#1{\par\medskip
\noindent\textbf{Keywords.} #1}

\newtheorem{theorem}{Theorem}[section]
\newtheorem{lemma}[theorem]{Lemma}

\newtheorem{corollary}[theorem]{Corollary}

\hypersetup{linkcolor=blue,urlcolor=red,citecolor=red}

\theoremstyle{definition}

\newtheorem{remark}[theorem]{Remark}

\numberwithin{equation}{section}

\begin{document}

%%%%% To ease editing, add:

\baselineskip=17pt

%%%%%%%%%%%%%%%%

%% In the running head, give an abbreviation of the title.
\titlerunning{}

\title{Observability  and unique continuation inequalities for the  Schr\"{o}dinger equation}

\author{Gengsheng Wang
\and
Ming Wang
\and
Yubiao Zhang}

\date{}

\maketitle

\address{
 G. Wang: Corresponding author, School of Mathematics and Statistics, Wuhan University, Wuhan, 430072, China; \email{wanggs62@yeah.net}
\and
M. Wang: School of Mathematics and Physics, China University of Geosciences, Wuhan,
430074, China; \email{mwangcug@outlook.com}
\and
Y. Zhang: Center for Applied Mathematics, Tianjin University, Tianjin, 300072, China; \email{yubiao\b{ }zhang@whu.edu.cn}}


%%%%%%%%

\begin{abstract}
In this paper, we present several observability and unique continuation inequalities for the free Schr\"{o}dinger equation in the whole space. The observations in these inequalities are made either at two points in time or one point in time. These inequalities correspond to different kinds of  controllability for the free Schr\"{o}dinger equation.
We also find  that the observability inequality at two points in time is equivalent to the uncertainty principle built up in  \cite{PJ}.

%% Keywords are optional
\end{abstract}

\section{Introduction}

An interesting   unique continuation property for Schr\"{o}dinger equations was contained in  \cite{IK}
(see also \cite{KPV}). It says that if $u$ solves the following  Schr\"{o}dinger equation:
\color{black}
\begin{eqnarray}\label{0611-origin}
i\partial_tu + \Delta u + Vu=0
\;\;\mbox{in}\;\;  \mathbb R^n\times(0,1),
\end{eqnarray}
(with
a time-dependent potential $V$
  in some suitable conditions and with $n\in \mathbb{N}^+\triangleq\{1,2,\dots\}$),
then
$$
u=0\;\;\mbox{in}\;\;B^c_R(0)\times\{0,1\}\Rightarrow u\equiv0.
$$
Here, $R>0$,  $B_R(0)$ is the closed ball in  $\mathbb R^n$, centered at the origin and  of radius $R>0$, and $B^c_R(0)$ denotes the complement of
   $B_R(0)$.
         In \cite{EKPV-3-1} (see also \cite[Theorems 3-4]{EKPV-6}),  it was presented that if  $u$ solves (\ref{0611-origin}) (with $V$ in some suitable conditions) and verifies that
   \begin{eqnarray*}
    \| e^{|x|^2/\alpha^2} u(x,0)\|_{L^2(\mathbb R^n;\mathbb C)}
    + \| e^{|x|^2/\beta^2} u(x,1)\|_{L^2(\mathbb R^n;\mathbb C)} <\infty
   \end{eqnarray*}
   for some positive constants $\alpha,\,\beta$ with $\alpha\beta<4$, then $u\equiv0$. It further proved that when $\alpha\beta=4 $, such property fails.
   The above mentioned two properties
  can be treated as  the qualitative unique continuation at two points in time.
It is natural to ask if one can have  an observability inequality at two points in time?

 In this paper, we will present several observability  and unique continuation inequalities (at either two points in time or one point in time)  for the following free Schr\"{o}dinger equation (or the
 Schr\"{o}dinger equation, for simplicity):
   \begin{eqnarray}\label{0229-sch-1}
\left\{\begin{array}{lll}
        i\partial_t u(x,t) + \Delta u(x,t) = 0,  &(x,t)\in  \mathbb R^n\times (0,\infty),\\
        u(x,0)\in L^2(\mathbb R^n;\mathbb C).
       \end{array}
\right.
\end{eqnarray}
  (Here and throughout this paper,  $n\in\mathbb N^+$ is arbitrarily fixed.)
  From perspective of applications, these inequalities correspond different controllabilities for the Schr\"{o}dinger equation.

  Throughout this paper, we write either $u(x,t;u_0)$ (with $(x,t)\in  \mathbb R^n\times (0,\infty)$) or $e^{i\Delta t}u_0$ (with $t\geq 0$) for
  the solution  of (\ref{0229-sch-1}) with the initial condition that $u(x,0)=u_0(x)$ over $\mathbb{R}^n$; The  Fourier transform  of $f\in L^1(\mathbb R^n;\mathbb C) \cap L^2(\mathbb R^n;\mathbb C)$ is given by
\begin{eqnarray*}
 \hat f(\xi) = \frac{1}{(2\pi)^{n/2}} \int_{\mathbb R^n} f(x) e^{-ix\cdot\xi} \,\mathrm dx, ~\xi\in\mathbb R^n
\end{eqnarray*}
and extended to all of $L^2(\mathbb R^n;\mathbb C)$ in the usual way;
%Denote by  $B_{r}(x)$ the closed ball in $\mathbb R^n$, centered at $x$ and of radius $r>0$;
Write respectively $A^c$  and  $|A|$
 for the complement  and the Lebesgue measure of a set  $A$ in $\mathbb R^n$;
 For each subset $A\subset\mathbb R^n$ and each $\lambda\in \mathbb R$, we let   $\lambda A\triangleq\{\lambda x\,:\, x\in A\}$; For all $a, b\in \mathbb{R}$, we write $a\wedge b \triangleq \min\{a,b\}$; For each $x\in \mathbb{R}^n$,  $|x|$ denotes to the $\mathbb{R}^n$-Euclidean norm of $x$;
 $\omega_n$ denotes the volume of the unit ball in $\mathbb R^n$.

%\noindent (i) For each bounded open nonempty subset $E\subset\mathbb R^n$, the mean width of $E$ is defined in the following manner
%\begin{eqnarray}\label{def-mean-width}
% w(E)  \triangleq  \frac{1}{|S^{n-1}|} \int_{S^{n-1}} \left(\sup_{x\in E} \langle x,\vec{v} \rangle_{\mathbb R^n}
% - \inf_{x\in E} \langle x,\vec{v} \rangle_{\mathbb R^n} \right) \,\mathrm d\sigma(\vec{v}),
%\end{eqnarray}
%where $|S^{n-1}|$ denotes the Hausdorff measure of $S^{n-1}$.
%
%\noindent(ii) For each $a>0$, $\beta>0$ and $N>0$, we define two subspaces $X_{a,\beta}$ and $\hat X_N$ as follows:
%\begin{eqnarray}\label{wangyuan2.11}
%X_{a,\beta}\triangleq\left\{f\in L^2(\mathbb{R}^n;\mathbb C)\; :\;  \int_{\mathbb R^n} |f(x)|^2 e^{a|x|^\beta} \,\mathrm dx<\infty\right\},
%\end{eqnarray}
%and
%\begin{equation}\label{wangmin2.17}
%\hat X_N\triangleq \left\{f\in L^2(\mathbb R^n;\mathbb C)\;:\; \;\mbox{supp}\; f\subset B_N \right\}
%\end{equation}
%where $B_N=B_N(0)$, $B_r(x)$ is the ball at $x$ with radius $r>0$.

There are three main theorems in this paper. The first one presents
 an observability inequality at two points in time for the equation (\ref{0229-sch-1}).

\begin{theorem}\label{theorem1}
Given ${x^\prime},\,{x^{\prime\prime}}\in\mathbb R^n$, $r_1,\,r_2>0$ and $T>S\geq0$,    there is a positive constant
$C\triangleq C(n)$ so that
\begin{eqnarray}\label{0229-sch-th1-0}
 \int_{\mathbb R^n}  |u_0(x)|^2 \,\mathrm dx
 &\leq&
 C e^{C r_1 r_2  \frac{1}{T-S}  }
 \color{black}
 \Big(\int_{B_{r_1}^c({x^\prime})}  |u(x,S;u_0)|^2 \,\mathrm dx
 + \int_{B_{r_2}^c({x^{\prime\prime}})}  |u(x,T;u_0)|^2 \,\mathrm dx\Big)
\end{eqnarray}
for all  $u_0\in L^2(\mathbb R^n;\mathbb C)$.

\end{theorem}

Several remarks on Theorem \ref{theorem1} are given in order:

\begin{itemize}
  \item Theorem \ref{theorem1} can be explained in the following manner: The integral on the left hand side of (\ref{0229-sch-th1-0})  can be treated as a recover term,
  while    the integrals on the right hand side of  (\ref{0229-sch-th1-0}) are regarded as observation terms.
    The inequality (\ref{0229-sch-th1-0})
     is understood as follows: Through  observing  a solution at two different points in time, each time outside of  a ball, one can estimate the recover term
     (which says, in plain language,  that one can recover this solution). This inequality is equivalent to the exact controllability for the impulse controlled
      Schr\"{o}dinger equation with controls acting at two points in time,  each time outside of  a ball (see Subsection 5.2).

  \item The observability inequality (\ref{0229-sch-th1-0}) seems to be new for us. Most   observability inequalities for Schr\"{o}dinger equations, in published papers, have  observations in time intervals. For instance, the paper \cite{Lebeau}
 presents an observability inequality for the Schr\"{o}dinger equation on a bounded domain $\Omega$ (in $\mathbb{R}^n$),  with an analytic boundary $\partial\Omega$.
  In that inequality, the observation is made over $\hat \omega\times(0,T)$, where $T>0$ and $\hat \omega\subset\partial\Omega$ is a subdomain satisfying the Geometric Control Condition.
     This condition was introduced in \cite{BLR} and then was used in \cite{DGL} to study the stabilization property and the exact controllability
for the nonlinear Schr\"{o}dinger equation on a two dimensional compact Riemannian manifold without boundary.
   The paper \cite{Phung-1} builds up an  observability estimate for the homogenous Schr\"{o}dinger equation on a bounded domain $\Omega$. In that inequality, the observation is made over
  $\omega\times(0,T)$, where $T>0$ and $\omega\subset\Omega$ is a subdomain satisfying the Geometric Control Condition.
  More recently, the paper \cite{ALM} (see \cite[Theorem 1.2]{ALM}) presents an observability inequality
  for Schr\"{o}dinger equations (with some potentials) on the disk of $\mathbb R^2$. The observation is made over $\omega\times (0,T)$, where $\omega$ is an open subset (in the disk) which may not satisfy   the Geometric Control Condition.

  \item The inequality (\ref{0229-sch-th1-0}) is ``optimal" in the following sense:  First, $\forall\,A\subset\mathbb R^n$, with  $m(A^c)>0$,  $\forall\,T>0$, the following conclusion is not true (see $(b)$ of  Remark~\ref{tianjinremark4.2}):  $\exists\,C>0$  so that
\begin{eqnarray*}
 \int_{\mathbb R^n}  |u_0(x)|^2 \,\mathrm dx
 \leq C \int_{A}  |u(x,T;u_0)|^2 \,\mathrm dx,
 ~\forall\, u_0\in L^2(\mathbb R^n;\mathbb C).
\end{eqnarray*}
This means that we cannot recover a solution by observing it at one  point in time and over a subset $A\subset \mathbb{R}^n$, with $|A^c|>0$; Second, $\forall\,{x^\prime},\,{x^{\prime\prime}}\in\mathbb R^n$, $r_1,\,r_2>0$ and $T>S\geq0$, the following conclusion is not true (see
$(a)$ of Remark~\ref{tianjinremark4.2}):  $\exists\,C>0$ so that
$$
 \int_{\mathbb R^n}  |u_0(x)|^2 \,\mathrm dx
 \leq C \Big(\int_{B_{r_1}^c({x^\prime})}  |u(x,S;u_0)|^2 \,\mathrm dx
 + \int_{B_{r_2}({x^{\prime\prime}})}  |u(x,T;u_0)|^2 \,\mathrm dx\Big),~\forall\, u_0\in L^2(\mathbb R^n;\mathbb C).
$$
 This means that we cannot recover a solution by observing it at two different points in time,  one time in a ball, while another time outside of a ball;
And last, $\forall\,{x^\prime},\,{x^{\prime\prime}}\in\mathbb R^n$, $r_1,\,r_2>0$ and $T>S\geq0$, the following conclusion is not true (see
$(c)$ of Remark~\ref{tianjinremark4.2}):  $\exists\,C>0$ so that
$$
 \int_{\mathbb R^n}  |u_0(x)|^2 \,\mathrm dx
 \leq C \Big(\int_{B_{r_1}^c({x^\prime})}  |u(x,S;u_0)|^2 \,\mathrm dx
 + \int_0^T\int_{B_{r_2}({x^{\prime\prime}})}  |u(x,t;u_0)|^2 \,\mathrm dx \mathrm dt\Big),~\forall\, u_0\in L^2(\mathbb R^n;\mathbb C).
$$
This can be comparable with the work in \cite{MZ-1}.

 \item   The proof of (\ref{0229-sch-th1-0})  is based on two properties as follows: First, the uncertainty principle built up in  \cite{PJ}; Second, the equivalence  between the uncertainty principle and
   the observability estimate which grows like  (\ref{0229-sch-th1-0}).
   The aforementioned  equivalence is indeed a connection between the uncertainty principle and the observability (at two time points) for the  Schr\"{o}dinger equation.
   Such equivalence
  is obtained in this paper (see  Lemma \ref{0309-sch-eq-twopoints}). Its proof relies on the identity \cite[(1.2)]{EKPV-2}
   (see   (\ref{0229-sch-th1-4}) in our paper).

  \item  The inequality (\ref{0229-sch-th1-0}) can be extended to the case
  where    $B_{r_1}^c({x^\prime})$ and $B_{r_2}^c({x^{\prime\prime}})$ are replaced by
two measurable sets $A^c$ and $B^c$, with $|A|<\infty$ and $|B|<\infty$. This can be easily seen from the proof of (\ref{0229-sch-th1-0}), as well as  Theorem~\ref{uncertainty-principle}
(which is the uncertainty principle built up in  \cite{PJ}) and Lemma \ref{0309-sch-eq-twopoints}.

\item From Theorem \ref{theorem1}, one can directly derive the following observability inequality: Given $x_0\in\mathbb R^n$, $r>0$ and $T>0$, there exists $C\triangleq C(n)>0$ so that
\begin{eqnarray*}
 \int_{\mathbb R^n} |u_0(x)|^2  \,\mathrm dx
 \leq Ce^{C r^2/T} \int_0^T  \left( \int_{B_{r}^c (x_0)} |u(x,t;u_0)|^2  \,\mathrm dx  \right)^{1/2} \,\mathrm dt
 \;\;\mbox{for all}\;\;
 u_0\in L^2(\mathbb R^n;\mathbb C).
\end{eqnarray*}
 This inequality is equivalent to the standard  $L^\infty$-exact controllability for the  Schr\"{o}dinger equation. The later is comparable to  \cite[Theorem 3.1]{RZhang}.

\end{itemize}

The second main theorem  gives a  unique continuation inequality at one time point for a class of solutions to the equation (\ref{0229-sch-1}). (This class of solutions consists of solutions whose initial data have exponential decay at infinity.)
 \begin{theorem}\label{theorem3}
 The following conclusions are true for all  $r>0$, $a>0$ and $T>0$:

\noindent (i)  There is   $C\triangleq C(n)>0$ and $\theta\triangleq\theta(n)\in(0,1)$, depending only on $n$,  so that
\begin{eqnarray}\label{0229-two-points-1}
 \int_{\mathbb R^n}|u_0(x)|^2 \,\mathrm dx
 \leq C \left( 1+\frac{r^n}{(aT)^n} \right) \left(
 \int_{B_r^c(0)}|u(x,T;u_0)|^2 \,\mathrm dx
 \right)^{\theta^{1+\frac{r}{aT}}}
 \left(\int_{\mathbb R^n}e^{ a|x|} |u_0(x)|^2 \,\mathrm dx\right)^{1-\theta^{1+\frac{r}{aT}}}
 \end{eqnarray}
for all  $u_0\in C_0^\infty(\mathbb R^n; \mathbb C)$.

\noindent (ii) There is $C\triangleq C(n)>0$ so that for any  $\beta>1$ and $\gamma\in(0,1)$,
\begin{eqnarray}\label{0229-two-points-beta}
 \int_{\mathbb R^n}|u_0(x)|^2 \,\mathrm dx
 \leq C e^{ \left(\frac{C^{\beta}r^{\beta}}{a(1-\gamma) T^\beta}\right)^{\frac{1}{\beta-1}} }
 \left(\int_{B_r^c(0)}|u(x,T;u_0)|^2 \,\mathrm dx\right)^{\gamma}
 \left(\int_{\mathbb R^n}e^{ a|x|^\beta} |u_0(x)|^2 \,\mathrm dx\right)^{1-\gamma},
  \end{eqnarray}
for all $u_0\in C_0^\infty(\mathbb R^n; \mathbb C)$.

\noindent (iii) Let $\alpha(s)$, $s\in\mathbb R^+$, be an increasing function  with $\lim_{s\rightarrow\infty} \frac{\alpha(s)}{s}=0$. Then for each   $\gamma\in(0,1)$, there is no  positive constant $C$ so that
\begin{eqnarray*}
 \int_{\mathbb R^n}|u_0(x)|^2 \,\mathrm dx
 \leq C\left(\int_{B_r^c(0)}|u(x,T;u_0)|^2 \,\mathrm dx\right)^{\gamma}
 \left(\int_{\mathbb R^n}e^{a\alpha(|x|)} |u_0(x)|^2 \,\mathrm dx\right)^{1-\gamma}\;\;\mbox{for all}\;\;u_0\in C_0^\infty(\mathbb R^n; \mathbb C).
 \end{eqnarray*}

\end{theorem}

Several remarks on Theorem \ref{theorem3} are given in order:

\begin{itemize}

\item  The motivation to build up Theorem~\ref{theorem3} is as follows:  According to the third remark after Theorem \ref{theorem1}, one cannot recover a solution by observing it at one time point and outside of a ball.
Theorem~\ref{theorem3} tells us what we can expect by observing  solutions at one time point and outside of a ball. The detail on the expectations will be explained in the next remark.

\item
The inequality (\ref{0229-two-points-1}) is a kind of unique continuation inequality at one point in time. From it, one can easily see that
$$
e^{\frac{a|x|}{2}}u_0(x)\in L^2(\mathbb{R}^n;\mathbb{C})\;\;\mbox{and}\;\;u(x,T;u_0)=0\;\;\mbox{over}\;\;B_r^c(0)\Rightarrow u(x,t;u_0)=0\;\;\mbox{over}\;\;\mathbb{R}^n\times[0,\infty).
$$
This inequality
can  also be explained from two perspectives. Perspective One: The integral on the left hand side of (\ref{0229-two-points-1})
is treated as a recover term,
  while on the right hand side of (\ref{0229-two-points-1}),   the integral over  $B_r^c(0)$ is regarded as observation term and the integral over the whole space $\mathbb{R}^n$ is viewed as a
   prior term (with respect to initial data) which provide some prior information on initial data  ahead of observations.
     The inequality (\ref{0229-two-points-1}) can be explained in the following way: If one knows in advance that the initial datum of a solution has  an exponential decay  at infinity,
       then  by observing this solution at one  point in time and outside of  a ball, one can  estimate the recover term (which says, in plain language,  that one can recover this solution). Perspective Two: Notice that  (\ref{0229-two-points-1}) is equivalent to that $\exists\,C>0$ and $\theta\in(0,1)$ s.t.  $\forall\,r,\,a,\,T>0$ and $\varepsilon>0$,
    \begin{eqnarray*}
     \int_{\mathbb R^n} |u_0(x)|^2 \,\mathrm dx
     \leq C \left( 1+\frac{r^n}{(aT)^n} \right) \left( \varepsilon^{1-\theta^{-1-\frac{r}{aT}}} \int_{B_r^c(0)} |u(x,T;u_0)|^2 \,\mathrm dx +  \varepsilon \int_{\mathbb R^n} e^{a|x|} |u_0(x)|^2 \,\mathrm dx \right)
    \end{eqnarray*}
for all $u_0\in C^\infty_0(\mathbb R^n;\mathbb C)$. Thus,
the inequality (\ref{0229-two-points-1}) can be understood as follows: Through observing a solution at one  point in time and outside of a ball, we can approximately recover this solution, with the  error:
 $$
  C \left( 1+\frac{r^n}{(aT)^n} \right) \varepsilon \int_{\mathbb R^n} e^{a|x|} |u_0(x)|^2 \,\mathrm dx.
 $$
Notice that if $\int_{\mathbb R^n} e^{a|x|} |u_0(x)|^2=\infty$, then the error is $\infty$.

\item The inequality (\ref{0229-two-points-1})   is equivalent to a kind of approximate controllability for the impulse controlled
      Schr\"{o}dinger equation with controls acting at one point in time. Notice that such controllability is not a standard controllability  (see Subsection 5.2).

\item Theorem~\ref{theorem3} is ``optimal" from two perspectives. Perspective One: If $\beta\geq 1$, then
for any $r>0$, $a>0$ and $T>0$, there is $C>0$ and $\theta\in (0,1)$ so that
\begin{eqnarray*}
\int_{\mathbb R^n}|u_0(x)|^2 \,\mathrm dx
 \leq C \left(
 \int_{B_r^c(0)}|u(x,T;u_0)|^2 \,\mathrm dx
 \right)^{\theta}
 \left(\int_{\mathbb R^n}e^{ a|x|^\beta} |u_0(x)|^2 \,\mathrm dx\right)^{1-\theta},
 ~\forall\, u_0\in C_0^\infty(\mathbb R^n;\mathbb C),
\end{eqnarray*}
while if $\beta\in (0,1)$, then for any $r>0$, $a>0$ and $T>0$, there is no $C>0$ or $\theta\in (0,1)$
so that the above inequality holds.
Perspective Two: For each $r>0$, $a>0$ and $T>0$, the following conclusion is not true (see $(c)$ of  Remark~\ref{tianjinremark4.2}):  $\exists\,C>0$ and $\exists\,\theta\in(0,1)$ so that
\begin{eqnarray*}
\int_{\mathbb R^n}|u_0(x)|^2 \,\mathrm dx
 \leq C \left(
 \int_{B_r(0)}|u(x,T;u_0)|^2 \,\mathrm dx
 \right)^{\theta}
 \left(\int_{\mathbb R^n}e^{ a|x|} |u_0(x)|^2 \,\mathrm dx\right)^{1-\theta},
 ~\forall\, u_0\in C_0^\infty(\mathbb R^n;\mathbb C).
\end{eqnarray*}
The above optimality implies in some sense that the choice of the weight $e^{a|x|}$, $x\in\mathbb R^n$ (with $a>0$) is reasonable (to ensure the type of unique continuation estimates build up in (i) and (ii) of Theorem \ref{theorem3}). In plain language, other types of weights are not expected.

\item The proof of Theorem \ref{theorem3} is mainly based on
\cite[Theorem 1.3]{AE}, which gives
an analytic interpolation inequality (see also                                             \cite{Vessella}), and                       an estimate for some kind of Euler integral in high dimension built up in  Lemma \ref{0229-euler-1} of the current paper and the identity \cite[(1.2)]{EKPV-2}
   (see   (\ref{0229-sch-th1-4}) in our paper).

\end{itemize}

The last main theorem gives another kind of   unique continuation inequality  at  one time point for a class of solutions to the equation (\ref{0229-sch-1}).
\begin{theorem}\label{theorem4}
Given ${x^\prime},\,{x^{\prime\prime}}\in\mathbb R^n$, $r_1,\,r_2>0$, $a>0$ and $T>0$, the following estimate holds for all $u_0\in C_0^\infty(\mathbb R^n; \mathbb C)$:
 \begin{eqnarray}\label{0405-sch-th4-control}
 & & \int_{B_{r_2}({x^{\prime\prime}})} |u(x,T;u_0)|^2 \, \mathrm dx
 \\
 &\leq&   C r_2^n \big((aT)\wedge r_1\big)^{-n}
  \left(\int_{B_{r_1}({x^\prime})} |u(x,T;u_0)|^2 \, \mathrm dx\right)^{\theta^p}
 \left(\int_{\mathbb R^n}e^{ a|x|} |u_0(x)|^2 \,\mathrm dx\right)^{1-\theta^p},
 \nonumber
\end{eqnarray}
where $C\triangleq C(n)>0$, $\theta\triangleq\theta(n)\in(0,1)$ and
\begin{eqnarray}\label{0410-def-p}
  p \triangleq 1+\frac{|{x^\prime}-{x^{\prime\prime}}|+r_1+r_2}{(aT)\wedge r_1}.
\end{eqnarray}
\end{theorem}

Several remarks on Theorem \ref{theorem4} are given in order:

\begin{itemize}

\item  The motivation to present Theorem~\ref{theorem4} is as follows:  According to Theorem \ref{theorem3}, as well as
Perspective Two in
the fourth remark after Theorem \ref{theorem3}, it should be interesting to ask for what we can expect by observing solutions at one time point and in a ball.

\item The inequality (\ref{0405-sch-th4-control}) is a kind of unique continuation inequality at one point in time. From it, one can easily check  that
 $$
 e^{\frac{a|x|}{2}}u_0(x)\in L^2(\mathbb{R}^n;\mathbb{C})\;\;\mbox{and}\;\;u(x,T;u_0)=0\;\;\mbox{over}\;\;B_{r_1}({x^\prime})\Rightarrow u(x,t;u_0)=0\;\;\mbox{over}\;\;\mathbb{R}^n\times [0,\infty).
 $$
  Indeed, the left hand side of the above, together with  (\ref{0405-sch-th4-control}), indicates that  for each ${x^{\prime\prime}}\in\mathbb R^n$ and each $r_2>0$,
  $u(\cdot,T;u_0)=0$ over $ B_{r_2}({x^{\prime\prime}})$. Thus, $u(x,T;u_0)=0$ over $\mathbb{R}^n$. This leads to that $u(x,t;u_0)=0$ over $\mathbb{R}^n\times [0,\infty)$.

  From (\ref{0405-sch-th4-control}), we can also have that
  $$
  u_0=0\;\;\mbox{over}\;\;B_{r_2}^c({x^{\prime\prime}})\;\;\mbox{and}\;\;u(x,T;u_0)=0\;\;\mbox{over}\;\;B_{r_1}({x^\prime})\Rightarrow
  u(x,t;u_0)=0\;\;\mbox{over}\;\;\mathbb{R}^n\times [0,\infty).
  $$

  \item
The inequality (\ref{0405-sch-th4-control}) can also be explained from two perspectives. Perspective One: The integral on the left hand side of (\ref{0405-sch-th4-control})
is treated as a recover term,
  while on the right hand side of (\ref{0405-sch-th4-control}),   the integral over  $B_{r_1}({x^\prime})$ is regarded as observation term and the integral over the whole space $\mathbb{R}^n$ is viewed as a
   prior term.
     The inequality (\ref{0405-sch-th4-control}) can be explained in the following way: If one knows in advance that the initial datum of a solution has  an exponential decay  at infinity,
       then  by observing this solution at one  point in time and in  a ball, one can  estimate the recover term (which says, in plain language,  that one can recover this solution over $B_{r_2}({x^{\prime\prime}})$ at time $T$).
                     Perspective Two: The inequality (\ref{0405-sch-th4-control}) is equivalent to that $\exists\,C>0$ and $\theta\in(0,1)$ s.t. $\forall\,{x^\prime},\,{x^{\prime\prime}}\in\mathbb R^n$, $r_1,r_2>0$, $a,\,T>0$ and $\varepsilon>0$,
    \begin{eqnarray*}
     & & \int_{B_{r_2}({x^{\prime\prime}})} |u(x,T;u_0)|^2 \, \mathrm dx
  \\
 &\leq&   C r_2^n \big((aT)\wedge r_1\big)^{-n}
  \left(  \varepsilon^{1-\theta^{-p}} \int_{B_{r_1}({x^\prime})} |u(x,T;u_0)|^2 \, \mathrm dx   +
 \varepsilon \int_{\mathbb R^n}e^{ a|x|} |u_0(x)|^2 \,\mathrm dx\right)
 %\;\forall\; u_0\in C_0^\infty(\mathbb R^n;\mathbb C).
 \nonumber
    \end{eqnarray*}
         for all $u_0\in C_0^\infty(\mathbb R^n;\mathbb C)$.
        Here, $p$ is given by (\ref{0410-def-p}).
Thus,  the inequality (\ref{0405-sch-th4-control}) can be understood as follows: Through observing a solution at one  point in time and in a ball, we can approximately recover this solution over $B_{r_2}({x^{\prime\prime}})$ at time $T$, with the  error:
 $$
   C r_2^n \big((aT)\wedge r_1\big)^{-n} \varepsilon \int_{\mathbb R^n} e^{a|x|} |u_0(x)|^2 \,\mathrm dx.
 $$
If $\int_{\mathbb R^n} e^{a|x|} |u_0(x)|^2=\infty$, then the error is $\infty$.

 Notice that the recover terms in  (\ref{0229-two-points-1})  and (\ref{0405-sch-th4-control})  are different. In (\ref{0229-two-points-1}), one tries to recover approximately a solution over $\mathbb{R}^n\times\{0\}$, while (\ref{0405-sch-th4-control})  tries to recover a solution over $B_{r_2}({x^{\prime\prime}})\times\{T\}$.

    \item  The inequality (\ref{0405-sch-th4-control})   is equivalent to a kind of approximate null controllability for the initial controlled
      Schr\"{o}dinger equation with controls acting at one point in time. Notice that such controllability is not a standard controllability  (see Subsection 5.2).

\item The proof of Theorem \ref{theorem4} is mainly based on \cite[Theorem 1.3]{AE}, which gives
an analytic interpolation inequality (see also                                             \cite{Vessella}),
                     an estimate for some kind of Euler integral in high dimension built up in  Lemma \ref{0229-euler-1} and the identity \cite[(1.2)]{EKPV-2}
   (see   (\ref{0229-sch-th1-4}) in our paper).

\end{itemize}

We next present three consequences of the above main theorems.

\begin{theorem}\label{proposition3-1}
Given $r>0$, $T>0$ and  $N> 0$, the following estimate is true for all $u_0\in L^2(\Omega;\mathbb C)$ with supp\,$u_0\subset B_N(0)$:
\begin{eqnarray}\label{0229-sch-th4-control-3}
  \int_{\mathbb R^n} |u_0(x)|^2 \, \mathrm dx
 &\leq&  e^{C \big(1+ \frac{rN}{T}\big)}  \int_{B_r^c(0)} |u(x,T;u_0)|^2 \, \mathrm dx,
\end{eqnarray}
where $C \triangleq C(n)>0$.
\end{theorem}

\begin{theorem}\label{theorem5}
Given $x_0,\,{x^\prime}\in\mathbb R^n$, $r>0$, $a>0$, $b> 0$ and $T>0$, the following inequality holds for all  $u_0\in C_0^\infty(\mathbb R^n; \mathbb C)$ and  $\varepsilon\in (0,1)$:
\begin{eqnarray}\label{0405-sch-th5-control-1}
 & &  \int_{\mathbb R^n} e^{-b|x-{x^\prime}|} |u(x,T;u_0)|^2 \, \mathrm dx
 \\
 &\leq&  C(x_0,{x^\prime},r,a,b,T)  \left(\varepsilon \int_{\mathbb R^n}e^{ a|x|} |u_0(x)|^2 \,\mathrm dx
  + \varepsilon e^{\varepsilon^{-1-\frac{C b^{-1}}{(aT)\wedge r}}} \int_{B_{r}(x_0)} |u(x,T;u_0)|^2 \, \mathrm dx \right),
  \nonumber
\end{eqnarray}
where
\begin{eqnarray*}\label{0410-th1.5-constant}
  C(x_0,{x^\prime},r,a,b,T) \triangleq
  \exp\left\{ C \Big[1 + \frac{  |x_0-{x^\prime}|+r + b^{-1} }{(aT)\wedge r}  \Big] \right\},
\end{eqnarray*}
with $C\triangleq C(n)>0$.

 \end{theorem}

 \begin{theorem}\label{theorem6}
Given  $x_0\in\mathbb R^n$, $r>0$, $a>0$ and $T>0$, the following estimate is true for all  $u_0\in C_0^\infty(\mathbb R^n; \mathbb C)$ and  $\varepsilon\in (0,1)$:
  \begin{eqnarray}\label{0405-sch-th5-control-2}
 & &  \int_{\mathbb R^n}  |u_0(x)|^2 \, \mathrm dx
 \\
 &\leq&  {C}(x_0,r,a,T)   \left(
 \varepsilon \Big(\int_{\mathbb R^n} e^{a|x|} |u_0(x)|^2  \mathrm dx
 +  \|u_0\|^2_{H^{n+3}(\mathbb R^n;\mathbb C)}\Big)
  + \varepsilon  e^{e^{\varepsilon^{-2}}} \int_{B_{r}(x_0)} |u(x,T;u_0)|^2 \, \mathrm dx
   \right), \nonumber
\end{eqnarray}
where
\begin{eqnarray*}
 {C}(x_0,r,a,T)
 \triangleq (1+T)^{2n+6} \exp\left\{ C^{1+\frac{|x_0|+r+1}{(aT)\wedge r}} \right\},
\end{eqnarray*}
with  $C\triangleq C(n)>0$.
\end{theorem}

Two notes on Theorem~\ref{proposition3-1}-Theorem~\ref{theorem6} are as follows:

\begin{itemize}

\item The inequalities in Theorem~\ref{proposition3-1}-Theorem~\ref{theorem6} are different kinds of unique continuation at one time point for the  Schr\"{o}dinger equation. They  correspond to different kinds of controllability which   are not standard controllability (see Subsection 5.3).

\item  Theorem~\ref{proposition3-1} is a direct consequence of the conclusion (i) in Theorem \ref{theorem3}.
Theorem~\ref{theorem5} is a consequence of Theorem \ref{theorem4}. Theorem~\ref{theorem6} is based on Theorem \ref{theorem4}, as well as a regularity propagation property for the  Schr\"{o}dinger equation (presented  in Lemma~\ref{lemma-0411-th1.6} of this paper).

\end{itemize}

 The main novelties of this paper are as follows: $(a)$ We build up observability estimate at two points in time for the  Schr\"{o}dinger equation in $\mathbb{R}^n$.
  $(b)$ We present several unique continuation (or observability) inequalities at one point in time for the  Schr\"{o}dinger equation in $\mathbb{R}^n$. These inequalities correspond to different kinds of controllability.  $(c)$ We find an equivalence between the observability at two different points in time and the uncertainty principle built up in  \cite{PJ}
   (see  Lemma \ref{0309-sch-eq-twopoints}).

   It should be interesting to extend our results to the following equations: $(a)$ Schr\"{o}dinger equations with potentials  in $\mathbb{R}^n$. $(b)$ Homogeneous Schr\"{o}dinger equations
   on a bounded domain.

   For the observability and unique continuation inequalities for Schr\"{o}dinger equations, we would like to mention \cite{B,Cruz,EKPV-1, EKPV-2,EKPV-3,EKPV-3-1,EKPV-4,EKPV-5,EKPV-6,IK-1,IK,Laurent,Lebeau,Masuda,NP,Seo,
   BYZ,Zuazua} and the references therein. For the uncertainty principle, we would like to mention \cite{HJ, PJ, Kovrijkine, Nazarov,SST} and the references therein. We think of that the uncertainty principle built up in some of these papers may be used to get some observability estimates for Schr\"{o}dinger equations. For interpolation inequalities for heat equations, we would like to mention \cite{AEWC,PW,PWCZ,WK} and the references therein.

The rest of the paper is organized as follows: Section 2 proves Theorem~\ref{theorem1}-Theorem \ref{theorem4}. Section 3 proves  Theorem~\ref{proposition3-1}-Theorem~\ref{theorem6}.
Section 4 provides some further comments on the main results.
Section 5 presents  applications of Theorem~\ref{theorem1}-Theorem~\ref{theorem6} to the controllability for the  Schr\"{o}dinger equation.

\bigskip

\section{Proofs of the main results}

This section is devoted to proving Theorem~\ref{theorem1}-Theorem~\ref{theorem4}.

\subsection{Proof of Theorem \ref{theorem1}}

In this subsection, we will prove Theorem \ref{theorem1}.
We first introduce in Theorem~\ref{uncertainty-principle} the uncertainty principle built up in  \cite{PJ}, then
show in Lemma~\ref{0309-sch-eq-twopoints}
the equivalence between the uncertainty principle and the observability at two points in time, finally give the proof of Theorem \ref{theorem1}.

\begin{theorem}\label{uncertainty-principle}
Given subsets $S,\Sigma\subset\mathbb R^n$, with $|S|<\infty$ and $|\Sigma|<\infty$, there is a positive constant
\begin{eqnarray}\label{uncertainty-principle-ineq-constant}
 C(n,S,\Sigma)  \triangleq  Ce^{C\min\left\{|S||\Sigma|,|S|^{1/n}w(\Sigma),|\Sigma|^{1/n}w(S)\right\}},
\end{eqnarray}
with $C\triangleq C(n)$,  so that
 for each $f\in L^2(\mathbb R^n;\mathbb C)$,
\begin{eqnarray*}\label{uncertainty-principle-ineq}
 \int_{\mathbb R^n_x}  |f(x)|^2 \,\mathrm dx
 \leq  C(n,S,\Sigma)   \left( \int_{\mathbb R^n_x\setminus S}  |f(x)|^2 \,\mathrm dx + \int_{\mathbb R^n_\xi \setminus \Sigma}  |\hat f(\xi)|^2 \,\mathrm d\xi\right).
\end{eqnarray*}
Here, $w(S)$ (or $w(\Sigma)$) denotes the mean width of  $S$ (or $\Sigma$).

\end{theorem}

\begin{remark}
For the detailed definition of $w(S)$ (the mean width of  $S$), we refer the readers to \cite{PJ}. Here, we would like to mention what follows: First, when  $S$  is a bounded and open subset of $\mathbb{R}^n$, $w(S)<\infty$; Second, when $S$ is a ball in $\mathbb{R}^n$, $w(S)$ is the diameter of the ball.

\end{remark}

\begin{lemma}\label{0309-sch-eq-twopoints}
Let $A$ and $B$ be two measurable subsets of $\mathbb R^n$. Then the following propositions are equivalent:

\noindent(i) There exists a positive constant $C_1(n,A,B)$ so that for each $f\in L^2(\mathbb R^n;\mathbb C)$,
 \begin{eqnarray}\label{0309-sch-eq1-0-0}
 \int_{\mathbb R^n_x}  |f(x)|^2 \,\mathrm dx
 &\leq&  C_1(n,A,B)
 \left( \int_{A}  |f(x)|^2 \,\mathrm dx + \int_{B}  |\hat f(\xi)|^2 \,\mathrm d\xi\right).
\end{eqnarray}

\noindent(ii) There exists a positive constant $C_2\big(n,A,B\big)$ so that for each $T>0$ and each $u_0\in L^2(\mathbb R^n;\mathbb C)$,
\begin{eqnarray}\label{0309-sch-eq1-0}
 \int_{\mathbb R^n}  |u_0(x)|^2 \,\mathrm dx
 &\leq& C_2(n,A,B)  \Big(\int_{ A}  |u_0(x)|^2 \,\mathrm dx
 + \int_{2TB}  |u(x,T;u_0)|^2 \,\mathrm dx\Big).
\end{eqnarray}
%where $2TB\triangleq \{2Tx\,:\, x\in B\}$.

%\noindent(iii) There exists a positive constant $C_3\big(n,A,B\big)$ so that for each $T>0$ and each $u_0\in L^2(\mathbb R^n;\mathbb C)$,
%\begin{eqnarray}\label{0309-sch-eq1-0}
% \int_{\mathbb R^n}  |u_0(x)|^2 \,\mathrm dx
% &\leq& C_3(n,A,B)  \Big(\int_{ 2TA}  |u_0(x)|^2 \,\mathrm dx
% + \int_{B}  |u(x,T;u_0)|^2 \,\mathrm dx\Big).
%\end{eqnarray}
%where $2TA\triangleq \{2Tx\,:\, x\in A\}$.

Furthermore, when one of the above two propositions   holds, the constants $C_1(n,A,B)$ and $C_2(n,A,B)$ can be chosen as the same  number.

\end{lemma}

\begin{proof}
Divide the proof into the following two steps:

\vskip 5pt
\textit{Step 1. To show that (i)$\Rightarrow$(ii)}

\noindent Suppose that (i) is true for $C_1(n,A,B)$. We first claim that for all $T>0$ and  $u_0\in L^2(\mathbb R^n;\mathbb C)$,
\begin{eqnarray}\label{0229-sch-th1-1}
 & & \int_{\mathbb R^n}  |u_0(x)|^2 \,\mathrm dx
 \nonumber\\
 &\leq&   C_1\big(n,A,B\big) \left(\int_{A}  |u_0(x)|^2 \,\mathrm dx
 + \frac{1}{(2T)^n} \int_{2TB}  |\widehat{e^{i|\xi|^2/4T}u_0(\xi)}
 (x/2T)|^2 \,\mathrm dx \right) .
\end{eqnarray}
%where
%\begin{eqnarray}\label{0229-sch-th1-C1}
% \widetilde{C} \triangleq Ce^{C\min\left\{|A^c||B^c|/T^n,|A^c|^{1/n}w(B^c)/T,|B^c|^{1/n}w(A^c)/T\right\}}
%\end{eqnarray}
Indeed, for arbitrarily fixed $T>0$ and  $u_0\in L^2(\mathbb R^n;\mathbb C)$, we define  a function $\widetilde{u}_0(\cdot)$ over $\mathbb{R}^n$ in the following manner:
 %\begin{eqnarray}\label{0229-sch-th1-def-Bprime}
%  \widetilde{B} \triangleq \{x\in\mathbb R^n  ~:~  2Tx\in B\}
% \end{eqnarray}
% and
 \begin{eqnarray}\label{0229-sch-th1-2}
  \widetilde{u}_0(x)\triangleq   e^{i|x|^2/4T} u_0(x),~x\in \mathbb R^n.
 \end{eqnarray}
 It is clear that $\widetilde{u}_0\in L^2(\mathbb R^n;\mathbb C)$. Then by (i), we have (\ref{0309-sch-eq1-0-0}), with  $f=\widetilde{u}_0$, i.e.,
 \begin{eqnarray*}
   \int_{\mathbb R^n}  |\widetilde{u}_0(x)|^2 \,\mathrm dx
 &\leq& C_1\big(n,A,B\big)
      \left(\int_{A}  |\widetilde{u}_0(x)|^2 \,\mathrm dx
 + \int_{B}  |\widehat{\widetilde{u}_0}(x)|^2 \,\mathrm dx \right).
 \end{eqnarray*}
 This, along with  (\ref{0229-sch-th1-2}), leads to (\ref{0229-sch-th1-1}).

We next notice  from  \cite[(1.2)]{EKPV-2} that for all
%$T\in\mathbb R$
$T>0$ and  $u_0\in L^2(\mathbb R^n;\mathbb C)$,
  \begin{eqnarray}\label{0229-sch-th1-4}
  (2iT)^{n/2} e^{-i|x|^2/4T} u(x,T;u_0) =  \widehat{e^{i|\xi|^2/4T}u_0(\xi)}(x/2T),~x\in\mathbb R^n.
 \end{eqnarray}
  Then from (\ref{0229-sch-th1-1}) and (\ref{0229-sch-th1-4}), it follows that
 \begin{eqnarray*}
 \int_{\mathbb R^n}  |u_0(x)|^2 \,\mathrm dx
 &\leq&  C_1\big(n,A,B\big) \left(\int_{A}  |u_0(x)|^2 \,\mathrm dx
 + \int_{2TB}  |u(x,T;u_0)|^2 \,\mathrm dx \right).
\end{eqnarray*}
%Since the conservation law holds for Schr\"{o}dinger equations, the above implies that
%\begin{eqnarray*}
% \int_{\mathbb R^n}  |u(T,x;u_0)|^2 \,\mathrm dx
% &\leq& \widehat C_1 \Big(\int_{\mathbb R^n \setminus A}  |u_0(x)|^2 \,\mathrm dx
% + \int_{\mathbb R^n \setminus B}  |u(T,x;u_0)|^2 \,\mathrm dx \Big).
%\end{eqnarray*}
Hence, the conclusion
(ii) is true, and  $C_2\big(n,A,B\big)$ can be taken as $C_1\big(n,A,B\big)$.

\vskip 5pt
\textit{Step 2. To prove that (ii)$\Rightarrow$(i)}

\noindent Suppose that (ii) is true for $C_2(n,A,B)$. Arbitrarily fix $f\in L^2(\mathbb R^n;\mathbb C)$. Define a function $u_f$ by
\begin{eqnarray}\label{0309-sch-1.9}
 u_f(x)=e^{-i|x|^2/2} f(x),~x\in\mathbb R^n.
\end{eqnarray}
From (\ref{0309-sch-1.9}) and (\ref{0229-sch-th1-4}) (where $u_0=u_f$ and $T=1/2$), it follows that
\begin{eqnarray*}
\hat f(\xi)=\widehat{e^{i|x|^2/2}u_f(x)}(\xi)  =(i)^{n/2} e^{-i|\xi|^2/2} u(\xi,1/2;u_f),~\xi\in\mathbb R^n.
\end{eqnarray*}
This, along with  (\ref{0309-sch-1.9}) and (\ref{0309-sch-eq1-0}) (where $u_0=u_f$ and $T=1/2$), yields that
\begin{eqnarray*}
  \int_{\mathbb R^n}  |f(x)|^2 \,\mathrm dx   =  \int_{\mathbb R^n}  |u_f(x)|^2 \,\mathrm dx
  &\leq& C_2(n,A,B)
    \left(\int_{ A}  |u_f(x)|^2 \,\mathrm dx
 + \int_{B}  |u(x,1/2;u_f)|^2 \,\mathrm dx  \right)
 \nonumber\\
 &\leq& C_2(n,A,B)   \left(\int_{ A}  |f(x)|^2 \,\mathrm dx
 + \int_{B}  |\hat f(\xi)|^2 \,\mathrm d\xi  \right).
\end{eqnarray*}
Hence, the conclusion  (i) is true and  $C_1(n,A,B)$ can be taken as $C_2(n,A,B)$.

\vskip 5pt
Finally,  from Step 1 - Step 2, we find that when one of  (i) and (ii) is true,
 the constants $C_1(n,A,B)$ and $C_2(n,A,B)$ can be chosen as the same positive number. This ends the proof of this lemma.

\end{proof}

We now use Theorem \ref{uncertainty-principle} and Lemma \ref{0309-sch-eq-twopoints} to prove Theorem \ref{theorem1}.

\begin{proof}[Proof of Theorem \ref{theorem1}]
Let ${x^\prime},\,{x^{\prime\prime}}\in\mathbb R^n$, $r_1,\,r_2>0$ and $T>S\geq 0$. Define
\begin{eqnarray}\label{0427-proof-th1.1-1}
A= B_{r_1}^c({x^\prime})
\;\;\mbox{and}\;\;
B=B_{r_2}^c ({x^{\prime\prime}}).
\end{eqnarray}
 By Theorem  \ref{uncertainty-principle}, we have  (\ref{0309-sch-eq1-0-0}), where
 $$
 (A,B)\;\;\mbox{is replaced by}\;\;\big(A,\frac{B}{2(T-S)}\big)
  \;\;\mbox{and}\;\;C_1(n,A,B)\;\;\mbox{is replaced by}\;\;C\big(n,A^c,\frac{B^c}{2(T-S)}\big),
 $$
 with   $C(n,\cdot,\cdot)$  given by (\ref{uncertainty-principle-ineq-constant}).
Thus we can apply   Lemma \ref{0309-sch-eq-twopoints} to get (\ref{0309-sch-eq1-0}), where
$$
(A,B)\;\;\mbox{is replaced by}\;\;\big(A,\frac{B}{2(T-S)}\big)\;\;\mbox{and}\;\;
C_2(n,A,B)\;\;\mbox{is replaced by}\;\;C\big(n,A^c,\frac{B^c}{2(T-S)}\big).
$$
  The latter, together with  (\ref{uncertainty-principle-ineq-constant}) and (\ref{0427-proof-th1.1-1}), indicates that there exists $C>0$ (depending only on $n$) so that     for each  $u_0\in L^2(\mathbb R^n;\mathbb C)$,
 \begin{eqnarray}\label{0427-proof-th1.1-2}
  \int_{\mathbb R^n}  |u_0(x)|^2 \,\mathrm dx
 %\nonumber\\
 \leq C\big(n,A^c,\frac{B^c}{2(T-S)}\big)
    \left(\int_{B_{r_1}^c({x^\prime})}  |u_0(x)|^2 \,\mathrm dx
 + \int_{B_{r_2}^c({x^{\prime\prime}})}  |u(x,T-S;u_0)|^2 \,\mathrm dx \right),
 \end{eqnarray}
  where
 \begin{eqnarray}\label{0427-proof-th1.1-3}
  C\big(n,A^c,\frac{B^c}{2(T-S)}\big)
 = C e^{C \min\left\{
  \omega_n
  r_1^n
  \omega_n
  r_2^n \frac{1}{2^n(T-S)^n},
  \omega_n
  ^{\frac{1}{n}} r_1  r_2 \frac{1}{T-S},
  \omega_n
  ^{\frac{1}{n}} r_2 r_1 \frac{1}{T-S} \right\} }
  \leq C e^{C  \omega_n
  ^{\frac{1}{n}} r_1 r_2 \frac{1}{T-S} },
 \end{eqnarray}
 with
  $\omega_n$
   the volume of the unit ball in $\mathbb R^n$.

 Finally, by (\ref{0427-proof-th1.1-2})   and (\ref{0427-proof-th1.1-3}), we obtain that
 \begin{eqnarray*}
   \int_{\mathbb R^n}  |u(x,S;u_0)|^2 \,\mathrm dx
 \leq C e^{C c_0^{\frac{1}{n}} r_1r_2   \frac{1}{T-S}  }
  \times  \Big(\int_{B_{r_1}^c({x^\prime})}  |u(x,S;u_0)|^2 \,\mathrm dx
 + \int_{B_{r_2}^c({x^{\prime\prime}})}  |u(x,T;u_0)|^2 \,\mathrm dx\Big).
 \end{eqnarray*}
Because of  the conversation law of the Schr\"{o}dinger equation, the above leads to the inequality in Theorem \ref{theorem1}.
  This ends the proof of this theorem.

\end{proof}

\subsection{Preliminaries on  Theorem \ref{theorem3} and Theorem \ref{theorem4}}

In the proofs of Theorem \ref{theorem3} and Theorem \ref{theorem4}, an interpolation inequality plays a key role. This inequality  will be presented in
Lemma~\ref{lemma-0525-adjust}. To prove Lemma~\ref{lemma-0525-adjust}, we need the following Lemma~\ref{lemma2.1}:

%\begin{lemma}\label{lemma-interval-analytic-interpolation}
%Let $I=[-r,r]$ and $E\subset I$ be a subset of positive measure. Define a quantity as follows
%\begin{eqnarray}\label{0405-q}
% q \triangleq  |I|/|E| +  \sqrt{|I|^2/|E|^2-1}.
%\end{eqnarray}
%Assume that $R>q r$. Then for each $f$ analytic over $2D_{R}$, the following estimate holds
%\begin{eqnarray}\label{0405-A-int-estimate}
% \|f\|_{L^\infty(I)} \leq \frac{4q}{R-r}
% \|f\|_{L^\infty(E)}^{1-\frac{\ln q}{\ln (R/r)}} \|f\|_{L^\infty(D_R)}^{\frac{\ln q}{\ln (R/r)}}.
%\end{eqnarray}
%Here, $D_R$  denotes the closed disk over $C$ centered at $0$ of radius $R$.
%
%\end{lemma}

\begin{lemma}\label{lemma2.1}
There exists an absolute constant $C$ so that for each $a>0$ and $\beta\in \mathbb N^n$,
\begin{eqnarray}\label{0229-euler-1}
 \left(\int_{\mathbb R^n} |\xi^{2\beta}| e^{-a|\xi|} \,\mathrm d\xi \right)^{1/2}
 \leq  \left(\frac{2n}{a}\right)^{n/2} \beta! \left(\frac{Cn}{a}\right)^{|\beta|}
 \;\;.
\end{eqnarray}

\end{lemma}

\begin{proof}
First, we observe that for all $a>0$ and $\beta=(\beta_1,\dots,\beta_n)\in \mathbb N^n$,
\begin{eqnarray}\label{0229-euler-2}
 \int_{\mathbb R^n} |\xi^{2\beta}| e^{-a|\xi|} \,\mathrm d\xi
 &\leq& \int_{\mathbb R^n} |\xi^{2\beta}| e^{-a(\Sigma_{i=1}^n|\xi_i|/n)} \,\mathrm d\xi\nonumber\\
 &=&\Pi_{i=1}^n\int_{\mathbb R_{\xi_i}} |\xi_i|^{2\beta_i} e^{-a|\xi_i|/n} \,\mathrm d\xi_i
 \nonumber\\
 &=& \Pi_{i=1}^n 2 \int_0^\infty r^{2\beta_i} e^{-ar/n} \,\mathrm dr\nonumber\\
 &=& \Pi_{i=1}^n  2\left(\frac{n}{a}\right)^{2\beta_i+1} \int_0^\infty t^{2\beta_i} e^{-t} \,\mathrm dt
 \nonumber\\
 &=& 2^n\left(\frac{n}{a}\right)^{2|\beta|+n} \Pi_{i=1}^n\Gamma(2\beta_i+1)\nonumber\\
 &=&  2^n\left(\frac{n}{a}\right)^{2|\beta|+n} \Pi_{i=1}^n(2\beta_i)!,
\end{eqnarray}
where  $\Gamma(\cdot)$ denotes  the second kind of  Euler's integral or the Gamma function.

We next claim that  there is an absolute constant $C>0$ so that
\begin{eqnarray}\label{wang2.3}
 \sqrt{(2\alpha)!}  \leq  \alpha! C^{\alpha}\;\;\mbox{for all}\;\;
  \alpha\in \mathbb N^+.
\end{eqnarray}
 In fact, using the Stirling's approximation for  factorials
  \begin{eqnarray*}\label{wang2.6}
  \ln(\eta!)=\eta\ln\eta -\eta+O(\ln\eta), ~\forall\,\eta\in \mathbb N^+,
   \end{eqnarray*}
we see that for all $\alpha\in\mathbb N^+$,
\begin{eqnarray*}\label{wang2.4}
  \ln\sqrt{(2\alpha)!}&=&\frac{1}{2} \Big(2\alpha\ln(2\alpha) - 2\alpha+O\big(\ln(2\alpha)\big)\Big)
  \nonumber\\
  &=& \ln\alpha!+ \alpha\ln2+O(\ln\alpha).
\end{eqnarray*}
Thus, there exists an absolute constant $C_1>1$ so that
\begin{eqnarray*}
\sqrt{(2\alpha)!} \leq \exp\left[\ln\alpha!+\alpha\ln C_1 \right]
 = \alpha! C_1^\alpha
\;\;\mbox{for all}\;\;
\alpha\in\mathbb N^+,
\end{eqnarray*}
which leads to (\ref{wang2.3}).

 Finally, (\ref{0229-euler-1}) follows from  (\ref{0229-euler-2})
and (\ref{wang2.3}) at once.
This ends the proof of this lemma.

\end{proof}

We now present an interpolation estimate for $L^2$-functions whose Fourier transforms have
compact supports.

\begin{lemma}\label{lemma-0525-adjust}
Given ${x^\prime},\,{x^{\prime\prime}}\in\mathbb R^n$, $r_1,\,r_2>0$ and $a>0$, there exist two constants $C\triangleq C(n)>0$ and $\theta\triangleq\theta(n)\in(0,1)$ so that for each
  $f\in L^2(\mathbb{R}^n; \mathbb{C})$, with
$\hat f\in C_0^\infty(\mathbb R^n; \mathbb C)$,
\begin{eqnarray}\label{0525-adjust-1}
  \int_{B_{r_2}({x^{\prime\prime}})} |f(x)|^2 \,\mathrm dx
 %\nonumber\\
  \leq C r_2^n(a^{-n}+r_1^{-n})  \left(\int_{B_{r_1}({x^\prime})} |f(x)|^2 \,\mathrm dx \right)^{\theta^{p}}  \left(\int_{\mathbb R^n_\xi} |\hat f(\xi)|^2 e^{a|\xi|} \,\mathrm d\xi\right)^{1-\theta^{p}},
\end{eqnarray}
where
\begin{eqnarray*}\label{0410-a-def-p}
  p\triangleq 1+\frac{|{x^\prime}-{x^{\prime\prime}}|+r_1+r_2}{a\wedge r_1}.
\end{eqnarray*}

\end{lemma}

\begin{proof}
The  proof is divided into two steps.

\vskip 5pt
\textit{Step 1. To show that there is $C\triangleq C(n)>0$ and $\theta\triangleq\theta(n)\in(0,1)$ so that (\ref{0525-adjust-1}), with $a=1$, holds for all
 ${x^\prime},\,{x^{\prime\prime}}\in\mathbb R^n$, $r_1>0$, $r_2>0$ and  $f\in L^2(\mathbb{R}^n; \mathbb{C})$, with
$\hat f\in C_0^\infty(\mathbb R^n; \mathbb C)$}
\vskip 5pt
\noindent
 Arbitrarily fix ${x^\prime},\,{x^{\prime\prime}}\in\mathbb R^n$, $r_1>0$,  $r_2>0$ and  $f\in L^2(\mathbb{R}^n; \mathbb{C})$, with
$\hat f\in C_0^\infty(\mathbb R^n; \mathbb C)$.  We first claim that there is an absolute constant $C>1$ so that
 \begin{eqnarray}\label{0405-lebeau-robbiano-2}
 \|\partial_x^\alpha f\|_{L^\infty(\mathbb R^n_x)}
  \leq (2\pi)^{-\frac{n}{2}} (2n)^{n/2} (Cn)^{|\alpha|}  \alpha!
  \sqrt{\int_{\mathbb R^n_\xi} |\hat f(\xi)|^2 e^{|\xi|} \,\mathrm d\xi}\;\;\mbox{for all}\;\;\alpha\in \mathbb N^n.
\end{eqnarray}
 In fact, since $\hat f\in C_0^\infty(\mathbb R^n; \mathbb C)$, we see that $f$ is analytic and for each multi-index $\alpha\in \mathbb N^n$,
\begin{eqnarray*}\label{0405-wangmin2.8}
 \partial_x^\alpha f(x)  = (2\pi)^{-\frac{n}{2}} \int_{\mathbb R^n_\xi}   e^{ix\cdot\xi} (i\xi)^\alpha \hat f(\xi) \,\mathrm d\xi, ~x\in \mathbb R^n.
\end{eqnarray*}
  From the above equality and   the H\"{o}lder inequality, we see that for each  multi-index $\alpha\in \mathbb N^n$,
\begin{eqnarray*}
 \|\partial_x^\alpha f\|_{L^\infty(\mathbb R^n_x)}  \leq  (2\pi)^{-\frac{n}{2}}
 \sqrt{\int_{\mathbb R^n_\xi} |\xi^{2\alpha}| e^{-|\xi|} \,\mathrm d\xi}
 \sqrt{\int_{\mathbb R^n_\xi} |\hat f(\xi)|^2 e^{|\xi|} \,\mathrm d\xi}.
\end{eqnarray*}
This, along with   Lemma \ref{lemma2.1}, leads to (\ref{0405-lebeau-robbiano-2}).

We next claim that there is $C_1\triangleq C_1(n)>0$ and $\theta_1\triangleq\theta_1(n)\in(0,1)$ (depending only on $n$) so that
\begin{eqnarray}\label{tianjin2.57}
 \int_{B_{r_2}({x^{\prime\prime}})} |f(x)|^2 \,\mathrm dx
 \leq
   \omega_n
 r_2^n (C_1 r_0^{-n/2}+1 )^2
 \left( M^2 \right)^{1-\frac{\theta_1}{2^{K}}}
 \left( \int_{B_{r_1}({x^\prime})} |f(x)|^2 \,\mathrm dx \right)^{\frac{\theta_1}{2^{K}}},
\end{eqnarray}
where
\begin{eqnarray}\label{0405-r0-depend}
 M\triangleq  \left(\frac{n}{\pi}\right)^{n/2} \sqrt{\int_{\mathbb R^n_\xi} |\hat f(\xi)|^2 e^{|\xi|} \,\mathrm d\xi},
 \;\;\;\;
 r_0\triangleq \frac{(Cn)^{-1}\wedge r_1}{5}<1
\end{eqnarray}
(with $C$ given by (\ref{0405-lebeau-robbiano-2})) and
\begin{eqnarray}\label{0404-def-K}
 K  \triangleq  \frac{|{x^\prime}-{x^{\prime\prime}}|+r_1+r_2}{r_0}.
\end{eqnarray}
%Here is the argument:
Let $M$ and $r_0$ be given by (\ref{0405-r0-depend}).
 From (\ref{0405-lebeau-robbiano-2}), we see that
  \begin{eqnarray*}
 |\partial_x^\alpha f(x) |
 \leq  M  \frac{ \alpha!}{(5r_0)^{|\alpha|}},
 ~x\in B_{4r_0}({x^\prime}).
\end{eqnarray*}
Then we can apply
 %\cite[Theorem 1.3]{AE} (where $R=2r_0$)
\cite[Theorem 1.3]{AE} where $R=2r_0$ (see also \cite{Vessella})
  to find that
 \begin{eqnarray*}
\|f\|_{L^\infty(B_{2r_0}({x^\prime}))}
 &\leq&  C_1^\prime   M^{1-\theta_1^\prime}
 \left(
  \omega_n^{1/2}
  |B_{r_0}({x^\prime})|^{-1}\|f\|_{L^1(B_{r_0}({x^\prime}))}
 \right)^{\theta_1^\prime},
\end{eqnarray*}
for some $C_1^\prime\triangleq C_1^\prime(n)>0$ and $\theta_1^\prime\triangleq\theta_1^\prime(n)\in(0,1)$, depending only on $n$.
Since $r_0<r_1$ (see (\ref{0405-r0-depend})), the above inequality, along with  the H\"{o}lder inequality, yields that
\begin{eqnarray}\label{0405-Yuanyuan2.8}
\|f\|_{L^\infty(B_{2r_0}({x^\prime}))}
 &\leq&  C_1^\prime   M^{1-\theta_1^\prime} \left(
   \omega_n^{1/2}
  |B_{r_0}({x^\prime})|^{-1/2}\|f\|_{L^2(B_{r_0}({x^\prime}))}
 \right)^{\theta_1^\prime}
    \nonumber\\
 &\leq&  C_1^\prime   M^{1-\theta_1^\prime} \left(
 r_0^{-n/2}\|f\|_{L^2(B_{r_1}({x^\prime}))}
 \right)^{\theta_1^\prime}.
\end{eqnarray}

Write $D_l(z)$ for the closed disk in the complex plane, centered at $z$ and of radius $l$. It is clear that
\begin{equation}\label{0405-disks}
 D_{r_0}((k+1)r_0) \subset D_{2r_0}(kr_0),~k=1,2,\dots.
\end{equation}
Arbitrarily fix $\vec{v}\in S^{n-1}$. Define a function $g$ over the real line in the following manner:
\begin{eqnarray}\label{0405-f-g}
 g(s)=\frac{1}{M}f({x^\prime}+s \vec{v}),~  s\in \mathbb R.
\end{eqnarray}
From (\ref{0405-f-g}) and (\ref{0405-lebeau-robbiano-2}), one can easily check that $g$ can be extended to be an analytic function over
\begin{equation}\label{0405-Omega-r0}
    \Omega_{r_0} \triangleq \{x+iy\in\mathbb C~:~ x,\,y\in\mathbb R,~|y|<5r_0\}
\end{equation}
and that the extension, still denoted by  $g$, has the property:
\begin{equation}\label{DXTJ2.25}
\|g\|_{L^\infty( \Omega_{r_0} )}\leq 1.
\end{equation}
 By    (\ref{0405-f-g}), (\ref{0405-Omega-r0}) and (\ref{DXTJ2.25}),  we see that
the function $z\mapsto  g(4r_0 z)$ is  analytic over $ D_1(0)$ and verifies that
$ \sup_{z\in D_{1}(0)}$$ |g(4r_0 z)| \leq 1$.
  Then we can apply \cite[Lemma 3.2]{AE} (to the above function) to find that
\begin{eqnarray}\label{tj-17-1}
 \sup_{z\in D_{1/2}(0)} |g(4r_0 z)| \leq C_2^{\prime} \sup_{x\in\mathbb R,~|x|\leq 1/5} |g(4r_0 x)|^{\theta_2^\prime}
\end{eqnarray}
for some $C_2^\prime\triangleq C_2^\prime(n)>0$ and $\theta_2^\prime\triangleq\theta_2^\prime(n)\in(0,1)$, depending only on $n$. Since $r_0<r_1$ (see (\ref{0405-r0-depend})), by  (\ref{tj-17-1}) and (\ref{0405-f-g}), we obtain that
\begin{eqnarray*}
 \|g\|_{L^\infty(D_{2r_0}(0))}
 \leq C_2^{\prime} \left(\frac{1}{M} \|f\|_{L^\infty(B_{2r_0}({x^\prime}))} \right)^{\theta_2^\prime}.
\end{eqnarray*}
This, along with (\ref{0405-Yuanyuan2.8}), yields that
\begin{eqnarray}\label{tj-17-2}
 \|g\|_{L^\infty(D_{2r_0}(0))}
 \leq C_2^\prime {C_1^\prime}^{\theta_2^\prime}  r_0^{-\theta_1^\prime \theta_2^\prime n/2}
 \left( \frac{1}{M} \|f\|_{L^2(B_{r_1}({x^\prime}))} \right)^{\theta_1^\prime \theta_2^\prime}.
\end{eqnarray}

Meanwhile, since $g$ is analytic over $  \Omega_{r_0}$,  we can apply    the Hadamard three-circle theorem (see for instance \cite[Theorem 3.1]{AE}) to get that for each $k=1,2,\dots$,
\begin{eqnarray}\label{TTJNNN2.29}
  \|g\|_{L^\infty(D_{2r_0}(kr_0))}
  \leq \|g\|_{L^\infty(D_{r_0}(kr_0))}^{1/2}  \|g\|_{L^\infty(D_{4r_0}(kr_0))}^{1/2}
  \leq \|g\|_{L^\infty(D_{r_0}(kr_0))}^{1/2}.
\end{eqnarray}
(Here, we used (\ref{DXTJ2.25}).) By (\ref{TTJNNN2.29}) and  (\ref{0405-disks}), we see that for each $k=1,2,\dots$,
\begin{eqnarray*}
  \|g\|_{L^\infty(D_{r_0}((k+1)r_0))} \leq \|g\|_{L^\infty(D_{2r_0}(kr_0))}
  \leq \|g\|_{L^\infty(D_{r_0}(kr_0))}^{1/2},
\end{eqnarray*}
from which,  it follows that for each $k=1,2,\dots$,
\begin{eqnarray*}
  \|g\|_{L^\infty(D_{r_0}((k+1)r_0))}
  \leq \|g\|_{L^\infty(D_{r_0}(kr_0))}^{\frac{1}{2}}
  \leq \cdots
  \leq \|g\|_{L^\infty(D_{r_0}(r_0))}^{(\frac{1}{2})^{k}}.
\end{eqnarray*}
This, along with   (\ref{0404-def-K}) and (\ref{DXTJ2.25}), yields that
\begin{eqnarray}\label{0405-step1-10}
   \|g\|_{L^\infty(\cup_{1\leq k\leq q}D_{r_0}(kr_0))}
 &=& \sup_{1\leq k\leq q} \|g\|_{L^\infty(D_{r_0}(kr_0))}
 \leq \sup_{1\leq k\leq q}  \|g\|_{L^\infty(D_{r_0}(r_0))}^{(\frac{1}{2})^{k-1}}
 \nonumber\\
 &\leq& \sup_{1\leq k\leq q}  \|g\|_{L^\infty(D_{r_0}(r_0))}^{(\frac{1}{2})^{q-1}}
 \leq \|g\|_{L^\infty(D_{r_0}(r_0))}^{(\frac{1}{2})^{K}},
\end{eqnarray}
where  $q$ is  the  integer so that
\begin{eqnarray}\label{0405-step1-r-r0-nr}
 q r_0\geq |{x^\prime}-{x^{\prime\prime}}|+r_1+r_2> (q-1) r_0.
\end{eqnarray}
Because it follows by  (\ref{0405-step1-r-r0-nr})  that
\begin{eqnarray*}
  \big[0,|{x^\prime}-{x^{\prime\prime}}|+r_1+r_2 \big] \subset \cup_{1\leq k\leq q}D_{r_0}(kr_0)
  \;\;\mbox{and}\;\;
  D_{r_0}(r_0) \subset D_{2r_0}(0),
\end{eqnarray*}
we see  from    (\ref{0405-step1-10}) that for all  $s\in \big[ 0,|{x^\prime}-{x^{\prime\prime}}|+r_1+r_2 \big]$,
\begin{eqnarray}\label{tianjin2.29}
 |g(s)|  \leq \|g\|_{L^\infty(\cup_{1\leq k\leq q}D_{r_0}(kr_0))}
 \leq \|g\|_{L^\infty(D_{r_0}(r_0))}^{(\frac{1}{2})^{K}}
 \leq \|g\|_{L^\infty(D_{2r_0}(0))}^{(\frac{1}{2})^{K}}.
\end{eqnarray}
%Meanwhile, by  (\ref{0405-Yuanyuan2.8-1}) and (\ref{0405-f-g}), we obtain that
%\begin{eqnarray*}
% \|g\|_{L^\infty(D_{2r_0}(0))}
% \leq \frac{1}{M} \|f\|_{L^\infty(\mathbb B_{2r_0}(x_{1}))}.
%\end{eqnarray*}

From  (\ref{0405-f-g}), (\ref{tianjin2.29})  and (\ref{tj-17-2}), we find that for all  $s\in \big[ 0,|{x^\prime}-{x^{\prime\prime}}|+r_1+r_2 \big]$,
\begin{eqnarray*}
 & & |f({x^\prime}+s\vec{v})|
 = M|g(s)|
 \leq M \|g\|_{L^\infty(D_{2r_0}(0))}^{\frac{1}{2^{K}}}
 \nonumber\\
 &\leq&   M  \left[ C_2^\prime {C_1^\prime}^{\theta_2^\prime}  r_0^{-\theta_1^\prime \theta_2^\prime n/2}
 \left( \frac{1}{M} \|f\|_{L^2(B_{r_1}({x^\prime}))} \right)^{\theta_1^\prime \theta_2^\prime} \right]^{\frac{1}{2^{K}}}
 \nonumber\\
 &=&  \left( C_2^\prime {C_1^\prime}^{\theta_2^\prime}  r_0^{-\theta_1^\prime \theta_2^\prime n/2} \right)^{2^{-K}}
 M^{1-\frac{\theta_1^\prime \theta_2^\prime}{2^{K}}}
 \|f\|_{L^2(B_{r_1}({x^\prime}))}^{\frac{\theta_1^\prime \theta_2^\prime}{2^{K}}}.
\end{eqnarray*}
Since the above inequality holds for all $\vec v\in S^{n-1}$ and $s\in \big[ 0,|{x^\prime}-{x^{\prime\prime}}|+r_1+r_2 \big]$, we see  that
\begin{eqnarray*}
 \sup_{|x-{x^\prime}| \leq |{x^\prime}-{x^{\prime\prime}}|+r_1+r_2} |f(x)|
 \leq  \left( C_2^\prime {C_1^\prime}^{\theta_2^\prime} r_0^{-\theta_1^\prime \theta_2^\prime n/2} \right)^{2^{-K}}
 M^{1-\frac{\theta_1^\prime \theta_2^\prime}{2^{K}}}
 \|f\|_{L^2(B_{r_1}({x^\prime}))}^{\frac{\theta_1^\prime \theta_2^\prime}{2^{K}}}.
\end{eqnarray*}
Because  $r_0<1$  (see (\ref{0405-r0-depend})), it follows from the above   that
\begin{eqnarray*}
 \sup_{|x-{x^\prime}| \leq |{x^\prime}-{x^{\prime\prime}}|+r_1+r_2} |f(x)|
 %&\leq& ( C_1^\prime r_0^{-\theta_1^\prime n/2}+1 )  M^{1-\frac{\theta_1^\prime}{2^{K}}}
%      \|f\|_{L^2(B_{r_1}(x_{1}))}^{\frac{\theta_1^\prime}{2^{K}}}
%      \nonumber\\
 &\leq&  \left( C_2^\prime {C_1^\prime}^{\theta_2^\prime} r_0^{-n/2}+1 \right)  M^{1-\frac{\theta_1^\prime \theta_2^\prime}{2^{K}}}
      \|f\|_{L^2(B_{r_1}({x^\prime}))}^{\frac{\theta_1^\prime \theta_2^\prime}{2^{K}}}.
\end{eqnarray*}
Since $B_{r_2}({x^{\prime\prime}})\subset B_{|{x^\prime}-{x^{\prime\prime}}|+r_1+r_2}({x^\prime})$, the above yields that
\begin{eqnarray*}
 \int_{B_{r_2}({x^{\prime\prime}})} |f(x)|^2 \,\mathrm dx
 &\leq&
 %|B_1(0)|
   \omega_n
  r_2^n \sup_{|x-{x^\prime}| \leq |{x^\prime}-{x^{\prime\prime}}|+r_1+r_2} |f(x)|^2
 \nonumber\\
 &\leq&
 %|B_1(0)|
   \omega_n
   r_2^n \left( C_2^\prime {C_1^\prime}^{\theta_2^\prime} r_0^{-n/2}+1 \right)^2  M^{2(1-\frac{\theta_1^\prime \theta_2^\prime}{2^{K}})}
      \|f\|_{L^2(B_{r_1}({x^\prime}))}^{\frac{2\theta_1^\prime \theta_2^\prime}{2^{K}}},
 \end{eqnarray*}
from which, (\ref{tianjin2.57}) follows at once.

Finally,  by (\ref{0405-r0-depend}), we see that
\begin{eqnarray*}
 M \geq  \|f\|_{L^2(B_{r_1}({x^\prime}))}
 \;\;\mbox{and}\;\;
 r_0 \geq \frac{(Cn)^{-1}}{5} (1\wedge r_1).
\end{eqnarray*}
These, combined with (\ref{tianjin2.57}) and  (\ref{0404-def-K}), yield that
\begin{eqnarray}\label{0513-bigbug}
& & \int_{B_{r_2}({x^{\prime\prime}})} |f(x)|^2 \,\mathrm dx
\nonumber\\
 &\leq& %|B_1(0)|
   \omega_n
  r_2^n (1+C_1)^2   (5Cn)^n [ (1\wedge r_1)^{-n/2}+1 ]^2  M^{2}
      \left( \frac{\|f\|_{L^2(B_{r_1}({x^\prime}))}^2}{M^2}
     \right)^{ \alpha_1}
\nonumber\\
 &\leq& 4
 %|B_1(0)|
  \omega_n
  (1+C_1)^2 (5Cn)^n  r_2^n (r_1^{-n}+1 )  M^{2}
      \left( \frac{\|f\|_{L^2(B_{r_1}({x^\prime}))}^2}{M^2}
      \right)^{ \alpha_2 },
%\nonumber
\end{eqnarray}
where
\begin{eqnarray*}
 \alpha_1 \triangleq \theta_1 \left(\frac{1}{2}\right)^{\frac{|{x^\prime}-{x^{\prime\prime}}|+r_1+r_2}{r_0}}
 \;\;\mbox{and}\;\;
 \alpha_2 \triangleq \min\left\{\theta_1,\left(\frac{1}{2}\right)^{5Cn} \right\}
      ^{1+\frac{|{x^\prime}-{x^{\prime\prime}}|+r_1+r_2}{1\wedge r_1}}.
\end{eqnarray*}
From (\ref{0513-bigbug}) and (\ref{0405-r0-depend}), we see that $f$ satisfies
  (\ref{0525-adjust-1}), with $a=1$.
 This proves the conclusion in Step 1.

\vskip 5pt
 \textit{Step 2. To show that there is $C\triangleq C(n)>0$ and $\theta\triangleq\theta(n)\in(0,1)$ so that (\ref{0525-adjust-1}), with $a>0$, holds for all
 ${x^\prime},\,{x^{\prime\prime}}\in\mathbb R^n$, $r_1>0$, $r_2>0$ and  $f\in L^2(\mathbb{R}^n; \mathbb{C})$, with
$\hat f\in C_0^\infty(\mathbb R^n; \mathbb C)$ }
\vskip 5pt
 \noindent
Arbitrarily fix ${x^\prime},\,{x^{\prime\prime}}\in\mathbb R^n$, $r_1>0$, $r_2>0$, $a>0$ and $f\in L^2(\mathbb{R}^n; \mathbb{C})$, with
$\hat f\in C_0^\infty(\mathbb R^n; \mathbb C)$. Define a function $g$ by
\begin{eqnarray*}
 g(x)=a^{\frac{n}{2}}f(a x),~x\in\mathbb R^n.
\end{eqnarray*}
It is clear that
\begin{eqnarray*}
g\in L^2(\mathbb{R}^n; \mathbb{C})   \;\;\mbox{and}\;\; \hat g(\xi)=a^{-\frac{n}{2}}\hat f(\xi/a),~\xi\in\mathbb R^n.
\end{eqnarray*}
Since $\hat f\in C_0^\infty(\mathbb R^n;\mathbb C)$, the above implies that $\hat g\in C_0^\infty(\mathbb R^n;\mathbb C)$.
Thus, we can use the conclusion in  Step 1 to see that there is  $C>0$ and $\theta\in (0,1)$, depending only on $n$, so that
\begin{eqnarray}\label{tianjin-2.34}
 & &  \int_{B_{\frac{r_2}{a}}(\frac{{x^{\prime\prime}}}{a})}  |g(x)|^2\,\mathrm dx
 \nonumber\\
 &\leq& C \big(\frac{r_2}{a}\big)^n  \left(1+ \big(\frac{r_1}{a} \big)^{-n}\right)  \left(\int_{B_{\frac{r_1}{a}}(\frac{{x^\prime}}{a})} |g(x)|^2 \,\mathrm dx \right)^{\theta^{p^\prime}}  \left(\int_{\mathbb R^n_\xi} |\hat g(\xi)|^2 e^{|\xi|} \,\mathrm d\xi\right)^{1-\theta^{p^\prime}},
\end{eqnarray}
where
\begin{eqnarray*}
 p^\prime = 1+\frac{|\frac{{x^\prime}}{a}-\frac{{x^{\prime\prime}}}{a}|+\frac{r_1}{a}+\frac{r_2}{a}}{1\wedge \frac{r_1}{a}}
 =1+\frac{|{x^\prime}-{x^{\prime\prime}}|+r_1+r_2}{a\wedge r_1}.
\end{eqnarray*}
 From (\ref{tianjin-2.34}),
we find that
\begin{eqnarray*}
 & & \int_{B_{r_2}({x^{\prime\prime}})}  |f(x)|^2\,\mathrm dx
  =  \int_{B_{\frac{r_2}{a}}(\frac{{x^{\prime\prime}}}{a})}  |g(x)|^2\,\mathrm dx
 %\nonumber\\
% &\leq& C_1 r_2^n(a^{-n}+r_1^{-n})  \left(\int_{B_{\frac{r_1}{a}}(\frac{{x^\prime}}{a})} |g(x)|^2 \,\mathrm dx \right)^{\theta_1^{p}}  \left(\int_{\mathbb R^n_\xi} |\hat g(\xi)|^2 e^{|\xi|} \,\mathrm d\xi\right)^{1-\theta_1^{p}}
 \nonumber\\
 &\leq& C r_2^n(a^{-n}+r_1^{-n}) \left(\int_{B_{r_1}({x^\prime})} |f(x)|^2 \,\mathrm dx \right)^{\theta^{p^\prime}}
 \left(\int_{\mathbb R^n_\xi} |\hat f(\xi)|^2 e^{a|\xi|} \,\mathrm d\xi\right)^{1-\theta^{p^\prime}}.
\end{eqnarray*}
 This proves the conclusion in Step 2 and completes the proof of this lemma.

\end{proof}

%Based on the above lemma, we can build up an interpolation estimate. Before stating it, we define a subspace  in the following manner
%\begin{equation}\label{equ2.6}
%X_a\triangleq\Big\{f\in L^2(\mathbb{R}^n_x)\; :\;  \int_{\mathbb R^n_\xi} | \hat f(\xi)|^2 e^{a|\xi|} \,\mathrm d\xi<\infty\Big\}, \;\;a>0.
%\end{equation}

Two consequences of  Lemma  \ref{lemma-0525-adjust} will be given in order. The first one (Corollary~\ref{lemmawang2.2}) is   another interpolation estimate for $L^2$-functions whose Fourier transforms have
compact supports, while the second one (Corollary~\ref{0331-Cor-Br-spectral-ineq}) is a kind of spectral inequality. (The name of spectral inequality in $\mathbb{R}^n$ arose from \cite{RM}, see \cite[Theorem 3.1]{RM}.)

\begin{corollary}\label{lemmawang2.2}
There exist two constants $C\triangleq C(n)>0$ and $\theta\triangleq\theta(n)\in(0,1)$ so that for each
$r>0$, $a>0$ and each $f\in L^2(\mathbb{R}^n; \mathbb{C})$, with
$\hat f\in C_0^\infty(\mathbb R^n; \mathbb C)$,
\begin{eqnarray}\label{0229-lebeau-robbiano-1}
 \int_{\mathbb R^n_x} |f(x)|^2 \,\mathrm dx
%\nonumber\\
\leq C \left(1+\frac{r^n}{a^n}\right) \left(\int_{B_r^c(0)} |f(x)|^2 \,\mathrm dx \right)^{\theta^{1+\frac{r}{a}}}
 \left(\int_{\mathbb R^n_\xi} |\hat f(\xi)|^2 e^{a|\xi|} \,\mathrm d\xi\right)^{1-\theta^{1+\frac{r}{a}}}.
\end{eqnarray}

\end{corollary}

\begin{proof}
Arbitrarily fix $r>0$, $a>0$ and $f\in L^2(\mathbb{R}^n; \mathbb{C})$, with
 $\hat f\in C_0^\infty(\mathbb R^n; \mathbb C)$.
 First of all, we claim that there exist two constants $C_1\triangleq C_1(n)>0$ and $\theta_1\triangleq\theta_1(n)\in(0,1)$ so that
 \begin{eqnarray}\label{0324-step1-luis}
 \int_{B_r(0)} |f(x)|^2 \,\mathrm dx
  \leq C_1 \left(1+\frac{r^n}{a^n}\right)
  \left( \int_{B_r^c(0)} |f(x)|^2 \,\mathrm dx \right)^{\theta_1^{1+r/a}}
   \left(\int_{\mathbb R^n_\xi} |\hat f(\xi)|^2 e^{a|\xi|} \,\mathrm d\xi\right)^{1-\theta_1^{1+r/a}}.
\end{eqnarray}
Indeed, for arbitrarily fixed  $\vec v\in S^{n-1}$, we have that
$ B_r(2r \vec v) \subset B_r^c(0)$.
Then according to Lemma \ref{lemma-0525-adjust}, where
$({x^\prime},{x^{\prime\prime}},r_1,r_2)=(2r \vec v,0,r,r)$,
there is $C_{11}\triangleq C_{11}(n)>0$ and $\theta_{11}\triangleq \theta_{11}(n)\in(0,1)$ so that
\begin{eqnarray}\label{0525-step1-luis-1}
 & & \int_{B_{r}(0)} |f(x)|^2 \,\mathrm dx
 \\
  &\leq& C_{11} r^n(a^{-n}+r^{-n})  \left(\int_{B_{r}(2r \vec v)} |f(x)|^2 \,\mathrm dx \right)^{\theta_{11}^{1+\frac{4r}{a\wedge r}}}  \left(\int_{\mathbb R^n_\xi} |\hat f(\xi)|^2 e^{a|\xi|} \,\mathrm d\xi\right)^{1-\theta_{11}^{1+\frac{4r}{a\wedge r}}}
  \nonumber\\
  &\leq& C_{11} r^n(a^{-n}+r^{-n})  \left(\int_{B_{r}^c(0)} |f(x)|^2 \,\mathrm dx \right)^{\theta_{11}^{1+\frac{4r}{a\wedge r}}}  \left(\int_{\mathbb R^n_\xi} |\hat f(\xi)|^2 e^{a|\xi|} \,\mathrm d\xi\right)^{1-\theta_{11}^{1+\frac{4r}{a\wedge r}}}.
  \nonumber
\end{eqnarray}
Since
\begin{eqnarray*}
 \frac{1}{a\wedge r} \leq \frac{1}{a} + \frac{1}{r},
 ~\theta_{11}\in(0,1)
 \;\;\mbox{and}\;\;
 \int_{B_{r}^c(0)} |f(x)|^2 \,\mathrm dx \leq \int_{\mathbb R^n_\xi} |\hat f(\xi)|^2 e^{a|\xi|} \,\mathrm d\xi,
\end{eqnarray*}
we find from (\ref{0525-step1-luis-1})  that
\begin{eqnarray*}
 \int_{B_{r}(0)} |f(x)|^2 \,\mathrm dx
%\nonumber\\
  &\leq& C_{11} r^n(a^{-n}+r^{-n})
  \left(
  \frac{ \int_{B_{r}^c(0)} |f(x)|^2 \,\mathrm dx }{ \int_{\mathbb R^n_\xi} |\hat f(\xi)|^2 e^{a|\xi|} \,\mathrm d\xi }
  \right) ^{\theta_{11}^{1+\frac{4r}{a\wedge r}}}
  \int_{\mathbb R^n_\xi} |\hat f(\xi)|^2 e^{a|\xi|} \,\mathrm d\xi
  \nonumber\\
  &\leq&  C_{11} (r^na^{-n}+1)
  \left(
  \frac{ \int_{B_{r}^c(0)} |f(x)|^2 \,\mathrm dx }{ \int_{\mathbb R^n_\xi} |\hat f(\xi)|^2 e^{a|\xi|} \,\mathrm d\xi }
  \right) ^{\theta_{11}^{5(1+\frac{r}{a})}}
  \int_{\mathbb R^n_\xi} |\hat f(\xi)|^2 e^{a|\xi|} \,\mathrm d\xi
  ,
\end{eqnarray*}
which leads to (\ref{0324-step1-luis}).

Next, since
\begin{eqnarray*}
 \int_{B_r^c(0)} |f(x)|^2 \,\mathrm dx \leq \int_{\mathbb R^n_x} |f(x)|^2 \,\mathrm dx
 = \int_{\mathbb R^n_{\xi}} |\hat f(\xi)|^2 \,\mathrm d\xi
 \leq \int_{\mathbb R^n_{\xi}} |\hat f(\xi)|^2 e^{a|\xi|} \,\mathrm d\xi,
\end{eqnarray*}
we have that
\begin{eqnarray*}
 \int_{B_r^c(0)} |f(x)|^2 \,\mathrm dx
 \leq \left(\int_{B_r^c(0)} |f(x)|^2 \,\mathrm dx\right)^{\theta_1^{1+r/a}}
 \left(\int_{\mathbb R^n_{\xi}} |\hat f(\xi)|^2 e^{a|\xi|} \,\mathrm d\xi\right)^{1-\theta_1^{1+r/a}},
\end{eqnarray*}
which, together with (\ref{0324-step1-luis}), leads to (\ref{0229-lebeau-robbiano-1}). this ends the proof of this corollary.

\end{proof}

%\begin{remark}
%Lemma \ref{lemmawang2.2} fails if $\widehat{f}$ only belongs to $L^2(e^{|\xi|^{1-\varepsilon}}d\xi)$ for some $\varepsilon>0$. In fact, Ingham  constructed the following example in dimension 1 in \cite{Ing}:
%$$
%\hat f(\xi) = \prod_{i=1}^\infty \frac{\sin a_n \xi}{a_n \xi}, \quad \xi\neq 0;~
%\hat f(0)=1.
%$$
%After choosing suitable positive sequence $a_n$, $\widehat{f} \in L^2(e^{|\xi|^{1-\varepsilon}}d\xi)$ but the support of $f$ is compact. Obviously, $f$ is not zero, thus Lemma \ref{lemmawang2.2} does not hold in this case.
%\end{remark}

\begin{corollary}\label{0331-Cor-Br-spectral-ineq}
There exists a positive constant $C\triangleq C(n)$ so that for each $r>0$ and $N\geq 0$,
\begin{eqnarray}\label{0307-BMc-spectral-ineq}
 \int_{\mathbb R^n} |f(x)|^2  \,\mathrm dx  \leq e^{C(1+rN)} \int_{B_r^c(0)} |f(x)|^2 \,\mathrm dx
\end{eqnarray}
 for all $f\in L^2(\mathbb R^n;\mathbb C)$ with supp\,$\hat f\subset B_N(0)$.

\end{corollary}

\begin{proof}
The proof is divided into the following two steps:

\vskip 5pt
\textit{Step 1. To show that there is $C\triangleq C(n)>0$ so that (\ref{0307-BMc-spectral-ineq}), with $r=1$, holds for all  $N\geq 0$ and $f\in L^2(\mathbb R^n;\mathbb C)$, with supp\,$\hat f\subset B_N(0)$}

\vskip 5pt

Arbitrarily fix $N\geq 0$ and then fix  $f\in L^2(\mathbb R^n;\mathbb C)$, with supp\,$\hat f\subset B_N(0)$.  By a standard density argument, we can apply Corollary \ref{lemmawang2.2} to verify that
there is $C_1\triangleq C_1(n)>0$ and $\theta_1 \triangleq \theta_1(n)\in(0,1)$ (only depending on $n$) so that
\begin{eqnarray}\label{JJTT2.39}
\int_{\mathbb R^n_x} |f(x)|^2 \,\mathrm dx
\leq C_1 \left(\int_{B_1^c(0)} |f(x)|^2 \,\mathrm dx \right)^{\theta_1}
 \left(\int_{\mathbb R^n_\xi} |\hat f(\xi)|^2 e^{|\xi|} \,\mathrm d\xi\right)^{1-\theta_1}.
\end{eqnarray}
Indeed, since $\hat f(\xi) e^{|\xi|/2}\in L^2(\mathbb{R}^n;\mathbb{C})$, we can choose $\{g_k\} \subset C_0^\infty (\mathbb R^n;\mathbb C)$,
 with supp\,$g_k\subset B_{k}(0)$, so that
\begin{eqnarray}\label{tianjin-2.39}
 \lim_{k\rightarrow\infty} \int_{\mathbb R^n_\xi} |g_k(\xi) - \hat f(\xi) e^{|\xi|/2} |^2  \,\mathrm d\xi =0.
\end{eqnarray}
 Meanwhile, since  supp\,$g_k\subset B_k(0)$ for all $k\in\mathbb N^+$, we can find  $\{h_k\} \subset C_0^\infty (\mathbb R^n;\mathbb C)$, with supp\,$h_k\subset B_{k+1}(0)$, so that
\begin{eqnarray*}
 \int_{\mathbb R^n_\xi} |h_k(\xi) -g_k(\xi) e^{-|\xi|/2} |^2  \,\mathrm d\xi  \leq \frac{1}{k} e^{-k-1}
 \;\;\mbox{for each}\;\; k\in\mathbb N^+.
\end{eqnarray*}
This implies that for each $k\in\mathbb N^+$,
\begin{eqnarray*}
 \int_{\mathbb R^n_\xi} |h_k(\xi)e^{|\xi|/2} - g_k(\xi)  |^2  \,\mathrm d\xi
 = \int_{B_{k+1}(0)} |h_k(\xi) -g_k(\xi) e^{-|\xi|/2} |^2 e^{|\xi|}  \,\mathrm d\xi \leq 1/k,
\end{eqnarray*}
which, together with (\ref{tianjin-2.39}), yields that
\begin{eqnarray}\label{JJJTT2.41}
 \lim_{k\rightarrow\infty} \int_{\mathbb R^n_\xi} |h_k(\xi) - \hat f(\xi)  |^2 e^{|\xi|} \,\mathrm d\xi =0.
\end{eqnarray}
Let $\{f_k\}\subset L^2(\mathbb R^n;\mathbb C)$ so that
\begin{eqnarray*}
  \hat f_k(\xi)  = h_k(\xi) ,~\xi\in\mathbb R^n
  \;\;\mbox{for each}\;\;
  k\in\mathbb N^+.
\end{eqnarray*}
Then by (\ref{JJJTT2.41}), we find that
\begin{eqnarray*}
 \{\hat f_k\} \subset C_0^\infty(\mathbb R^n;\mathbb C),
 ~\lim_{k\rightarrow\infty} \int_{\mathbb R^n_\xi} |\hat f_k(\xi) - \hat f(\xi)  |^2 e^{|\xi|} \,\mathrm d\xi =0
 \;\;\mbox{and}\;\;
 \lim_{k\rightarrow\infty} \|f_k-f\|_{L^2(\mathbb R^n;\mathbb C)}=0.
\end{eqnarray*}
From these, we can apply  Corollary \ref{lemmawang2.2} (where $a=1$ and $r=1$) to get (\ref{JJTT2.39}).

 Since supp\,$\hat f\subset B_N(0)$, it follows from (\ref{JJTT2.39}) that
\begin{eqnarray*}
  \int_{\mathbb R^n_x} |f(x)|^2 \,\mathrm dx
\leq C_1 \left(\int_{B_1^c(0)} |f(x)|^2 \,\mathrm dx \right)^{\theta_1}
 e^{(1-\theta_1)N} \left(\int_{\mathbb R^n_\xi} |\hat f(\xi)|^2  \,\mathrm d\xi\right)^{1-\theta_1}.
\end{eqnarray*}
Since the Fourier transform is an isometry, we obtain from the above inequality that
\begin{eqnarray*}
  \int_{\mathbb R^n_x} |f(x)|^2 \,\mathrm dx
  \leq  C_1^{1/\theta_1} e^{(1-\theta_1)N/\theta_1} \int_{B_1^c(0)} |f(x)|^2 \,\mathrm dx
  %\nonumber\\
  = e^{[\ln C_1+(1-\theta_1)N]/\theta_1} \int_{B_1^c(0)} |f(x)|^2 \,\mathrm dx.
\end{eqnarray*}
Hence,  (\ref{0307-BMc-spectral-ineq}), with  $r=1$, is true.

\vskip 5pt
\textit{Step 2. To show that there is $C\triangleq C(n)>0$ so that (\ref{0307-BMc-spectral-ineq}), with $r>0$, holds for all  $N\geq 0$ and $f\in L^2(\mathbb R^n;\mathbb C)$, with supp\,$\hat f\subset B_N(0)$}

\vskip 5pt

 For this purpose, arbitrarily fix $N\geq 0$ and $r>0$. Then fix  $f\in L^2(\mathbb R^n;\mathbb C)$ with supp\,$\hat f\subset B_N(0)$. Define a function $g$ by
\begin{eqnarray}\label{0331-cor2.7-1}
 g(x)= r^{n/2} f(rx),~x\in\mathbb R^n.
\end{eqnarray}
One can easily check that
\begin{eqnarray}\label{0331-cor2.7-1-1}
 \hat g(\xi) = (2\pi)^{-n/2} \int_{\mathbb R^n_x} r^{n/2} f(rx) e^{-i x\cdot\xi} \,\mathrm dx
 %\nonumber\\
 = r^{-n/2} \hat f(\xi/r)\;\;\mbox{for a.e.}\;\;\xi\in\mathbb R^n.
\end{eqnarray}
Since  supp\,$\hat f\subset B_N(0)$, we see from (\ref{0331-cor2.7-1-1}) that  supp\,$\hat g\subset B_{rN}(0)$. Thus,
 according to the conclusion in Step 1, there is $C\triangleq C(n)$ so that (\ref{0307-BMc-spectral-ineq}),
 with $(f,r,N)$ replaced by $(g,1,rN)$, is true. That is,
 \begin{eqnarray*}
 \int_{\mathbb R^n} |g(x)|^2 \,\mathrm dx \leq e^{C(1+rN)} \int_{B_1^c(0)} |g(x)|^2 \,\mathrm dx.
\end{eqnarray*}
This, along with (\ref{0331-cor2.7-1}) and (\ref{0331-cor2.7-1-1}), yields that
\begin{eqnarray*}
 & & \int_{\mathbb R^n} |f(x)|^2 \,\mathrm dx
  =  \int_{\mathbb R^n} |g(x)|^2 \,\mathrm dx
  \nonumber\\
  &\leq& e^{C(1+rN)} \int_{B_1^c(0)} |g(x)|^2 \,\mathrm dx
  = e^{C(1+rN)} \int_{B_r^c(0)} |f(x)|^2 \,\mathrm dx.
\end{eqnarray*}
Hence, (\ref{0307-BMc-spectral-ineq}), with  $r>0$ is true. We end the proof of this corollary.

\end{proof}

\subsection{Proofs of Theorem \ref{theorem3} and Theorem \ref{theorem4}}

We first prove  Theorem \ref{theorem3}.

\begin{proof}[Proof of Theorem \ref{theorem3}]
Throughout this proof, we  arbitrarily fix
$$
r>0,\; a>0,\; T>0\;\;\mbox{and}\;\;u_0\in C_0^\infty(\mathbb R^n;\mathbb C).
$$
   Define a function $ f$ as follows:
 \begin{eqnarray}\label{0229-two-points-3}
  f(x)\triangleq   e^{-i|x|^2/4T} u(x,T;u_0),~x\in \mathbb R^n.
 \end{eqnarray}
 From (\ref{0229-two-points-3}) and   (\ref{0229-sch-th1-4}), we find  that
  \begin{eqnarray*}
  (2iT)^{n/2} f(x) =  \widehat{e^{i|\xi|^2/4T}u_0(\xi)} (x/2T),~x\in\mathbb R^n.
 \end{eqnarray*}
 This yields  that for a.e. $\xi\in\mathbb R^n$,
 \begin{eqnarray}\label{0229-two-points-4}
  \hat f(\xi) &=& \frac{1}{(2\pi)^{n/2}}  \int_{\mathbb R^n_x} f(x) e^{-i x\cdot\xi} \,\mathrm dx
  = \frac{(2iT)^{-n/2}}{(2\pi)^{n/2}}  \int_{\mathbb R^n_x} (2iT)^{n/2} f(x) e^{-i x\cdot\xi} \,\mathrm dx
  \nonumber\\
  &=& \frac{(2iT)^{-n/2}}{(2\pi)^{n/2}} (2T)^{n} \int_{\mathbb R^n_x} (2iT)^{n/2} f(2Tx) e^{-i x\cdot(2T\xi)} \,\mathrm dx
  \nonumber\\
  &=& \frac{(-2iT)^{n/2}}{(2\pi)^{n/2}}  \int_{\mathbb R^n_x} \widehat{e^{i|\eta|^2/4T}u_0(\eta)} (x) e^{i x\cdot(-2T\xi)} \,\mathrm dx
  \nonumber\\
  &=& (-2iT)^{n/2}  e^{i|\eta|^2/4T}u_0(\eta)|_{\eta=-2T\xi}  = (-2iT)^{n/2} e^{iT|\xi|^2} u_0(-2T\xi).
 \end{eqnarray}

  We are going to prove  the conclusions (i)-(iii) in the theorem one by one.

 \vskip 5pt
We first show  the conclusion (i) of Theorem \ref{theorem3}. By (\ref{0229-two-points-3}), we have that
 \begin{eqnarray*}
 \int_{\mathbb R^n_x} |u(x,T;u_0)|^2 \,\mathrm dx
  = \int_{\mathbb R^n_x} |f(x)|^2 \,\mathrm dx.
  \end{eqnarray*}
 Then by Corollary \ref{lemmawang2.2}, where $a$ is replaced by $2Ta$, we find  that
 \begin{eqnarray*}
 \int_{\mathbb R^n_x} |u(x,T;u_0)|^2 \,\mathrm dx
   &\leq& C \left(1+\frac{r^n}{(2Ta)^n}\right) \left( \frac{ \int_{B_r^c(0)} |f(x)|^2 \,\mathrm dx }{\int_{\mathbb R^n_\xi} |\hat f(\xi)|^2 e^{2Ta|\xi|} \,\mathrm d\xi} \right)^{\theta^{1+\frac{r}{2Ta}}}
  \int_{\mathbb R^n_\xi} |\hat f(\xi)|^2 e^{2Ta|\xi|} \,\mathrm d\xi
 \nonumber\\
 &\leq& C \left(1+\frac{r^n}{(Ta)^n}\right) \left( \frac{ \int_{B_r^c(0)} |f(x)|^2 \,\mathrm dx }{\int_{\mathbb R^n_\xi} |\hat f(\xi)|^2 e^{2Ta|\xi|} \,\mathrm d\xi} \right)^{\theta^{1+\frac{r}{Ta}}}
  \int_{\mathbb R^n_\xi} |\hat f(\xi)|^2 e^{2Ta|\xi|} \,\mathrm d\xi,
 \end{eqnarray*}
 for some $C\triangleq C(n)>0$ and $\theta\triangleq \theta(n)\in(0,1)$ (depending only on $n$).
 From this, (\ref{0229-two-points-3}) and (\ref{0229-two-points-4}), after some computations, we obtain that
 \begin{eqnarray*}
 & & \int_{\mathbb R^n_x} |u(x,T;u_0)|^2 \,\mathrm dx
 \nonumber\\
 &\leq& C \left(1+\frac{r^n}{(aT)^n}\right) \left(\int_{B_r^c(0)} |u(x,T;u_0)|^2 \,\mathrm dx \right)^{\theta^{1+\frac{r}{aT}}}
 \left(\int_{\mathbb R^n_\xi} |u_0(\xi)|^2 e^{a|\xi|} \,\mathrm d\xi\right)^{1-\theta^{1+\frac{r}{aT}}}.
 \end{eqnarray*}
 The above inequality, together with  the conversation law of the Schr\"{o}dinger equation,  leads to (\ref{0229-two-points-1}). Hence,
 the conclusion (i) of the theorem is true.

 \vskip 5pt
 We next show the conclusion (ii) of Theorem \ref{theorem3}. Arbitrarily fix $\beta>1$ and $\gamma\in (0,1)$. We divide the  proof  into the following two steps:

 \vskip 5pt
  \textit{Step 1. To show that there exists  $C\triangleq C(n)$ so that
 \begin{eqnarray}\label{0331-th1.2-ii-step1}
  \int_{\mathbb R^n_x} |f(x)|^2 \,\mathrm dx
  \leq C e^{ \left(\frac{C^{\beta}r^{\beta}}{aT^\beta(1-\gamma)}\right)^{\frac{1}{\beta-1}} }
 \left(\int_{B_r^c(0)}|f(x)|^2 \,\mathrm dx\right)^{\gamma}
 \left(\int_{\mathbb R^n_\xi}e^{ a|2T \xi|^\beta} |\hat f(\xi)|^2 \,\mathrm d\xi\right)^{1-\gamma}
 \end{eqnarray}}
Indeed,  for an arbitrarily fixed $N\geq0$, we  define two functions $g_1$ and $g_2$ in $L^2(\mathbb R^n;\mathbb C)$ so that
 \begin{eqnarray*}
  \hat g_1 \triangleq \chi_{B_N(0)} \hat f
  \;\;\mbox{and}\;\;
  \hat g_2 \triangleq \chi_{B_N^c(0)} \hat f.
 \end{eqnarray*}
 It is clear that $f=g_1+g_2$ in $L^2(\mathbb R^n;\mathbb C)$. Then by applying Corollary \ref{0331-Cor-Br-spectral-ineq} to $g_1$, we obtain that
 \begin{eqnarray}\label{tianjin2.34}
  \int_{\mathbb R^n_x} |f(x)|^2 \,\mathrm dx
  &\leq& 2 \int_{\mathbb R^n_x} |g_1(x)|^2 \,\mathrm dx + 2 \int_{\mathbb R^n_x} |g_2(x)|^2 \,\mathrm dx
  \nonumber\\
  &\leq& 2 e^{C(1+rN)}\int_{B_r^c(0)} |g_1(x)|^2 \,\mathrm dx + 2 \int_{\mathbb R^n_x} |g_2(x)|^2 \,\mathrm dx
  \nonumber\\
  &\leq& 4 e^{C(1+rN)}\int_{B_r^c(0)} \big(|f(x)|^2+|g_2(x)|^2\big) \,\mathrm dx + 2 \int_{\mathbb R^n_x} |g_2(x)|^2 \,\mathrm dx
  \nonumber\\
  &\leq& 4 e^{C(1+rN)}\int_{B_r^c(0)} |f(x)|^2 \,\mathrm dx
   + 6 e^{C(1+rN)} \int_{\mathbb R^n_x} |g_2(x)|^2 \,\mathrm dx,
 \end{eqnarray}
 for some $C>0$, depending only on $n$.
 Meanwhile, since  the Fourier transform is an isometry, we have that
 \begin{eqnarray*}
  \int_{\mathbb R^n_x} |g_2(x)|^2 \,\mathrm dx
  &=& \int_{\mathbb R^n_\xi} |\hat{g_2}(\xi)|^2 \,\mathrm d\xi
  = \int_{\mathbb R^n_\xi} |\chi_{B_N^c(0)} (\xi) \hat f(\xi)|^2 \,\mathrm d\xi
  \nonumber\\
  &=& e^{-a (2TN)^\beta } \int_{\mathbb R^n_\xi} |\chi_{B_N^c(0)} (\xi) \hat f(\xi)|^2 e^{a(2TN)^\beta }\,\mathrm d\xi.
 \end{eqnarray*}
 This, along with (\ref{tianjin2.34}), yields  that
 \begin{eqnarray}\label{0331-th1.2-ii-step1-1}
  \int_{\mathbb R^n_x} |f(x)|^2 \,\mathrm dx
  &\leq& 4 e^{C(1+rN)}\int_{B_r^c(0)} |f(x)|^2 \,\mathrm dx
   + 6 e^{C(1+rN)-a(2TN)^\beta } \int_{\mathbb R^n_\xi} |\hat f(\xi)|^2 e^{a|2T\xi|^\beta}\,\mathrm d\xi.
 \end{eqnarray}
 Since it follows  from the Young inequality that
 \begin{eqnarray*}
  CrN&=& \Big[ Cr \big((1-\gamma) a(2T)^\beta \big)^{-\frac{1}{\beta}} \Big]   \Big[ \big((1-\gamma) a (2T)^\beta \big)^{\frac{1}{\beta}} N \Big]
  \nonumber\\
  &\leq&  (1-\frac{1}{\beta}) \Big[ Cr \big( (1-\gamma) a(2T)^\beta \big)^{-\frac{1}{\beta}} \Big]^{\frac{\beta}{\beta-1}}  +
       \frac{1}{\beta} \Big[ \big((1-\gamma) a(2T)^\beta \big)^{\frac{1}{\beta}} N \Big]^\beta
       \nonumber\\
  &\leq&    \Big[ (Cr)^{\beta}/  \big( a(2T)^\beta (1-\gamma) \big) \Big]^{\frac{1}{\beta-1}}
       + (1-\gamma) a (2TN)^\beta,
 \end{eqnarray*}
 we get from (\ref{0331-th1.2-ii-step1-1}) that
 \begin{eqnarray*}
 & & \int_{\mathbb R^n_x} |f(x)|^2 \,\mathrm dx
 \nonumber\\
  &\leq& 6 e^{C+\left(\frac{C^{\beta}r^{\beta}}{a(2T)^\beta(1-\gamma)}\right)^{\frac{1}{\beta-1}}}
  \Big( e^{(1-\gamma) a (2TN)^\beta }  \int_{B_r^c(0)} |f(x)|^2 \,\mathrm dx
   + e^{-\gamma a (2TN)^\beta }  \int_{\mathbb R^n_\xi} |\hat f(\xi)|^2 e^{a |2T\xi|^\beta}\,\mathrm d\xi
   \Big).
 \end{eqnarray*}
 Since $N$ was arbitrarily taken from $[0,\infty)$, the above indicates that for all $\varepsilon\in(0,1)$,
 \begin{eqnarray*}
  \int_{\mathbb R^n_x} |f(x)|^2 \,\mathrm dx
  &\leq& 6 e^{C+\left(\frac{C^{\beta}r^{\beta}}{a(2T)^\beta (1-\gamma)}\right)^{\frac{1}{\beta-1}}}
  \Big( \varepsilon^{-(1-\gamma)}  \int_{B_r^c(0)} |f(x)|^2 \,\mathrm dx
   + \varepsilon^{\gamma}  \int_{\mathbb R^n_\xi} |\hat f(\xi)|^2
   e^{a |2T\xi|^\beta}\,\mathrm d\xi
   \Big).
 \end{eqnarray*}
 One can directly check that the above inequality holds for all $\varepsilon>0$.
% Because $\int_{\mathbb R^n_x} |f(x)|^2 \,\mathrm dx \leq \int_{\mathbb R^n_\xi} |\hat f(\xi)|^2 e^{a|\xi|^\beta}\,\mathrm d\xi$,  we deduce that for all $\varepsilon>0$,
% \begin{eqnarray*}
%  \int_{\mathbb R^n_x} |f(x)|^2 \,\mathrm dx
%  &\leq& 6 e^{C+\left(\frac{C^{\beta}r^{\beta}}{a(1-\gamma)}\right)^{\frac{1}{\beta-1}}}
%  \Big( \varepsilon^{-(1-\gamma)}  \int_{B_r^c(0)} |f(x)|^2 \,\mathrm dx
%   + \varepsilon^{\gamma}  \int_{\mathbb R^n_\xi} |\hat f(\xi)|^2 e^{a|\xi|^\beta}\,\mathrm d\xi
%   \Big).
% \end{eqnarray*}
Minimizing it  w.r.t. $\varepsilon>0$  leads to (\ref{0331-th1.2-ii-step1}). Here, we used the inequality:
$$
\inf_{\varepsilon>0} \left( \varepsilon^{-(1-\gamma)} A  + \varepsilon^\gamma B \right) \leq 2 A^{\gamma} B^{1-\gamma}
\;\;\mbox{for all}\;\;
A,\,B\geq 0.
$$
 This ends the proof of Step 1.

 \vskip 5pt
 \textit{Step 2. To prove (\ref{0229-two-points-beta})}

 \noindent
  From (\ref{0229-two-points-3}), (\ref{0331-th1.2-ii-step1}) and (\ref{0229-two-points-4}), after some computations, we see that
 \begin{eqnarray*}
  & & \int_{\mathbb R^n_x} |u(x,T;u_0)|^2 \,\mathrm dx
  = \int_{\mathbb R^n_x} |f(x)|^2 \,\mathrm dx
  \nonumber\\
  &\leq& C e^{ \left(\frac{C^{\beta}r^{\beta}}{a T^\beta(1-\gamma)}\right)^{\frac{1}{\beta-1}} }
 \left(\int_{B_r^c(0)}|f(x)|^2 \,\mathrm dx\right)^{\gamma}
 \left(\int_{\mathbb R^n_\xi}e^{ a |2T\xi|^\beta} |\hat f(\xi)|^2 \,\mathrm d\xi\right)^{1-\gamma}
 \nonumber\\
 &\leq& C e^{ \left(\frac{C^{\beta}r^{\beta}}{aT^\beta(1-\gamma)}\right)^{\frac{1}{\beta-1}} }
  \left(\int_{B_r^c(0)} |u(x,T;u_0)|^2 \,\mathrm dx \right)^{\gamma}
 \left(\int_{\mathbb R^n_\xi} |u_0(\xi)|^2 e^{a|\xi|^\beta} \,\mathrm d\xi\right)^{1-\gamma},
 \end{eqnarray*}
 %This yields  that
% \begin{eqnarray*}
%  & & \int_{\mathbb R^n_x} |u(x,T;u_0)|^2 \,\mathrm dx
%  \nonumber\\
% &\leq&C e^{ \left(\frac{C^{\beta}r^{\beta}}{ a(1-\gamma) T^\beta}\right)^{\frac{1}{\beta-1}} }
%  \left(\int_{B_r^c(0)} |u(x,T;u_0)|^2 \,\mathrm dx \right)^{\gamma}
% \left(\int_{\mathbb R^n_\xi} |u_0(\xi)|^2 e^{a|\xi|^\beta} \,\mathrm d\xi\right)^{1-\gamma},
% \end{eqnarray*}
which, along with  the conversation law of the Schr\"{o}dinger equation, leads to  (\ref{0229-two-points-beta}). This ends the proof of the conclusion (ii).

 \vskip 5pt
 (iii)
   By contradiction, suppose that the conclusion (iii) was not true. Then there would exist $\hat r>0$, $\hat a>0$,  $\widehat T>0$,  $\hat \gamma\in(0,1)$, $\hat C>0$
    and an increasing function $\hat\alpha(s)$ defined over $[0,\infty)$, with $\lim_{s\rightarrow\infty} s^{-1} \hat\alpha(s)=0$, so that for each $v_0\in C_0^\infty(\mathbb R^n;\mathbb C)$,  the solution of (\ref{0229-sch-1}) satisfies that
\begin{eqnarray}\label{tianjin2.37}
 \int_{\mathbb R^n}|v_0(x)|^2 \,\mathrm dx
 \leq \hat C\left(\int_{B_{\hat r}^c(0)}|u(x,\widehat T;v_0)|^2 \,\mathrm dx\right)^{\hat\gamma}
 \left(\int_{\mathbb R^n}e^{\hat a \hat\alpha(|x|)} |v_0(x)|^2 \,\mathrm dx\right)^{1-\hat\gamma}.
 \end{eqnarray}
 %where
% \begin{eqnarray}\label{0413-th1.2-iii-1}
%  \lim_{s\rightarrow\infty} \frac{\alpha(s)}{s}=0.
%  \end{eqnarray}
 Arbitrarily fix $g\in L^2(\mathbb R^n;\mathbb C)$ with $\hat g \in C_0^\infty(\mathbb R^n;\mathbb C)$. Define $v_{0,g}\in C_0^\infty(\mathbb R^n;\mathbb C)$ in the following manner:
 \begin{eqnarray}\label{tianjin2.38}
  \hat g(\xi) = (-2i\widehat T)^{n/2} e^{i\widehat T|\xi|^2} v_{0,g}(-2\widehat T\xi),~\xi\in\mathbb R^n.
 \end{eqnarray}
  One can easily check that
  \begin{eqnarray}\label{tianjin2.39}
   g(x) =  e^{-i|x|^2/4\widehat T} u(x,\widehat T;v_{0,g}),~x\in\mathbb R^n.
 \end{eqnarray}
 Indeed, let $f_g$ verify that
 \begin{eqnarray}\label{tianjin2.40}
   f_g(x) =  e^{-i|x|^2/4\widehat T} u(x,\widehat T;v_{0,g}),~x\in\mathbb R^n.
 \end{eqnarray}
 Then by (\ref{0229-two-points-3}), (\ref{0229-two-points-4}) (where $(T,u_0)=(\widehat T,v_{0,g})$) and (\ref{tianjin2.38}), we find that
 \begin{eqnarray*}
  \hat f_g(\xi)=  (-2i\widehat T)^{n/2} e^{i\widehat T|\xi|^2} v_{0,g}(-2\widehat T\xi)
  =  \hat g(\xi),~\xi\in\mathbb R^n,
 \end{eqnarray*}
 which implies that $f_g=g$. This, along with (\ref{tianjin2.40}), leads to (\ref{tianjin2.39}).

  By (\ref{tianjin2.39}), the conversation law (for the Schr\"{o}dinger equation), (\ref{tianjin2.37}) and   (\ref{tianjin2.38}), we get that
 \begin{eqnarray*}
  \int_{\mathbb R^n_x} |g(x)|^2 \,\mathrm dx
  %&=& \int_{\mathbb R^n}|u(x,\widehat T;u_0)|^2 \,\mathrm dx   = \int_{\mathbb R^n}|u_0(x)|^2 \,\mathrm dx
%  \nonumber\\
%  &\leq& \hat C\left(\int_{B_r^c(0)}|u(x,\widehat T;u_0)|^2 \,\mathrm dx\right)^{1-\hat\gamma}
% \left(\int_{\mathbb R^n_\xi}e^{\alpha(|\xi|)} |u_0(\xi)|^2 \,\mathrm d\xi\right)^{\hat\gamma}
% \nonumber\\
 &=& \int_{\mathbb R^n_x} |u(x,\widehat T;v_{0,g})|^2 \,\mathrm dx = \int_{\mathbb R^n_x} |v_{0,g}(x)|^2 \,\mathrm dx
 \nonumber\\
 &\leq& \hat C\left(\int_{B_{\hat r}^c(0)}|u(x,\widehat T;v_{0,g})|^2 \,\mathrm dx\right)^{\hat\gamma}
 \left(\int_{\mathbb R^n}e^{\hat a \hat\alpha(|x|)} |v_{0,g}(x)|^2 \,\mathrm dx\right)^{1-\hat\gamma}
 \nonumber\\
 &=& \hat C\left(\int_{B_{\hat r}^c(0)}|g(x)|^2 \,\mathrm dx\right)^{\hat\gamma}
 \left(\int_{\mathbb R^n_\xi}e^{\hat a \hat\alpha(2\widehat T|\xi|)} |\hat g(\xi)|^2 \,\mathrm d\xi\right)^{1-\hat\gamma}.
 \end{eqnarray*}
  By this, using a standard density argument, we can show  that for each $g\in L^2(\mathbb R^n;\mathbb C)$  with supp\,$\hat g$ compact,
 \begin{eqnarray*}
  \int_{\mathbb R^n_x} |g(x)|^2 \,\mathrm dx
  %&=& \int_{\mathbb R^n}|u(x,\widehat T;u_0)|^2 \,\mathrm dx   = \int_{\mathbb R^n}|u_0(x)|^2 \,\mathrm dx
%  \nonumber\\
%  &\leq& \hat C\left(\int_{B_r^c(0)}|u(x,\widehat T;u_0)|^2 \,\mathrm dx\right)^{1-\hat\gamma}
% \left(\int_{\mathbb R^n_\xi}e^{\alpha(|\xi|)} |u_0(\xi)|^2 \,\mathrm d\xi\right)^{\hat\gamma}
% \nonumber\\
 &\leq& \hat C\left(\int_{B_{\hat r}^c(0)}|g(x)|^2 \,\mathrm dx\right)^{\hat\gamma}
 \left(\int_{\mathbb R^n_\xi}e^{\hat a \hat\alpha(2\widehat T|\xi|)} |\hat g(\xi)|^2 \,\mathrm d\xi\right)^{1-\hat\gamma}.
 \end{eqnarray*}
 Since $\hat\alpha(\cdot)$ is increasing and because  the Fourier transform is an isometry, the above yields that   that for each $N\geq 1$ and each $g\in L^2(\mathbb R^n;\mathbb C)$  with supp\,$\hat g\subset B_N(0)$,
 \begin{eqnarray}\label{0331-th1.2-iii-1}
  \int_{\mathbb R^n_x} |g(x)|^2 \,\mathrm dx
  &\leq& \hat C \left(\int_{B_{\hat r}^c(0)}|g(x)|^2 \,\mathrm dx\right)^{\hat\gamma}
 \left(\int_{\mathbb R^n_\xi}e^{\hat a \hat\alpha(2\widehat T N)} |\widehat {g}(\xi)|^2 \,\mathrm d\xi\right)^{1-\hat\gamma}
 \nonumber\\
 &=& \hat C e^{(1-\hat\gamma)\hat a \hat\alpha(2\widehat T N)} \left(\int_{B_{\hat r}^c(0)}|g(x)|^2 \,\mathrm dx\right)^{\hat\gamma}
 \left(\int_{\mathbb R^n_x} |g(x)|^2 \,\mathrm dx\right)^{1-\hat\gamma}.
 \end{eqnarray}

    Two observations are given in order: First, according to \cite[Proposition 3.4]{RM}, there is $C_0>0$ and $N_0>0$ so that for each $N\geq N_0$, there is  $f_N\in L^2(\mathbb R^n;\mathbb C)\setminus\{0\}$ with supp $\widehat {f_N}\subset B_N(0)$ such that
 \begin{eqnarray*}
  e^{C_0 N}  \int_{B_{\hat r}^c(0)}|f_N(x)|^2 \,\mathrm dx
  \leq \int_{\mathbb R^n_x} |f_N(x)|^2 \,\mathrm dx.
 \end{eqnarray*}
 Second, (\ref{0331-th1.2-iii-1})  implies that $N\geq 1$ and each $g\in L^2(\mathbb R^n;\mathbb C)$  with supp\,$\hat g\subset B_N(0)$,
 \begin{eqnarray*}
  \int_{\mathbb R^n_x} |g(x)|^2 \,\mathrm dx
  \leq \hat C^{\frac{1}{\hat\gamma}} e^{ \frac{1-\hat\gamma}{\hat\gamma} \hat a \hat\alpha(2\widehat T N) } \int_{B_{\hat r}^c(0)}|g(x)|^2 \,\mathrm dx.
 \end{eqnarray*}
 These  two observations show that for each $N\geq N_0$,
 \begin{eqnarray*}
 e^{C_0 N}  \leq \hat C^{\frac{1}{\hat\gamma}} e^{ \frac{1-\hat\gamma}{\hat\gamma} \hat a \hat\alpha(2\widehat T N) },
 \end{eqnarray*}
from which,  it follows that
 \begin{eqnarray*}
  0<  \frac{\hat\gamma C_0}{2(1-\hat\gamma) \hat a \widehat T}
  \leq \lim_{N\rightarrow\infty}
  \frac{\hat\alpha(2\widehat T N)}{2\widehat T N}.
 \end{eqnarray*}
 This  leads to a contradiction, since $\lim_{s\rightarrow\infty} s^{-1} \hat \alpha(s)=0$. Hence,  the conclusion (iii) is true.

 \vskip 5pt
 In summary, we finish the proof of this theorem.

\end{proof}

We are on the position to prove Theorem \ref{theorem4}.
%We first show an inequality for functions whose Fourier transformation have compact supports.
%We first need an inequality for functions whose Fourier transformation have compact supports, which is Lemma \ref{lemma-0525-adjust}. This inequality plays the same role   as Corollary~\ref{lemmawang2.2} in the proof of Theorem~\ref{theorem3}.

\begin{proof}[Proof of Theorem \ref{theorem4}]

\vskip 5pt
 %\textit{Step 3. To show (\ref{0405-sch-th4-control})}\\
 Arbitrarily fix ${x^\prime},\,{x^{\prime\prime}}\in\mathbb R^n$, $r_1,\,r_2>0$, $a>0$, $T>0$ and  $u_0\in C_0^\infty(\mathbb R^n;\mathbb C)$.  Define a function $ f$ as follows:
 \begin{eqnarray}\label{0410-two-points-3}
  f(x)\triangleq   e^{-i|x|^2/4T} u(x,T;u_0),~x\in \mathbb R^n.
 \end{eqnarray}
By the same  way to get  (\ref{0229-two-points-4}), we obtain that
 \begin{eqnarray*}
  \hat f(\xi) =  (-2iT)^{n/2} e^{iT|\xi|^2} u_0(-2T\xi),~\xi\in\mathbb R^n.
 \end{eqnarray*}
This, along with  (\ref{0410-two-points-3}) and Lemma \ref{lemma-0525-adjust} (where $a$ is replaced by $2aT$), yields that
 \begin{eqnarray}\label{goubuli2.55}
  & & \int_{B_{r_2}({x^{\prime\prime}})} |u(x,T;u_0)|^2 \,\mathrm dx
  = \int_{B_{r_2}({x^{\prime\prime}})} |f(x)|^2 \,\mathrm dx
  \nonumber\\
  &\leq& C_1 r_2^n \big((2aT)^{-n}+r_1^{-n}\big)  \left(\int_{B_{r_1}({x^\prime})} |f(x)|^2 \,\mathrm dx \right)^{\theta_1^{p_1}}  \left(\int_{\mathbb R^n_\xi} |\hat f(\xi)|^2 e^{2aT|\xi|} \,\mathrm d\xi\right)^{1-\theta_1^{p_1}}
 \nonumber\\
 &\leq& C_1 r_2^n \big((aT)^{-n}+r_1^{-n}\big)
 \left(\int_{B_{r_1}({x^\prime})} |u(x,T;u_0)|^2 \,\mathrm dx \right)^{\theta_1^{p_1}}  \left(\int_{\mathbb R^n_x} |u_0(x)|^2 e^{a|x|} \,\mathrm dx\right)^{1-\theta_1^{p_1}}
 \end{eqnarray}
for some $C_1\triangleq C_1(n)>0$ and $\theta_1\triangleq \theta_1(n)\in (0,1)$, where
$$
p_1\triangleq 1+\frac{|{x^\prime}-{x^{\prime\prime}}|+r_1+r_2}{(2aT)\wedge r_1}.
$$
 Since
 $$
 (aT)^{-1}+r_1^{-1} \leq 2((aT)\wedge r_1)^{-1},
 \;\;(aT)\wedge r_1 \leq (2aT)\wedge r_1\;\;\mbox{and}\;\;\theta_1\in(0,1),
 $$
  we get from (\ref{goubuli2.55})  that
\begin{eqnarray*}
  & & \int_{B_{r_2}({x^{\prime\prime}})} |u(x,T;u_0)|^2 \,\mathrm dx
  \nonumber\\
 &\leq& C_1 r_2^n \big((aT)^{-1}+r_1^{-1}\big)^n
 \int_{\mathbb R^n} |u_0(x)|^2 e^{a|x|} \,\mathrm dx
 \left(
 \frac{\int_{B_{r_1}({x^\prime})} |u(x,T;u_0)|^2 \,\mathrm dx}{\int_{\mathbb R^n} |u_0(x)|^2 e^{a|x|} \,\mathrm dx}
 \right)^{\theta_1^{\beta_1}}
 \nonumber\\
 &\leq& C_1 r_2^n 2^n\big((aT)\wedge r_1\big)^{-n}
 \int_{\mathbb R^n} |u_0(x)|^2 e^{a|x|} \,\mathrm dx
 \left(
 \frac{\int_{B_{r_1}({x^\prime})} |u(x,T;u_0)|^2 \,\mathrm dx}{\int_{\mathbb R^n} |u_0(x)|^2 e^{a|x|} \,\mathrm dx}
 \right)^{\theta_1^{\beta_2}},
 \end{eqnarray*}
 with
 $$
 \beta_1\triangleq {1+\frac{|{x^\prime}-{x^{\prime\prime}}|+r_1+r_2}{(2aT)\wedge r_1}}\;\;\mbox{and}\;\;\beta_2\triangleq 1+\frac{|{x^\prime}-{x^{\prime\prime}}|+r_1+r_2}{(aT)\wedge r_1}.
 $$
This implies that (\ref{0405-sch-th4-control}) is true.
We end the proof of this theorem.

\end{proof}

\bigskip
\section{Proofs of Theorem~\ref{proposition3-1}-Theorem~\ref{theorem6}}

%\subsection{The Proof of Theorem \ref{proposition3-1}}

 Theorem \ref{proposition3-1} is indeed a direct consequence of Theorem \ref{theorem3}, while the proofs of both Theorem \ref{theorem5} and Theorem \ref{theorem6} rely on   Theorem \ref{theorem4} and other properties. We begin with the proof of Theorem \ref{proposition3-1}.

\begin{proof}[Proof of Theorem \ref{proposition3-1}]
Arbitrarily fix $r>0$, $T>0$, $N>0$ and $u_0\in L^2(\mathbb R^n;\mathbb C)$ with supp $u_0\subset B_N(0)$.
By a standard density argument, we can apply (i) of Theorem~\ref{theorem3} (where $a=\frac{r}{T}$)  to get that
for some $C\triangleq C(n)>0$ and $\theta\triangleq \theta(n)\in(0,1)$ (depending only on $n$),
\begin{eqnarray}\label{0229-sch-th4-control-3-1}
  \int_{\mathbb R^n} |u_0(x)|^2 \, \mathrm dx
  \leq  2C  \left(\int_{B_r^c(0)} |u(x,T;u_0)|^2 \, \mathrm dx\right)^{\theta^2}
  \left( \int_{\mathbb R^n} e^{\frac{r}{T}|x|} |u_0(x)|^2 \, \mathrm dx \right)^{1-\theta^2}
\end{eqnarray}
 At the same time, since supp $u_0\subset B_N(0)$,  we have  that
\begin{eqnarray*}
  \int_{\mathbb R^n} e^{\frac{r}{T}|x|} |u_0(x)|^2 \, \mathrm dx
  \leq e^{\frac{r}{T}N} \int_{\mathbb R^n}  |u_0(x)|^2 \, \mathrm dx.
\end{eqnarray*}
This, along with (\ref{0229-sch-th4-control-3-1}), yields that
\begin{eqnarray*}
 \int_{\mathbb R^n} |u_0(x)|^2 \, \mathrm dx
 \leq  (2C)^{\frac{1}{\theta^2}} e^{ \frac{1-\theta^2}{\theta^2} \frac{rN}{T} } \int_{B_r^c(0)} |u(x,T;u_0)|^2 \, \mathrm dx.
\end{eqnarray*}
Hence, (\ref{0229-sch-th4-control-3}) stands. This ends the proof of Theorem \ref{proposition3-1}.

\end{proof}

%\subsection{Proofs of Theorem \ref{theorem5} and Theorem \ref{theorem6}}

The following lemma will be used in the proofs of Theorem \ref{theorem5} and Theorem \ref{theorem6}.

\begin{lemma}\label{0407-lemma-int}
Let  $x\in(0,1)$ and $\theta\in(0,1)$.  Then the following conclusions are true:

\noindent(i) For each $a>0$,
\begin{eqnarray}\label{0407-lemma-int-0}
 \sum_{k=1}^\infty   x^{\theta^k} e^{-ak}
 \leq  \frac{e^a}{|\ln\theta|}  \Gamma\Big(\frac{a}{|\ln\theta|}\Big) |\ln x|^{-\frac{a}{|\ln\theta|}},
\end{eqnarray}
where $\Gamma(\cdot)$ denotes the second kind of Euler integral.

\noindent(ii) For each $\varepsilon>0$ and $\alpha>0$,
\begin{eqnarray}\label{0407-lemma-int-00}
 \sum_{k=1}^\infty   x^{\theta^k} k^{-1-\varepsilon}
 \leq \frac{4}{\varepsilon} \alpha^\varepsilon e^{\varepsilon \ln\varepsilon+\varepsilon + e\alpha^{-1}\theta^{-1}}  \big(\ln (\alpha|\ln x|+ e)\big)^{-\varepsilon}.
\end{eqnarray}

\end{lemma}

\begin{proof}
(i) Since $\theta\in(0,1)$, it follows that
\begin{eqnarray}\label{0407-lemma-int-1}
   \sum_{k=1}^\infty   x^{\theta^k} e^{-ak}
 \leq \sum_{k=1}^\infty  \int_{k} ^{k+1}  x^{\theta^\tau} e^{-a(\tau-1)} \,\mathrm d\tau
 = e^a\int_1^\infty x^{\theta^\tau} e^{-a\tau} \,\mathrm d\tau.
\end{eqnarray}
Next, because $x\in(0,1)$, we find that
\begin{eqnarray*}
 x^{\theta^\tau}=\exp[-e^{\ln|\ln x|+\tau\ln\theta}].
\end{eqnarray*}
Then, by changing variable $s=\ln|\ln x|+ \tau \ln \theta$ and noticing that $\theta\in(0,1)$, we find   that
\begin{eqnarray*}
 \int_1^\infty x^{\theta^\tau} e^{-a\tau} \,\mathrm d\tau
 &=& \int_{-\infty}^{\ln|\ln x|+\ln\theta} \frac{1}{|\ln\theta|} e^{-e^s} e^{\frac{a}{|\ln\theta|}(s-\ln|\ln x|)} \,\mathrm ds
 \nonumber\\
 &=& \frac{|\ln x|^{-\frac{a}{|\ln\theta|}}}{|\ln\theta|} \int_{-\infty}^{\ln|\ln x|+\ln\theta} e^{\frac{a}{|\ln\theta|}s-e^s} \,\mathrm ds,
\end{eqnarray*}
from which, it follows that
\begin{eqnarray*}
 & & \int_1^\infty x^{\theta^\tau} e^{-a\tau} \,\mathrm d\tau
 = \frac{|\ln x|^{-\frac{a}{|\ln\theta|}}}{|\ln\theta|} \int_{-\infty}^{\ln|\ln x|+\ln\theta} (e^s)^{\frac{a}{|\ln\theta|}} e^{-e^s} e^{-s}\,\mathrm d e^s
 \nonumber\\
 &=& \frac{|\ln x|^{-\frac{a}{|\ln\theta|}}}{|\ln\theta|} \int_{0}^{|\ln x|\theta} \eta^{\frac{a}{|\ln\theta|}-1} e^{-\eta} \,\mathrm d\eta
 \leq \frac{|\ln x|^{-\frac{a}{\ln\theta}}}{|\ln\theta|} \int_{0}^{\infty} \eta^{\frac{a}{|\ln\theta|}-1} e^{-\eta} \,\mathrm d\eta.
\end{eqnarray*}
This, along with (\ref{0407-lemma-int-1}), leads to (\ref{0407-lemma-int-0}) and ends the proof of the conclusion (i).

(ii) Since $\theta\in(0,1)$, it follows that
\begin{eqnarray}\label{0407-lemma-int-2}
   \sum_{k=1}^\infty   x^{\theta^k} k^{-1-\varepsilon}
 \leq \sum_{k=1}^\infty  \frac{(k+1)^{1+\varepsilon}}{k^{1+\varepsilon}} \int_{k} ^{k+1}  x^{\theta^\tau} \tau^{-1-\varepsilon} \,\mathrm d\tau
 \leq 2^{1+\varepsilon} \int_1^\infty x^{\theta^\tau} \tau^{-1-\varepsilon} \,\mathrm d\tau.
\end{eqnarray}
Next, because $x\in(0,1)$, we see that
\begin{eqnarray*}
 x^{\theta^\tau}=\exp[-|\ln x|e^{\tau\ln\theta}].
\end{eqnarray*}
Since $\theta\in(0,1)$, the above yields that
\begin{eqnarray*}
 & &  \int_1^\infty x^{\theta^\tau} \tau^{-1-\varepsilon} \,\mathrm d\tau
 = |\ln\theta|^\varepsilon \int^{\ln\theta}_{-\infty}  e^{-|\ln x|e^s} |s|^{-1-\varepsilon} \,\mathrm ds
 \nonumber\\
 &=& |\ln\theta|^\varepsilon \int^{\ln\theta}_{-\infty}
 e^{-\frac{|\ln x|}{e^{-s}}}  |\ln e^{-s}|^{-1-\varepsilon} (-e^{s}) \,\mathrm  de^{-s}
 = |\ln\theta|^\varepsilon \int_{\frac{1}{\theta}}^{\infty}
 e^{-\frac{|\ln x|}{\eta}}  |\ln \eta|^{-1-\varepsilon} \eta^{-1} \,\mathrm  d\eta.
\end{eqnarray*}
From this, we find that for each $N\geq \frac{1}{\theta}$,
\begin{eqnarray*}
   \int_1^\infty x^{\theta^\tau} \tau^{-1-\varepsilon} \,\mathrm d\tau
 &=& |\ln\theta|^\varepsilon \Big[ \int_{\frac{1}{\theta}}^{N}
 e^{-\frac{|\ln x|}{\eta}}  |\ln \eta|^{-1-\varepsilon} \eta^{-1} \,\mathrm  d\eta
 + \int_{N}^{\infty}
 e^{-\frac{|\ln x|}{\eta}}  |\ln \eta|^{-1-\varepsilon} \eta^{-1} \,\mathrm  d\eta \Big]
 \nonumber\\
 &\leq& |\ln\theta|^\varepsilon \Big[ e^{-\frac{|\ln x|}{N}} \int_{\frac{1}{\theta}}^{N}
   |\ln \eta|^{-1-\varepsilon} \eta^{-1} \,\mathrm  d\eta
 + \int_{N}^{\infty}  |\ln \eta|^{-1-\varepsilon} \eta^{-1} \,\mathrm  d\eta \Big]
 \nonumber\\
 &=& \frac{1}{\varepsilon} \Big[ e^{-\frac{|\ln x|}{N}} \big(1-|\ln\theta|^\varepsilon (\ln N)^{-\varepsilon}\big) + |\ln\theta|^\varepsilon (\ln N)^{-\varepsilon} \Big]
 \nonumber\\
 &\leq& \frac{1}{\varepsilon} \big[ e^{-\frac{|\ln x|}{N}}  + |\ln\theta|^\varepsilon (\ln N)^{-\varepsilon} \big].
\end{eqnarray*}
Let $\alpha>0$. Taking $N=\sqrt{\alpha|\ln x|+ e\theta^{-2}}$ in the above inequality leads to that
\begin{eqnarray}\label{Tianjin3.6}
   \int_1^\infty x^{\theta^\tau} \tau^{-1-\varepsilon} \,\mathrm d\tau
  \leq \frac{1}{\varepsilon} \Big[ e^{-\frac{|\ln x|}{\sqrt{\alpha|\ln x|+ e\theta^{-2}}}}
   + |\ln\theta|^\varepsilon 2^\varepsilon \big( \ln (\alpha|\ln x|+ e\theta^{-2}) \big)^{-\varepsilon}
   \Big]
\end{eqnarray}
Since
\begin{eqnarray*}
 -\frac{|\ln x|}{\sqrt{\alpha|\ln x|+ e\theta^{-2}}}
&=& -\frac{\alpha^{-1}(\alpha|\ln x|+ e\theta^{-2})}{\sqrt{\alpha|\ln x|+ e\theta^{-2}}}
+ \frac{\alpha^{-1} e\theta^{-2}}{\sqrt{\alpha|\ln x|+ e\theta^{-2}}}
\nonumber\\
&\leq& -\alpha^{-1}\sqrt{\alpha|\ln x|+ e\theta^{-2}}+\alpha^{-1}e\theta^{-1},
\end{eqnarray*}
and because
\begin{eqnarray*}
  0<\theta<1
  \;\;\mbox{and}\;\;
   (\ln s)^\varepsilon
   %\leq s^\varepsilon
   \leq \alpha^\varepsilon e^{\varepsilon\ln\varepsilon-\varepsilon+\alpha^{-1}s}
   \;\;\mbox{for all}\;\; s>1,
\end{eqnarray*}
we  find from (\ref{Tianjin3.6})     that
\begin{eqnarray*}
& &  \int_1^\infty x^{\theta^\tau} \tau^{-1-\varepsilon} \,\mathrm d\tau
  \leq \frac{1}{\varepsilon} \Big[ e^{-\alpha^{-1}\sqrt{\alpha|\ln x|+ e\theta^{-2}}+\alpha^{-1}e\theta^{-1}}  + |\ln\theta^{-1}|^\varepsilon 2^\varepsilon  \big( \ln (\alpha|\ln x|+ e\theta^{-2} ) \big)^{-\varepsilon} \Big]
  \nonumber\\
  &\leq& \frac{1}{\varepsilon} \Big[ \alpha^\varepsilon e^{\varepsilon \ln\varepsilon-\varepsilon + e\alpha^{-1}\theta^{-1}} 2^\varepsilon \big( \ln (\alpha|\ln x|+ e\theta^{-2}) \big)^{-\varepsilon}
    + \alpha^\varepsilon e^{\varepsilon \ln\varepsilon-\varepsilon + \alpha^{-1}\theta^{-1}} 2^\varepsilon \big( \ln (\alpha|\ln x|+ e\theta^{-2}) \big)^{-\varepsilon}
    \Big]
    \nonumber\\
    &\leq& \frac{2}{\varepsilon} \alpha^\varepsilon e^{\varepsilon \ln\varepsilon-\varepsilon + e\alpha^{-1}\theta^{-1}} 2^\varepsilon \big( \ln (\alpha|\ln x|+ e) \big)^{-\varepsilon}.
\end{eqnarray*}
This, together with  (\ref{0407-lemma-int-2}), leads to  (\ref{0407-lemma-int-00}), and  ends the proof of the conclusion (ii).

\vskip 5pt
In summary, we finish the proof of this lemma.

\end{proof}

We now on the position to prove Theorem \ref{theorem5}.

\begin{proof}[Proof of Theorem \ref{theorem5}]
Let $x_0,\,{x^\prime}\in\mathbb R^n$, $r>0$, $a>0$, $b> 0$ and $T>0$.  It suffices to show the desired inequality (\ref{0405-sch-th5-control-1}) for any
$u_0\in C_0^\infty(\mathbb R^n;\mathbb C)\setminus\{0\}$ and $\varepsilon\in (0,1)$.

For this purpose, we  arbitrarily fix $u_0\in C_0^\infty(\mathbb R^n;\mathbb C)\setminus\{0\}$. Define the following three numbers
\begin{eqnarray*}
 A_1 \triangleq \int_{\mathbb R^n} |u_0(x)|^2 e^{a|x|} \mathrm dx;
 ~
 B_1 \triangleq  \int_{B_r(x_0)} |u(x,T;u_0)|^2 \,\mathrm dx;\;
  R_b \triangleq \int_{\mathbb R^n} e^{-b|x-{x^\prime}|}  |u(x,T;u_0)|^2 \,\mathrm dx.
\end{eqnarray*}
The proof is divided into the following several steps.

\vskip 5pt
\textit{Step 1. To  prove that there exist two positive constants $C_1 \triangleq C_1(n)$ and $C_2 \triangleq C_2(n)$ so that
\begin{eqnarray}\label{0410-th1.5-1}
   R_b\leq  C_3(x_0,{x^\prime},r,a,b,T)  g\Big(\frac{A_1}{B_1}\Big)  A_1,
\end{eqnarray}
where
\begin{eqnarray}\label{0410-th1.5-1-1}
 C_3(x_0,{x^\prime},r,a,b,T) \triangleq   1  +
 C_1 \Gamma\big(C_2 b((aT)\wedge r)\big) \exp\left[b^{-1}((aT)\wedge r)^{-1}+ b  (|x_0-{x^\prime}|+r)\right],
\end{eqnarray}
and
\begin{eqnarray}\label{0410-th1.5-2}
  g(\eta)\triangleq  (\ln \eta)^{-C_2b ((aT)\wedge r)},~\eta>1
\end{eqnarray}}

\noindent In fact, by Theorem \ref{theorem4} (with $({x^\prime},{x^{\prime\prime}},r_1,r_2)$ being replaced by $(x_0,{x^\prime},r,2kb^{-1})$), with $k\in\mathbb N^+$, and the definitions of $A_1$ and $B_1$, we see that
for each $k\in\mathbb N^+$,
\begin{eqnarray*}
 \int_{B_{2kb^{-1}}({x^\prime})} |u(x,T;u_0)|^2 \,\mathrm dx
  \leq C 2^n  b^{-n} k^n \big((aT)\wedge r\big)^{-n} \Big(\frac{B_1}{A_1}\Big)^{\theta^{1+\frac{|x_0-{x^\prime}|+2kb^{-1}+r}{(aT)\wedge r}}} A_1
\end{eqnarray*}
for some $C>0$ and $\theta\in(0,1)$ depending only on $n$.
This, along with the fact that
$k\leq n e^{\frac{1}{n}k}$ for all $k\in\mathbb N^+$,
yields that
\begin{eqnarray}\label{0410-th1.5-2-1}
 & & \int_{\mathbb R^n} e^{-b|x-{x^\prime}|}  |u(x,T;u_0)|^2 \,\mathrm dx
 \leq \sum_{k=1}^\infty \int_{2(k-1)b^{-1}\leq |x-{x^\prime}|< 2kb^{-1}} e^{-2(k-1)} |u(x,T;u_0)|^2 \,\mathrm dx
 \nonumber\\
 &\leq& C (2n)^n b^{-n} \big((aT)\wedge r\big)^{-n} e^2 \left(\sum_{k=1}^\infty e^{-k} \Big(\frac{B_1}{A_1}\Big)^{\theta^{1+\frac{|x_0-{x^\prime}|+2kb^{-1}+r}{(aT)\wedge r}}}  \right) A_1.
\end{eqnarray}
Meanwhile, since $B_1<A_1$ (which follows from   the definitions of $A_1$ and $B_1$, the conversation law for the Schr\"{o}dinger equation and the fact that  $u_0\neq 0$),
we can apply (i) of Lemma \ref{0407-lemma-int}, where
$$
(a, x, \theta)= \Big(1, \Big({B_1}/{A_1}\Big)^{\theta^{1+\frac{|x_0-{x^\prime}|+r}{(aT)\wedge r}}}, \theta^{\frac{2}{b((aT)\wedge r)}}\Big),
$$
to get that
$$
\sum_{k=1}^\infty e^{-k} \Big(\frac{B_1}{A_1}\Big)^{\theta^{1+\frac{|x_0-{x^\prime}|+2kb^{-1}+r}{(aT)\wedge r}}}
\leq
 \frac{ e b((aT)\wedge r) }{2|\ln\theta|}
\Gamma\left(\frac{b ((aT)\wedge r)}{2|\ln\theta|}\right)
 \left[\theta^{1+\frac{|x_0-{x^\prime}|+r}{(aT)\wedge r}}
 |\ln \frac{B_1}{A_1}| \right]^{-\frac{b ((aT)\wedge r)}{2|\ln\theta|}}.
$$
This, together with (\ref{0410-th1.5-2-1}) and the facts that
$x^{n-1}\leq (n-1)! e^{x}$ for all $x>0$ and that $(aT)\wedge r \leq r$,
indicates that
\begin{eqnarray*}
 & & \int_{\mathbb R^n} e^{-b|x-{x^\prime}|}  |u(x,T;u_0)|^2 \,\mathrm dx
 \\
 &\leq& C (2n)^n b^{-n} \big((aT)\wedge r\big)^{-n} e^3
  \frac{b((aT)\wedge r)}{|\ln\theta|}
 \Gamma\left(\frac{b ((aT)\wedge r)}{2|\ln\theta|}\right)
 \left[\theta^{1+\frac{|x_0-{x^\prime}|+r}{(aT)\wedge r}}
 |\ln \frac{B_1}{A_1}| \right]^{-\frac{b ((aT)\wedge r)}{2|\ln\theta|}}  A_1
 \nonumber\\
 &=& \frac{C(2n)^n e^3}{|\ln\theta|} \big(b ((aT)\wedge r)\big)^{-n+1}    e^{ \frac{1}{2} b  ((aT)\wedge r+|x_0-{x^\prime}|+r)} \Gamma\left(\frac{b ((aT)\wedge r)}{2|\ln\theta|}\right)
 \Big(\ln \frac{A_1}{B_1} \Big)^{-\frac{b ((aT)\wedge r)}{2|\ln\theta|}}  A_1
 \nonumber\\
 &\leq& \frac{C(2n)^n e^3}{|\ln\theta|}  (n-1)! e^{b^{-1} ((aT)\wedge r)^{-1}+ b  (|x_0-{x^\prime}|+r)}
 \Gamma\left(\frac{b ((aT)\wedge r)}{2|\ln\theta|}\right)
 \Big(\ln \frac{A_1}{B_1} \Big)^{-\frac{b ((aT)\wedge r)}{2|\ln\theta|}}  A_1 .
 \nonumber
\end{eqnarray*}
This, as well as (\ref{0410-th1.5-2}), shows  (\ref{0410-th1.5-1}).

\vskip 5pt
\textit{Step 2. To  show (\ref{0405-sch-th5-control-1}) for the above-mentioned $u_0$ and any $\varepsilon\in (0,1)$ }

\noindent
Let $C_1\triangleq C_1(n)$ and $C_2\triangleq C_2(n)$ be given by Step 1.
Since
\begin{eqnarray*}
 \varepsilon e^{\varepsilon^{-1-\frac{\alpha}{ b ((aT)\wedge r)}}}
 \leq \varepsilon e^{\varepsilon^{-1-\frac{\beta}{ b ((aT)\wedge r)}}},
 \;\;\mbox{when}\;\;  0<\alpha<\beta\;\;\mbox{and}\;\;\varepsilon\in (0,1),
\end{eqnarray*}
it suffices to show that for each $\varepsilon\in(0,1)$,
\begin{eqnarray}\label{0410-th1.5-3}
  R_b\leq  C_4(x_0,{x^\prime},r,a,b,T)
  \left(\varepsilon A_1 +  \varepsilon e^{\varepsilon^{-1-\frac{1}{C_2b ((aT)\wedge r)}}} B_1 \right),
\end{eqnarray}
where
\begin{eqnarray*}\label{0419-th1.5-step1-4}
C_4(x_0,{x^\prime},r,a,b,T) \triangleq C_1 \exp\left\{ 2(C_1+C_2^{-1}+1)(C_2+1) \Big[1+ \frac{ b^{-1} + |x_0-{x^\prime}|+r}{(aT)\wedge r} \Big] \right\}.
\end{eqnarray*}

 The proof of (\ref{0410-th1.5-3}) is organized by  two parts.

\vskip 5pt
\textit{Part 2.1. To show (\ref{0410-th1.5-3}) in the case that $b\leq\frac{1}{C_2((aT)\wedge r)}$}

 \noindent
First, we claim that for each $\varepsilon\in (0,1)$,
\begin{eqnarray}\label{0419-th1.5-step1-2}
  R_b\leq  C_3
  \left(\varepsilon A_1 +  \varepsilon e^{\varepsilon^{-\frac{1}{C_2b ((aT)\wedge r)}}} B_1 \right),
\end{eqnarray}
where $C_3\triangleq C_3(x_0,{x^\prime},r,a,b,T)$ is given by (\ref{0410-th1.5-1-1}).
In fact, for an arbitrarily fix $\varepsilon>0$, there are only two possibilities: either $R_b\leq C_3 \varepsilon A_1$ or $R_b>C_3 \varepsilon A_1$.
In the first case, (\ref{0419-th1.5-step1-2}) is obvious. In the second case, we first claim that
\begin{eqnarray}\label{0410-th1.5-3-1}
 0<\varepsilon<\frac{R_b}{C_3 A_1}<1.
\end{eqnarray}
Indeed, the first and the second  inequalities in (\ref{0410-th1.5-3-1}) is clear. To prove the last inequality in (\ref{0410-th1.5-3-1}),
two facts are given in order: First,
we  observe from  (\ref{0410-th1.5-1-1}) that $C_3>1$. Second, by the definitions of $A_1$ and $R_b$, using the conversation law of the Schr\"{o}dinger equation, we find that
\begin{eqnarray*}
R_b=\int_{\mathbb{R}^n}e^{-b|x-{x^\prime}|}|u(x,T;u_0)|^2dx
\leq\int_{\mathbb{R}^n}|u(x,T;u_0)|^2dx=\int_{\mathbb{R}^n}|u_0(x)|^2dx\leq \int_{\mathbb{R}^n}e^{a|x|}|u_0(x)|^2dx=
 A_1.
\end{eqnarray*}
These two facts lead to the last inequality in (\ref{0410-th1.5-3-1}) at once.

Since $b\leq\frac{1}{C_2((aT)\wedge r)}$, we see that the function
$x\mapsto xe^{x^{-\frac{1}{C_2 b ((aT)\wedge r)}}}$, with its domain $(0,1)$,
 is decreasing. This, along with   (\ref{0410-th1.5-3-1}), indicates    that
\begin{eqnarray}\label{0410-th1.5-5}
  \frac{R_b}{C_3 A_1} e^{(\frac{R_b}{C_3 A_1})^{-\frac{1}{C_2 b ((aT)\wedge r)}}}
  \leq    \varepsilon e^{\varepsilon^{-\frac{1}{C_2 b ((aT)\wedge r)}}}.
\end{eqnarray}
Meanwhile, since the function:  $f(x)=e^{x^{-\frac{1}{C_2 b ((aT)\wedge r)}}}$, with
its domain $(0,\infty)$,
  is decreasing and its inverse is the function $g$ (given by (\ref{0410-th1.5-2})), we get from (\ref{0410-th1.5-1}) that
\begin{eqnarray}\label{0410-th1.5-4}
 \frac{A_1}{B_1} = f(g(\frac{A_1}{B_1}))
  \leq f(\frac{R_b}{C_3 A_1}) = e^{(\frac{R_b}{C_3 A_1})^{-\frac{1}{C_2 b ((aT)\wedge r)}}}.
\end{eqnarray}
From (\ref{0410-th1.5-4}) and (\ref{0410-th1.5-5}), it follows  that
\begin{eqnarray*}
 R_b &=& C_3 \frac{R_b}{C_3 A_1} \frac{A_1}{B_1} B_1
 \leq C_3 \Big[ \frac{R_b}{C_3 A_1} e^{(\frac{R_b}{C_3 A_1})^{-\frac{1}{C_2 b ((aT)\wedge r)}}}
 \Big]  B_1
 \nonumber\\
 &\leq& C_3  \varepsilon e^{\varepsilon^{-\frac{1}{C_2 b ((aT)\wedge r)}}} B_1.
\end{eqnarray*}
Since $\varepsilon$ was arbitrarily taken from $(0,1)$, the above leads to (\ref{0419-th1.5-step1-2}) for the case that $R_b>C_3\varepsilon A_1$.
Hence, (\ref{0419-th1.5-step1-2}) is true.

Next, we claim that
\begin{eqnarray}\label{0419-th1.5-step1-3}
C_3(x_0,{x^\prime},r,a,b,T)  \leq \exp\left\{ 2(C_1+C_2^{-1}+1) \Big[ 1+ \frac{ b^{-1} + |x_0-{x^\prime}|+r}{(aT)\wedge r} \Big] \right\}.
\end{eqnarray}
 To this end, we first observe that for each $s\in(0,1]$,
\begin{eqnarray}\label{TTianjin3.18}
\Gamma(s) &=& \int_0^\infty e^{-x} x^{s-1} \,\mathrm dx
= \int_0^1 e^{-x} x^{s-1} \,\mathrm dx + \int_1^\infty e^{-x} x^{s-1} \,\mathrm dx
\nonumber\\
&\leq& \sum_{k=0}^\infty \int_{e^{-k-1}}^{e^{-k}}  e^{(1-s)(k+1)} \,\mathrm dx + \int_1^\infty e^{-x}  \,\mathrm dx
\nonumber\\
&=& (e-1) \frac{1}{e^s-1} + e^{-1}  \leq (e-1)s^{-1}+1 \leq 2s^{-1}+1
\leq e^{2s^{-1}}.
\end{eqnarray}
Since we are in the case that $b\leq\frac{1}{C_2((aT)\wedge r)}$, it follows from  (\ref{0410-th1.5-1-1})
and (\ref{TTianjin3.18}), with  $s=C_2b((aT)\wedge r)$, that
\begin{eqnarray*}
 C_3(x_0,{x^\prime},r,a,b,T) &\leq& 1+ e^{C_1} e^{2C_2^{-1} b^{-1}((aT)\wedge r)^{-1}}
 \exp\left[b^{-1}((aT)\wedge r)^{-1}+ b  (|x_0-{x^\prime}|+r)\right]
 \nonumber\\
 &\leq& e\cdot \exp\left[C_1+ (2C_2^{-1}+1)b^{-1}((aT)\wedge r)^{-1}+ C_2^{-1}
 \frac{|x_0-{x^\prime}|+r}{(aT)\wedge r } \right].
\end{eqnarray*}
This leads to (\ref{0419-th1.5-step1-3}).

Now, by (\ref{0419-th1.5-step1-2}) and (\ref{0419-th1.5-step1-3}), we reach the aim of Part 2.1.

\vskip 5pt
\textit{Part 2.2. To show (\ref{0410-th1.5-3}) in the case that $b>\frac{1}{C_2((aT)\wedge r)}$}

\noindent In this case,   it follows from  the definition of $R_b$   that
$R_b  \leq R_{\frac{1}{C_2((aT)\wedge r)}}$.
Then by (\ref{0419-th1.5-step1-2}) and (\ref{0419-th1.5-step1-3}) (where $b$ is replaced by $\frac{1}{C_2((aT)\wedge r)}$), we find that for each $\varepsilon\in(0,1)$,
\begin{eqnarray*}
 R_b &\leq&  \exp\left\{ 2(C_1+C_2^{-1}+1)  \Big[1+  \frac{C_2((aT)\wedge r)+|x_0-{x^\prime}|+r}{(aT)\wedge r} \Big] \right\}    \left(\varepsilon A_1 +  \varepsilon e^{\varepsilon^{-1}} B_1 \right)
 \nonumber\\
 &\leq& \exp\left\{ 2(C_1+C_2^{-1}+1)   \Big[1+C_2+ \frac{b^{-1}+|x_0-{x^\prime}|+r}{(aT)\wedge r} \Big] \right\}
 \left(\varepsilon A_1 +  \varepsilon e^{\varepsilon^{-1-\frac{1}{C_2b ((aT)\wedge r)}}} B_1 \right),
\end{eqnarray*}
from which, we reach the aim of Part 2.2.

\vskip 5pt
In summary, we finish the proof of (\ref{0410-th1.5-3}), which completes the proof of the theorem.

\end{proof}

Next, we are going to prove Theorem \ref{theorem6}. Before it, one lemma will be introduced.

\begin{lemma}\label{lemma-0411-th1.6}
Given  $k\in \mathbb{N}^+$, there exists a constant $C(k,n)$ so that for any $T>0$ and $u_0\in C_0^\infty(\mathbb R^n;\mathbb C)$,
\begin{eqnarray}\label{0411-th1.6-lemma-0}
  \int_{\mathbb R^n} |x|^{2k} |u(x,T;u_0)|^2 \,\mathrm dx
 \leq   C(k,n) (1+T)^{2k}  \left( \|u_0\|^2_{H^{2k}(\mathbb R^n;\mathbb C)}
  + \int_{\mathbb R^n} |x|^{4k} | u_0(x)|^2 \,\mathrm dx \right).
\end{eqnarray}

\end{lemma}

\begin{proof}
Arbitrarily fix $k\in\mathbb N^+$, $T>0$ and $u_0\in C_0^\infty(\mathbb R^n;\mathbb C)$. For each $x\in\mathbb R^n$, write $x=(x_1,\cdots,x_n)$. One can directly check that for each
 $j\in \{1,\dots,n\}$, the operators $\big(x_j+2i(t-T)\partial_{x_j}\big)^k$ and $i\partial_t+\Delta$ are commutative. This yields that for each $j\in \{1,\dots,n\}$,
\begin{eqnarray*}
 (i\partial_t+\Delta)  \big(x_j+2i(t-T)\partial_{x_j}\big)^k u(x,t;u_0)
 &=& \big(x_j+2i(t-T)\partial_{x_j}\big)^k  (i\partial_t+\Delta) u(x,t;u_0)
 \nonumber\\
 &=&0,~(x,t)\in\mathbb R^n \times \mathbb R^+,
\end{eqnarray*}
from which, it follows that for each $j\in \{1,\dots,n\}$,
\begin{eqnarray*}
 u(x,t;u_j)= \big(x_j+2i(t-T)\partial_{x_j}\big)^k u(x,t;u_0),~(x,t)\in\mathbb R^n \times \mathbb R^+,
\end{eqnarray*}
where $u_j(x)\triangleq  (x_j-2iT\partial_{x_j})^k u_0(x),~x\in\mathbb R^n.$
In particular, we have that for each $j\in \{1,\dots,n\}$,
\begin{eqnarray*}
  u(x,T;u_j)= x_j^k u(x,T;u_0),~x\in\mathbb R^n.
\end{eqnarray*}
These, along with  the conversation law for the Schr\"{o}dinger equation, yields that for each $j\in \{1,\dots,n\}$,
\begin{eqnarray}\label{0409-lemma-th1.6-1}
 & &   \int_{\mathbb R^n}  |x_j^k u(x,T;u_0)|^2 \,\mathrm dx
 =  \int_{\mathbb R^n}  |u(x,T;u_j)|^2 \,\mathrm dx
 \nonumber\\
 &=& \int_{\mathbb R^n}  |u_j(x)|^2 \,\mathrm dx
 = \int_{\mathbb R^n}  |(x_j-2iT\partial_{x_j})^k u_0(x)|^2 \,\mathrm dx.
\end{eqnarray}

Next, we claim that there exists  $C_1(k,n)>0$  so that for each $j\in \{1,\dots,n\}$,
\begin{eqnarray}\label{0411-th1.6-lemma-2}
  \int_{\mathbb R^n}  |(x_j-2iT\partial_{x_j})^k u_0(x)|^2 \,\mathrm dx
  \leq C_1(k,n) (1+T)^{2k}  \left( \|u_0\|^2_{H^{2k}(\mathbb R^n;\mathbb C)}
  + \int_{\mathbb R^n} |x|^{4k} | u_0(x)|^2 \,\mathrm dx \right).
\end{eqnarray}
For this purpose, we arbitrarily fix $j$ from $\{1,\dots,n\}$.
  Since the operator $ (x_j-i\partial_{x_j})^{2k}$ is a polynomial of $x_j$ and $\partial_{x_j}$, with degree $2k$, and because
   $$
   [\partial_{x_j}, x_j] \triangleq \partial_{x_j} x_j - x_j \partial_{x_j} =1,
   $$
    the polynomial $ (x_j-i\partial_{x_j})^{2k}$ is a linear combination of the following monomials
\begin{eqnarray*}
 \left\{ x_j^{r}\partial_{x_j}^{s} ~:~ 0\leq r+s\leq 2k,~
 r,\,s\in\mathbb N^+\cup\{0\} \right\}.
\end{eqnarray*}
By this, we see  that
\begin{eqnarray}\label{0411-th1.6-lemma-1-1}
 & &  \int_{\mathbb R^n}  |(x_j-i\partial_{x_j})^k v(x)|^2 \,\mathrm dx
 = \int_{\mathbb R^n}  \left\langle (x_j-i\partial_{x_j})^{2k} v(x), v(x) \right\rangle_{\mathbb C} \,\mathrm dx
 \nonumber\\
 &\leq& C_2(k,n) \sum_{0\leq r+s\leq 2k}  \int_{\mathbb R^n}  | \langle \partial_{x_j}^{s} v(x), x_j^{r} v(x) \rangle_{\mathbb C} | \,\mathrm dx,
 %\nonumber\\
% &\leq&  C_1(k,n) \sum_{m=0}^{2k}  \left( \int_{\mathbb R^n}  |\partial_{x_j}^{m} v(x)|^2 \,\mathrm dx  +  \int_{\mathbb R^n}  |x_j^{m} v(x)|^2 \,\mathrm dx \right)
% \nonumber\\
% &=& C_1(k,n) \sum_{m=0}^{2k}  \left( \int_{\mathbb R^n_\xi}  |\xi_j^{m} \hat v(\xi)|^2 \,\mathrm d\xi  +  \int_{\mathbb R^n_x}  |x_j^{m} v(x)|^2 \,\mathrm dx \right)
% \nonumber\\
% &\leq& C_1(k,n)  \left( \int_{\mathbb R^n_\xi} |\xi|^{4k}  | \hat v(\xi)|^2 \,\mathrm d\xi  +  \int_{\mathbb R^n_x} |x|^{4k}  |v(x)|^2 \,\mathrm dx \right).
\end{eqnarray}
where $v$ is the function defined by
\begin{eqnarray}\label{0411-th1.6-lemma-1}
 v(x) \triangleq u_0(\sqrt{2T}x),~ x\in\mathbb R^n,
\end{eqnarray}
and where and through the proof,  $C_2(k,n)$ stands for a positive constant (depending only on $k,n$), which may vary in different contexts.

From (\ref{0411-th1.6-lemma-1}) and (\ref{0411-th1.6-lemma-1-1}), we find that
\begin{eqnarray*}
 & & \int_{\mathbb R^n}  |(x_j-2iT\partial_{x_j})^k u_0(x)|^2 \,\mathrm dx
  = \int_{\mathbb R^n}  |(x_j-2iT\partial_{x_j})^k v(\frac{x}{\sqrt{2T}})|^2 \,\mathrm dx
  \nonumber\\
  &=& (2T)^{k+\frac{n}{2}} \int_{\mathbb R^n}  |(x_j-i\partial_{x_j})^k v(x)|^2 \,\mathrm dx
  \nonumber\\
  &\leq& C_2(k,n) (2T)^{k+\frac{n}{2}} \sum_{0\leq r+s\leq 2k}  \int_{\mathbb R^n}  | \langle \partial_{x_j}^{s} v(x), x_j^{r} v(x) \rangle_{\mathbb C} | \,\mathrm dx
  \nonumber\\
  %&=& C_1(k,n) (2T)^{k+\frac{n}{2}} \sum_{0\leq r+s\leq 2k}  \int_{\mathbb R^n}  | \langle \partial_{x_j}^{s} u_0(\sqrt{2T}x), x_j^{r} u_0(\sqrt{2T}x) \rangle_{\mathbb C} | \,\mathrm dx
%  \nonumber\\
  &=& C_2(k,n)  \sum_{0\leq r+s\leq 2k}  (2T)^{\frac{2k+s-r}{2}}  \int_{\mathbb R^n}  | \langle \partial_{x_j}^{s} u_0(x), x_j^{r} u_0(x) \rangle_{\mathbb C} | \,\mathrm dx
  \nonumber\\
  &\leq& C_2(k,n) (1+T)^{2k} \sum_{0\leq r+s\leq 2k}
  \left( \int_{\mathbb R^n}  |\partial_{x_j}^{s} u_0(x)|^2 \,\mathrm dx
  + \int_{\mathbb R^n}  | x_j^{r} u_0(x)|^2 \,\mathrm dx \right)
  \nonumber\\
  &\leq& C_2(k,n) (1+T)^{2k}  \left( \|u_0\|^2_{H^{2k}(\mathbb R^n;\mathbb C)}
  + \int_{\mathbb R^n} |x|^{4k} |  u_0(x)|^2 \,\mathrm dx \right).
\end{eqnarray*}
This leads to (\ref{0411-th1.6-lemma-2}).

Finally, since
\begin{eqnarray*}
 |x|^{2k}= n^k \Big(\frac{x_1^2+\cdots+x_n^2}{n}\Big)^k
 \leq n^{k-1} \big(x_1^{2k}+\cdots+x_n^{2k}\big)
 , ~x\in\mathbb R^n,
\end{eqnarray*}
  it follows from (\ref{0409-lemma-th1.6-1})  that
\begin{eqnarray*}
 & &   \int_{\mathbb R^n} |x|^{2k} |u(x,T;u_0)|^2 \,\mathrm dx
  \leq n^{k-1} \sum_{j=1}^n \int_{\mathbb R^n}  |x_j^k u(x,T;u_0)|^2 \,\mathrm dx
  \nonumber\\
  &\leq&   n^{k-1} \sum_{j=1}^n \int_{\mathbb R^n}  |(x_j-2iT\partial_{x_j})^k u_0(x)|^2 \,\mathrm dx.
\end{eqnarray*}
This, along with (\ref{0411-th1.6-lemma-2}), leads to (\ref{0411-th1.6-lemma-0}). We end the proof of this lemma.

\end{proof}

\begin{remark}\label{tianjinremark3.3}
Lemma~\ref{lemma-0411-th1.6} gives a quantitative property for solutions of (\ref{0229-sch-1}). This quantitative property
is comparable with the following qualitative property for solutions of (\ref{0229-sch-1}):
If $u_0 \in L^2(|x|^{4k}dx) \cap H^{2k}$ for some $k\in\mathbb N^+\cup\{0\}$, then
 \begin{eqnarray*}\label{0423-th1.6-lemma-pre}
  e^{iT\triangle}u_0 \in L^2(|x|^{4k}dx) \cap H^{2k} \;\;\mbox{for all}\;\;  T\in \mathbb R^+.
 \end{eqnarray*}
  The above-mention  qualitative property was given in  \cite{NP}.

\end{remark}

We now give  the proof of Theorem \ref{theorem6}.

\begin{proof}[Proof of Theorem \ref{theorem6}]
Let $x_0\in\mathbb R^n$, $r>0$, $a>0$ and $T>0$.
When $u_0=0$,  (\ref{0405-sch-th5-control-2}) holds clearly for all $\varepsilon\in (0,1)$.
We now  arbitrarily fix $u_0\in C_0(\mathbb R^n;\mathbb C)\setminus\{0\}$. Define the following three numbers:
\begin{eqnarray*}
 A_2&\triangleq& \int_{\mathbb R^n} |u_0(x)|^2 e^{a|x|} \mathrm dx
 +  \|u_0\|^2_{H^{n+3}(\mathbb R^n;\mathbb C)},
 ~
 B_2\triangleq  \int_{B_r(x_0)} |u(x,T;u_0)|^2 \,\mathrm dx,
  \nonumber\\
 A_3&\triangleq& \int_{\mathbb R^n_\xi} |u_0(x)|^2 e^{a|x|} \mathrm dx.
\end{eqnarray*}

\vskip 5pt
\textit{Step 1. To prove that there exists a  constant $C_1\triangleq C_1(n)>1$  so that
\begin{eqnarray}\label{0411-proof-th1.6-2}
 \sup_{1 \leq \eta \leq 2}
 \int_{\mathbb R^n} (1+|x|)^{-n-1-\eta}  |u(x,T;u_0)|^2 \,\mathrm dx
 &\leq&  C(x_0,r,a,T)  \tilde{g}\left(\frac{A_2}{B_2}\right) A_2,
\end{eqnarray}
where the constant $C(x_0,r,a,T)$ is given by
\begin{eqnarray}\label{0411-proof-th1.6-2-constant}
 C(x_0,r,a,T) \triangleq     e^{C_1^{1+\frac{|x_0|+r+1}{(aT)\wedge r}}},
\end{eqnarray}
and  the function $\tilde g$ is defined by
\begin{eqnarray}\label{0411-th1.6-3}
 \tilde g(\eta)\triangleq    \frac{1}{\ln (\ln \eta+ e)},~ \eta \geq 1
\end{eqnarray}}
\noindent
By the definitions of $A_2$ and $A_3$, we see that  $A_3\leq A_2$. Then by Theorem \ref{theorem4} (where $({x^\prime},{x^{\prime\prime}},r_1,r_2)$ is replaced by $(x_0,0,r,k)$) and the definitions of $A_2$ and $B_2$, we find that when $k\in\mathbb N^+$,
\begin{eqnarray*}
 \int_{B_k} |u(x,T;u_0)|^2 \,\mathrm dx
  &\leq& C k^n  \big((aT)\wedge r\big)^{-n}
  B_2^{\theta^{1+\frac{|x_0|+k+r}{(aT)\wedge r}}}
   A_3^{1-\theta^{1+\frac{|x_0|+k+r}{(aT)\wedge r}}}
  \nonumber\\
  &\leq& C k^n \big((aT)\wedge r\big)^{-n}
   B_2^{\theta^{1+\frac{|x_0|+k+r}{(aT)\wedge r}}}
  A_2^{1-\theta^{1+\frac{|x_0|+k+r}{(aT)\wedge r}}}
\end{eqnarray*}
for some $C>0$ and $\theta\in(0,1)$ depending only on $n$.
   The above inequality yields that for each  $\eta\in [1,2]$,
\begin{eqnarray}\label{0411-th1.6-3-2}
 & & \int_{\mathbb R^n} (1+|x|)^{-n-1-\eta}  |u(x,T;u_0)|^2 \,\mathrm dx
 \leq \sum_{k=1}^\infty \int_{k-1\leq |x|<k} k^{-n-1-\eta} |u(x,T;u_0)|^2 \,\mathrm dx
 \nonumber\\
 &\leq& C \big((aT)\wedge r\big)^{-n}  \left( \sum_{k=1}^\infty  k^{-1-\eta} \Big(\frac{B_2}{A_2}\Big)^{ \theta^{1+\frac{|x_0|+k+r}{(aT)\wedge r}} }  \right) A_2.
\end{eqnarray}
 Since $u_0\neq 0$, by the definitions of $A_2$ and $B_2$, and by the conversation law  for the Schr\"{o}dinger equation, we obtain that $B_2<A_2$. Then by  (ii) of Lemma \ref{0407-lemma-int}, where
$$
(x,\theta,\varepsilon,\alpha)
= \Big( \Big({B_2}/{A_2}\Big)^{\theta^{1+\frac{|x_0|+r}{(aT)\wedge r}}},
\theta^{\frac{1}{(aT)\wedge r}},\eta,\theta^{-1-\frac{|x_0|+r}{(aT)\wedge r}} \Big),
$$
we see that for each  $\eta\in [1,2]$,
\begin{eqnarray}\label{tianJin3.31}
\sum_{k=1}^\infty  k^{-1-\eta} \Big(\frac{B_2}{A_2}\Big)^{ \theta^{1+\frac{|x_0|+k+r}{(aT)\wedge r}} }
\leq   \frac{4}{\eta}   \theta^{-\eta-\eta\frac{|x_0|+r}{(aT)\wedge r}}
 e^{\eta \ln\eta +\eta + e\theta^{1+\frac{|x_0|+r-1}{(aT)\wedge r}}}
 \frac{1}{\big(\ln(|\ln \frac{B_2}{A_2}|+e)\big)^\eta}.
\end{eqnarray}
Therefore, we have that
\begin{eqnarray}\label{0514-th1.6-step1-1}
 & & \int_{\mathbb R^n} (1+|x|)^{-n-1-\eta}  |u(x,T;u_0)|^2 \,\mathrm dx
 \nonumber\\
 &\leq&   \frac{4C}{\eta}  \big((aT)\wedge r\big)^{-n}
  \theta^{-\eta-\eta\frac{|x_0|+r}{(aT)\wedge r}} e^{\eta \ln\eta +\eta + e\theta^{1+\frac{|x_0|+r-1}{(aT)\wedge r}}}
 \frac{A_2}{\big(\ln(|\ln \frac{B_2}{A_2}|+e)\big)^\eta}
 \nonumber\\
 &\leq&  4C ((aT)\wedge r)^{-n}
  \theta^{-2-2\frac{|x_0|+r}{(aT)\wedge r}} e^{2\ln 2 +2 +  e\theta^{-\frac{1}{(aT)\wedge r}}}
 \frac{A_2}{\ln(|\ln \frac{B_2}{A_2}|+e)}
 \nonumber\\
 &\leq& 4C n! e^{\frac{1}{(aT)\wedge r}} e^{\theta^{-2-2\frac{|x_0|+r}{(aT)\wedge r}}}  e^{2\ln 2 +2 +  e\theta^{-\frac{1}{(aT)\wedge r}}}
 \frac{A_2}{\ln(|\ln \frac{B_2}{A_2}|+e)}
 \nonumber\\
 &\leq& 4C n! e^{2\ln 2 +2}  e^{(\theta^{-2}+e+1)\theta^{-2\frac{|x_0|+r+1}{(aT)\wedge r}} }
  \frac{A_2}{\ln(|\ln \frac{B_2}{A_2}|+e)}.
\end{eqnarray}
(In  the first inequality of (\ref{0514-th1.6-step1-1}), we used (\ref{0411-th1.6-3-2}) and (\ref{tianJin3.31}); In  the last three inequalities of (\ref{0514-th1.6-step1-1}),
we used the facts  that
$$
\theta\in(0,1)
\;\;\mbox{and}\;\;
((aT)\wedge r)^{-n}\leq n! e^{\frac{1}{(aT)\wedge r}}
 \leq n! e^{ \theta^{-2 \frac{1}{(aT)\wedge r}} }.)
$$
 Since $\theta\in(0,1)$, (\ref{0411-proof-th1.6-2}) follows from (\ref{0514-th1.6-step1-1}), as well as  (\ref{0411-proof-th1.6-2-constant}) and (\ref{0411-th1.6-3}). This ends the proof of Step 1.

\vskip 5pt
  \textit{Step 2. To show that there exists  $C_2\triangleq C_2(n)>1$ so that
  \begin{eqnarray}\label{0411-th1.6-4}
    \int_{\mathbb R^n} |u_0(x)|^2 \,\mathrm dx
 \leq  C_3(x_0,r,a,T)
    \frac{A_2}{\sqrt{\ln (\ln \frac{A_2}{B_2}+ e)}},
  \end{eqnarray}
  where
  \begin{eqnarray}\label{0411-th1.6-c4}
   C_3(x_0,r,a,T) \triangleq  (1+T)^{2n+6}  e^{C_2^{1+\frac{|x_0|+r+1}{(aT)\wedge r}}}
  \end{eqnarray}}
  \noindent
  Choose $\eta_0\in \{1,2\}$ so that
  \begin{eqnarray*}\label{0411-th1.6-5}
   n+1+\eta_0 =0 ~~  (\mod 2).
  \end{eqnarray*}
  By Lemma \ref{lemma-0411-th1.6} (where $k=\frac{n+1+\eta_0}{2}$), it follows that
  \begin{eqnarray*}
   \int_{\mathbb R^n} |x|^{n+1+\eta_0} |u(x,T;u_0)|^2 \,\mathrm dx
   \leq C_{31} (1+T)^{n+1+\eta_0}
   \Big(   \|u_0\|^2_{H^{n+1+\eta_0}(\mathbb R^n;\mathbb C)}
      + \int_{\mathbb R^n} |x|^{2(n+1+\eta_0)} | u_0(x)|^2 \,\mathrm dx
  \Big)
  \end{eqnarray*}
  for some $C_{31}>0$ depending only on $n$. The above inequality yields that
  \begin{eqnarray}\label{tianjin3.36}
   & & \int_{\mathbb R^n} (1+|x|)^{n+1+\eta_0} |u(x,T;u_0)|^2 \,\mathrm dx
   \leq  \int_{\mathbb R^n} 2^{n+1+\eta_0} (1+|x|^{n+1+\eta_0}) |u(x,T;u_0)|^2 \,\mathrm dx
  \\
  &\leq& C_{32} (1+T)^{n+1+\eta_0}
  \Big(
  \int_{\mathbb R^n}  |u(x,T;u_0)|^2 \,\mathrm dx
   + \|u_0\|^2_{H^{n+1+\eta_0}(\mathbb R^n;\mathbb C)}
      + \int_{\mathbb R^n} |x|^{2(n+1+\eta_0)} | u_0(x)|^2 \,\mathrm dx
  \Big) \nonumber
  \end{eqnarray}
  for some $C_{32}>0$ depending only on $n$.
  Since
  $$
   (a|x|)^{2(n+1+\eta_0)} \leq [2(n+1+\eta_0)]!  e^{a|x|},~x\in \mathbb R^n,
  $$
 and because
  \begin{eqnarray*}
     \max\{1,a^{-2(n+1+\eta_0)}\}
    &=&\max\{1,(aT)^{-2(n+1+\eta_0)} T^{2(n+1+\eta_0)}\}
    \nonumber\\
  &\leq& (1+T)^{2(n+1+\eta_0)}  \max\{1,(aT)\wedge r)^{-2(n+1+\eta_0)}\}
  \nonumber\\
  &\leq&  (1+T)^{3(n+3)}  \big( 1+((aT)\wedge r)^{-1} \big)^{2(n+3)},
  \end{eqnarray*}
 we obtain from  (\ref{tianjin3.36}) and the definition of $A_2$
   that
  \begin{eqnarray}\label{0514-th1.6-step2-1}
  & & \int_{\mathbb R^n} (1+|x|)^{n+1+\eta_0} |u(x,T;u_0)|^2 \,\mathrm dx
  \nonumber\\
  &\leq& C_{33} (1+T)^{n+1+\eta_0}  \Big( \|u_0\|^2_{H^{n+3}(\mathbb R^n;\mathbb C)}
      + \int_{\mathbb R^n} a^{-2(n+1+\eta_0)}   e^{a|x|} | u_0(x)|^2 \,\mathrm dx
  \Big)
  \nonumber\\
  &\leq& C_{33} (1+T)^{n+1+\eta_0}  \max\{1,a^{-2(n+1+\eta_0)}\}  A_2
  \nonumber\\
  %&\leq& C_{33}(1+T)^{3(n+1+\eta_0)}  \max\{1,(aT)^{-2(n+1+\eta_0)}\}  A_2
%  \nonumber\\
  &\leq& C_{33} (1+T)^{4(n+3)}  \big( 1+((aT)\wedge r)^{-1} \big)^{2(n+3)}  A_2
  \end{eqnarray}
  for some $C_{33}>0$ depending only on $n$.

  Now, by the conversation law for the Schr\"{o}dinger equation, (\ref{0514-th1.6-step2-1}) and (\ref{0411-proof-th1.6-2}), we find  that
  \begin{eqnarray}\label{0411-th1.6-5-1}
  & &  \int_{\mathbb R^n} |u_0(x)|^2 \,\mathrm dx
   = \int_{\mathbb R^n} |u(x,T;u_0)|^2 \,\mathrm dx
    \\
   &\leq& \left( \int_{\mathbb R^n} (1+|x|)^{n+1+\eta_0} |u(x,T;u_0)|^2 \,\mathrm dx \right)^{\frac{1}{2}}
    \left( \int_{\mathbb R^n} (1+|x|)^{-n-1-\eta_0} |u(x,T;u_0)|^2 \,\mathrm dx \right)^{\frac{1}{2}}
    \nonumber\\
    %&\leq&  \sqrt{C_4(n) [2(n+3)]! } (1+T)^{n+3} \max\{1,a^{-(n+3)}\} \sqrt{C(x_0,r,a,T)}  \frac{A_2}{\sqrt{\ln (\ln \frac{A_2}{B_2}+ e)}}
%    \nonumber\\
    &\leq& \sqrt{C_{33}} (1+T)^{2n+6} (1+((aT)\wedge r)^{-1})^{n+3} \sqrt{C(x_0,r,a,T)}  \frac{A_2}{\sqrt{\ln (\ln \frac{A_2}{B_2}+ e)}}
    \nonumber\\
    &\leq& \sqrt{C_{33}} (1+T)^{2n+6} (n+3)! e^{1+((aT)\wedge r)^{-1}} \sqrt{C(x_0,r,a,T)}  \frac{A_2}{\sqrt{\ln (\ln \frac{A_2}{B_2}+ e)}}.
    \nonumber
  \end{eqnarray}
  (Notice that in the last inequality in (\ref{0411-th1.6-5-1}), we used that
  $ x^{n+3}\leq (n+3)! e^x \;\;\mbox{for all}\;\; x>0$.)
  Now, (\ref{0411-th1.6-4}) follows from (\ref{0411-th1.6-5-1}) and (\ref{0411-proof-th1.6-2-constant}) at once. This ends the proof of Step 2.

\vskip 5pt
\textit{Step 3. To show (\ref{0405-sch-th5-control-2}) for the above-mentioned $u_0$ and each $\varepsilon\in(0,1)$}

\noindent
  It suffices to show that for each $\varepsilon\in(0,1)$,
\begin{eqnarray}\label{0411-th1.6-6}
  S\triangleq \int_{\mathbb R^n} |u_0(x)|^2 \,\mathrm dx
  \leq  C_3   \left(\varepsilon A_2 + \varepsilon  e^{e^{\varepsilon^{-2}}} B_2 \right),
\end{eqnarray}
where $C_3 \triangleq C_3(x_0,r,a,T)$ is given by (\ref{0411-th1.6-c4}).
In fact,  for an arbitrarily fixed $\varepsilon>0$, there are only two possibilities: either $S\leq C_3 \varepsilon A_2$ or $S>C_3 \varepsilon A_2$.
In the first case, (\ref{0411-th1.6-6}) is obvious. In the second case, since $C_3>1$ (see (\ref{0411-th1.6-c4})), it follows from the definitions of $S$ and $A_2$  that
\begin{eqnarray}\label{0411-th1.6-6-1}
 0<\varepsilon<\frac{S}{C_3 A_2}<1.
\end{eqnarray}
 Since  the function: $x\mapsto x e^{e^{x^{-2}}}$, with its domain $(0,1)$,
is decreasing,   we see from (\ref{0411-th1.6-6-1}) that
\begin{eqnarray}\label{0411-th1.6-8}
\frac{S}{C_3 A_2} e^{e^{(\frac{S}{C_3 A_2})^{-2}}}
  \leq    \varepsilon e^{e^{\varepsilon^{-2}}}.
\end{eqnarray}

Meanwhile, since the function $x\mapsto e^{-e} e^{e^{x^{-2}}}$, with its domain $(0,1)$,
 is decreasing and because the inverse of the aforementioned function is the function: $x\mapsto \frac{1}{\sqrt{\ln (\ln x+ e)}}$, with its domain $(1, \infty)$, we get from (\ref{0411-th1.6-4})  that
\begin{eqnarray}\label{0411-th1.6-7}
 \frac{A_2}{B_2} \leq  e^{-e} e^{e^{(\frac{S}{C_3 A_2})^{-2}}}.
\end{eqnarray}

Now, it follows from (\ref{0411-th1.6-7}) and (\ref{0411-th1.6-8})  that
\begin{eqnarray*}
 S &=& C_3 \frac{S}{C_3 A_2} \frac{A_2}{B_2} B_2
 \leq C_3 \Big[ \frac{S}{C_3 A_2} e^{-e} e^{e^{(\frac{S}{C_3 A_2})^{-2}}} \Big]  B_2
 \nonumber\\
 &\leq& C_3 \varepsilon e^{-e} e^{e^{\varepsilon^{-2}}} B_2
 \leq C_3 \varepsilon  e^{e^{\varepsilon^{-2}}} B_2.
\end{eqnarray*}
Because $\varepsilon$ was arbitrarily taken from $(0,1)$,
the above leads to (\ref{0411-th1.6-6}). This ends the proof of (\ref{0405-sch-th5-control-2}).

\vskip 5pt

In summary, we  complete the proof of this theorem.

\end{proof}

\bigskip
\section{Further comments on the main results}

The purpose of this section is to present the next Theorem~\ref{theorem-further-study}. From it, we can see that
the inequalities in Theorem~\ref{theorem1} and  Theorem \ref{theorem3} cannot be improved greatly (see Remark~\ref{tianjinremark4.2}). For instance,
in the inequality
(\ref{0229-sch-th1-0}) in
Theorem~\ref{theorem1},  $(B_{r_1}^c({x^\prime}), B_{r_2}^c({x^{\prime\prime}}))$ cannot be replaced by  $(B_{r_1}^c({x^\prime}), B_{r_2}({x^{\prime\prime}}))$.

%These statements will answer why Theorem \ref{theorem1} can not hold in the case that either one of the observable region is a ball or the observation was put at  one time point, and why Theorem \ref{theorem3} is not true in the case that the observable region is a ball. The aforementioned statements are stated  in the following theorem.

\begin{theorem}\label{theorem-further-study}
The following conclusions are true:

\noindent(i) Let ${x^\prime},\,{x^{\prime\prime}}\in\mathbb R^n$, $r_1,\,r_2>0$ and $T>0$. Then there exists a sequence  $\{u_k\}\subset L^2(\mathbb R^n;\mathbb C)$, with
\begin{eqnarray}\label{0514-further-study-i-1}
 \int_{\mathbb R^n}  |u_k(x)|^2 \,\mathrm dx=1\;\;\mbox{for all}\;\; k\in\mathbb N^+,
\end{eqnarray}
so  that
\begin{eqnarray}\label{0514-further-study-i-2}
\lim_{k\rightarrow\infty} \int_{B_{r_1}^c({x^\prime})}  |u_k(x)|^2 \,\mathrm dx
=
 \lim_{k\rightarrow\infty}  \int_{B_{r_2}({x^{\prime\prime}})}  |u(x,T;u_k)|^2 \,\mathrm dx= 0.
\end{eqnarray}

\noindent (ii) Let ${x^\prime},\,{x^{\prime\prime}}\in\mathbb R^n$, $r_1,\,r_2>0$, $S_1>0$ and $S_2>0$. Then there exists a sequence $\{u_k\}\subset L^2(\mathbb R^n;\mathbb C)$, with
\begin{eqnarray}\label{further-study-ii-1}
 \int_{\mathbb R^n}  |u_k(x)|^2 \,\mathrm dx=1\;\;\mbox{for all}\;\; k\in\mathbb N^+,
\end{eqnarray}
so that
\begin{eqnarray}\label{further-study-ii-2}
\lim_{k\rightarrow\infty} \int_{B_{r_1}^c({x^\prime})}  |u(x,S_1;u_k)|^2 \,\mathrm dx
=
 \lim_{k\rightarrow\infty}  \int_0^{S_2} \int_{B_{r_2}({x^{\prime\prime}})}  |u(x,t;u_k)|^2 \,\mathrm dx \mathrm dt= 0.
\end{eqnarray}

\noindent (iii) For each  subset $A\subset\mathbb R^n$, with  $m(A^c)>0$,  and each $T>0$, there does not exist a positive constant $C>0$ so that
\begin{eqnarray}\label{further-study-i-0}
 \int_{\mathbb R^n}  |u_0(x)|^2 \,\mathrm dx
 \leq C \int_{A}  |u(x,T;u_0)|^2 \,\mathrm dx
\end{eqnarray}
for all  $u_0\in L^2(\mathbb R^n;\mathbb C)$.

\noindent (iv) For each $x_0\in\mathbb R^n$, $r>0$, $a>0$ and $T>0$, there exists a sequence of $\{u_k\}\subset C_0^\infty(\mathbb R^n;\mathbb C)$ and $M>0$ so that
\begin{eqnarray}\label{0502-further-study-ii-1}
 \int_{\mathbb R^n} e^{a|x|} |u_k(x)|^2 \,\mathrm dx \leq M
 \;\;\mbox{and}\;\;
 \int_{\mathbb R^n}  |u_k(x)|^2 \,\mathrm dx=1\;\;\mbox{for all}\;\;k\in\mathbb N^+
\end{eqnarray}
and so that
\begin{eqnarray}\label{0502-further-study-ii-2}
  \lim_{k\rightarrow\infty} \int_{B_{r}(x_0)}  |u(x,T;u_k)|^2 \,\mathrm dx=0.
\end{eqnarray}

\end{theorem}

\begin{proof}
For each    $\tau\in \mathbb R\setminus\{0\}$ and  $f\in L^2(\mathbb R^n;\mathbb C)$, we define a function $u_{\tau,f}$ by
 \begin{eqnarray}\label{further-study-3}
  u_{\tau,f}(x)\triangleq   e^{-i|x|^2/4\tau} f(x),~x\in \mathbb R^n.
 \end{eqnarray}
 By
 %(\ref{0229-sch-th1-4})
  \cite[(1.2)]{EKPV-2}
   and (\ref{further-study-3}), we see  that for all    $\tau\in \mathbb R\setminus\{0\}$ and  $f\in L^2(\mathbb R^n;\mathbb C)$,
  \begin{eqnarray*}
  (2i\tau)^{n/2} e^{-i|x|^2/4\tau}
  u(x,\tau;u_{\tau,f})
    =  \widehat{e^{i|\xi|^2/4\tau}u_{\tau,f}(\xi)} (x/2\tau)=\hat f(x/2\tau),~x\in\mathbb R^n.
 \end{eqnarray*}
 (Here and in what follows, $u(x,\tau;u_{\tau,f})=(e^{i\Delta \tau} u_{\tau,f})(x)$ when $\tau<0$.)
%Thanks to (\ref{0229-two-points-4}),
Thus, one has that for all    $\tau\in \mathbb R\setminus\{0\}$ and  $f\in L^2(\mathbb R^n;\mathbb C)$,
 \begin{eqnarray}\label{further-study-5}
  u(x,\tau;u_{\tau,f})
  =  (2i\tau)^{-n/2} e^{i|x|^2/4\tau}  \hat f(x/2\tau),~x\in\mathbb R^n.
 \end{eqnarray}
 Now, we prove the conclusions (i)-(iv) one by one.

 \vskip 5pt
(i) Let ${x^\prime},\,{x^{\prime\prime}}\in\mathbb R^n$, $r_1,\,r_2>0$ and $T>0$. Let  $g$ be a function so that
\begin{eqnarray}\label{0514-last}
 g\in C_0^\infty(\mathbb R^n;\mathbb C)
 \;\;\mbox{and}\;\;  \|g\|_{L^2(\mathbb R^n;\mathbb C)}=1.
\end{eqnarray}
For each $k\in\mathbb N^+$, let
\begin{eqnarray}\label{0514-further-study-7}
 g_k(x)\triangleq k^{n/2} g(k(x-{x^\prime}))  ,~x\in\mathbb R^n.
\end{eqnarray}
 We define a  sequence of $\{u_k\}\subset L^2(\mathbb R^n;\mathbb C)$ as follows:
\begin{eqnarray}\label{0514-further-study-8}
 u_k(x) \triangleq  e^{-i|x|^2/4T}  g_k \big(x \big),~x\in\mathbb R^n,~k\in\mathbb N^+.
\end{eqnarray}
By  (\ref{further-study-3}) and   (\ref{0514-further-study-8}), we have that
$$
u_{T, g_k}=u_k\;\;\mbox{for all}\;\; k\in \mathbb{N}^+.
$$
From this, (\ref{further-study-5}) and (\ref{0514-further-study-7}), after some computations, we see  that for each $k\in\mathbb N^+$,
\begin{eqnarray}\label{0514-further-study-8-1}
 u(x,T;u_k)=  (2iT)^{-n/2} e^{i|x|^2/4T}  k^{-n/2}
 \hat g(\frac{x}{2Tk})
   e^{-i x\cdot {x^\prime}/2T},
   ~x\in\mathbb R^n.
\end{eqnarray}

Three observations are given in order: First,  by (\ref{0514-further-study-8}) and (\ref{0514-further-study-7}), we find that
\begin{eqnarray*}
  \lim_{k\rightarrow\infty}\int_{B_{r_1}^c({x^\prime})}  |u_k(x)|^2 \,\mathrm dx
  &=& \lim_{k\rightarrow\infty}\int_{B_{kr_1}^c(0)}   |g(x)|^2 \,\mathrm dx=0;
\end{eqnarray*}
Second,  from (\ref{0514-further-study-8}), (\ref{0514-further-study-7}) and (\ref{0514-last}), we see that
\begin{eqnarray*}
 \int_{\mathbb R^n_x}   |u_k(x)|^2 \,\mathrm dx
  &=& \int_{\mathbb R^n_x}   |g_k(x)|^2 \,\mathrm dx =1\;\;\mbox{for all}\;\;k\in\mathbb N^+;
\end{eqnarray*}
Third,  from (\ref{0514-further-study-8-1}) and \eqref{0514-last}, we obtain that
\begin{eqnarray*}
 \lim_{k\rightarrow\infty}\int_{B_{r_2}({x^{\prime\prime}})}  |u(x,T;u_k)|^2 \,\mathrm dx
 = \lim_{k\rightarrow\infty}\int_{B_{\frac{r_2}{2Tk} }(\frac{{x^{\prime\prime}}}{2Tk} )}   |\hat g(x)|^2 \,\mathrm dx=0.
\end{eqnarray*}
Now, from the above three observations, we get  (\ref{0514-further-study-i-1}) and (\ref{0514-further-study-i-2}). This ends the proof the conclusion (i).

 \vskip 5pt
(ii) Let ${x^\prime},\,{x^{\prime\prime}}\in\mathbb R^n$, $r_1,\,r_2>0$, $S_1>0$ and $S_2>0$. Let $g$ and $g_k$, with $k\in\mathbb N^+$,  satisfy (\ref{0514-last}) and (\ref{0514-further-study-7}),
respectively.
 Since the Schr\"{o}dinger equation is time-reversible,  we can find a  sequence  $\{u_k\}\subset L^2(\mathbb R^n;\mathbb C)$ so that
\begin{eqnarray}\label{0502-further-study-8}
 v_k(x) \triangleq u(x,S_1;u_k) = g_k (x),~x\in\mathbb R^n,~k\in\mathbb N^+.
\end{eqnarray}
By (\ref{0502-further-study-8}), (\ref{0514-last}) and (\ref{0514-further-study-7}),
 we find that
\begin{eqnarray}\label{0502-further-study-9}
  \lim_{k\rightarrow\infty}\int_{B_{r_1}^c({x^\prime})}  |v_k(x)|^2 \,\mathrm dx
  &=& \lim_{k\rightarrow\infty}\int_{B_{kr_1}^c(0)}   |g(x)|^2 \,\mathrm dx=0
\end{eqnarray}
and
\begin{eqnarray}\label{0502-further-study-9-1}
 \int_{\mathbb R^n_x}   |v_k(x)|^2 \,\mathrm dx
  &=& \int_{\mathbb R^n_x}   |g_k(x)|^2 \,\mathrm dx =1\;\;\mbox{for all}\;\;k\in\mathbb N^+.
\end{eqnarray}

Next, by (\ref{0502-further-study-8}) and (\ref{further-study-3}), we have that
\begin{eqnarray*}
 v_k = u_{\tau,f}
 \;\;\mbox{with}\;\;
 (\tau,f)=(t,e^{i|\cdot|^2/4t}g_k(\cdot)).
\end{eqnarray*}
Then by  (\ref{further-study-5}),  we get that for each $k\in\mathbb N^+$,
\begin{eqnarray}\label{0502-further-study-8-1}
 u(x,t;v_k)=  (2it)^{-n/2} e^{i|x|^2/4t}
 \widehat{e^{i|\xi|^2/4t}g_k(\xi)} (x/2t),
   ~(x,t)\in\mathbb R^n \times (\mathbb R\setminus\{0\}).
\end{eqnarray}
Meanwhile, from (\ref{0514-further-study-7}), it follows that for all $t\in\mathbb R\setminus\{0\}$ and a.e. $x\in\mathbb R^n$,
\begin{eqnarray*}
  \widehat{e^{i|\xi|^2/4t}g_k(\xi)} (x)
  &=&(2\pi)^{-n/2} \int_{\mathbb R^n_\xi} e^{-i x\cdot\xi} e^{i|\xi|^2/4t}g_k(\xi) \,\mathrm d\xi
  \nonumber\\
  &=& (2\pi)^{-n/2} \int_{\mathbb R^n_\xi} e^{-i x\cdot\xi} e^{i|\xi|^2/4t} k^{n/2} g(k(\xi-{x^\prime})) \,\mathrm d\xi
  \nonumber\\
  &=& (2\pi)^{-n/2} k^{-n/2} e^{-i x\cdot {x^\prime}} \int_{\mathbb R^n_\xi} e^{-i x\cdot\xi/k} e^{i|\xi/k + {x^\prime}|^2/4t}  g(\xi) \,\mathrm d\xi.
\end{eqnarray*}
This, along with (\ref{0502-further-study-8-1}) and (\ref{0514-last}), yields that for each $t\in \mathbb R\setminus\{0\}$,
\begin{eqnarray*}
  \int_{B_{r_2}({x^{\prime\prime}})}  |u(x,t;v_k)|^2 \,\mathrm dx
  &\leq& |B_{r_2}({x^{\prime\prime}})| \sup_{x\in B_{r_2}({x^{\prime\prime}})} |u(x,t;v_k)|^2
  \nonumber\\
  &\leq&  |B_{r_2}({x^{\prime\prime}})| \Big(  (4\pi|t|k)^{-n/2} \int_{\mathbb R^n_\xi} |g(\xi)| \,\mathrm d\xi \Big)^2,
\end{eqnarray*}
which implies that
\begin{eqnarray}\label{0514-point-1}
  \lim_{k\rightarrow\infty} \int_{B_{r_2}({x^{\prime\prime}})}  |u(x,t;v_k)|^2 \,\mathrm dx=0
  \;\;\mbox{for each}\;\;  t\in\mathbb R \setminus\{0\}.
\end{eqnarray}
At the same time, by the conservation law for the Schr\"{o}dinger equation and (\ref{0502-further-study-9-1}), we find that for all $k$ and $t\in \mathbb R\setminus\{0\}$,
\begin{eqnarray*}
  \int_{B_{r_2}({x^{\prime\prime}})}  |u(x,t;v_k)|^2 \,\mathrm dx
  \leq \int_{\mathbb R^n}  |u(x,t;v_k)|^2 \,\mathrm dx
  =  \int_{\mathbb R^n}  |v_k(x)|^2 \,\mathrm dx=1.
\end{eqnarray*}
By this and (\ref{0514-point-1}), we can apply  the Lebesgue dominated convergence theorem to get that
\begin{eqnarray}\label{0514-point-2}
 \lim_{k\rightarrow\infty} \int_{-S_1}^{S_2-S_1} \int_{B_{r_2}({x^{\prime\prime}})}  |u(x,t;v_k)|^2 \,\mathrm dx \mathrm dt =0.
\end{eqnarray}
Since $v_k(x)=u(x,S_1;u_k)$, $x\in\mathbb R^n$ (see (\ref{0502-further-study-8})), by (\ref{0502-further-study-9}), (\ref{0502-further-study-9-1}) and (\ref{0514-point-2}), one can directly check that
 the above-mentioned sequence $\{u_k\}$ satisfies (\ref{further-study-ii-1}) and (\ref{further-study-ii-2}).
 This ends the proof of the conclusion (ii).

 \vskip 5pt
 (iii) By contradiction, suppose that the conclusion (iii) in this theorem  was not true. Then there would exist
 $A_0\subset\mathbb R^n$, with $m(A_0^c)>0$, $C_1>0$ and $T>0$ so that
\begin{eqnarray}\label{further-study-6}
 \int_{\mathbb R^n}  |u_0(x)|^2 \,\mathrm dx
 \leq C_1 \int_{A_0}  |u(x,T;u_0)|^2 \,\mathrm dx\;\;\mbox{for all}\;\;u_0\in L^2(\mathbb R^n;\mathbb C).
\end{eqnarray}
From  (\ref{further-study-3}), (\ref{further-study-6}) and (\ref{further-study-5}), we find  that for each $f\in L^2(\mathbb R^n;\mathbb C)$,
\begin{eqnarray*}
  \int_{\mathbb R^n_{\xi}}  |\hat f(\xi)|^2 \,\mathrm d\xi
=\int_{\mathbb R^n_x}  |f(x)|^2 \,\mathrm dx
=  \int_{\mathbb R^n}  |u_{T,f}(x)|^2 \,\mathrm dx
 \leq  C_1 \int_{A_0}  |u(x,T;u_{T,f})|^2 \,\mathrm dx
= C_1 \int_{A_0/2T}  |\hat f(\xi)|^2 \,\mathrm d\xi .
\end{eqnarray*}
Since $|A^c_0|>0$, by taking $f\in L^2(\mathbb R^n;\mathbb C)\setminus\{0\}$ with supp\,$\hat f\subset A_0^c/2T$ in the above inequality, we are led to a contradiction.
Hence,  the conclusion (iii) in this theorem is true.

\vskip 5pt
(iv) Arbitrarily  fix $x_0\in\mathbb R^n$, $r>0$, $a>0$ and $T>0$. Let $g\in C_0^\infty(\mathbb R^n;\mathbb C)$ be a function so that
\begin{eqnarray}\label{further-study-7}
 \int_{\mathbb R^n_\xi}  |\hat g(\xi)|^2 \,\mathrm d\xi
 = \int_{\mathbb R^n_x}  |g(x)|^2 \,\mathrm dx
 = 1.
\end{eqnarray}
Let  $\vec v\in S^{n-1}$. We define a  sequence  $\{u_k\}\subset C_0^\infty(\mathbb R^n;\mathbb C)$ by
\begin{eqnarray}\label{further-study-8}
u_k(x) \triangleq  e^{-i  |x|^2/4T} e^{-kix\cdot \vec{v}} g (x),~x\in\mathbb R^n.
\end{eqnarray}
By (\ref{further-study-8})  and  (\ref{further-study-3}), we have that
\begin{eqnarray*}
 u_k=u_{\tau,f},
 \;\;\mbox{with}\;\;
 \tau=T\;\;\mbox{and}\;\;f(x)=e^{-kix\cdot \vec{v}} g (x),\;\;x\in\mathbb R^n,
\end{eqnarray*}
from which and (\ref{further-study-5}), it follows that for each $k\in\mathbb N^+$,
\begin{eqnarray*}
 u(x,T;u_k) =  (2iT)^{-n/2} e^{i|x|^2/4T}  \hat g \big(\frac{x+k\vec v}{2T} \big),~x\in\mathbb R^n,~k\in\mathbb N^+.
\end{eqnarray*}
This yields that for each $k\in\mathbb N^+$,
\begin{eqnarray*}
 \int_{B_r(x_0)}  |u(x,T;u_k)|^2 \,\mathrm dx
 = \int_{B_{\frac{r}{2T}}(\frac{x_0+k\vec v}{2T})}  |\hat g(x)|^2 \,\mathrm dx.
\end{eqnarray*}
Since $\int_{\mathbb R^n}  |\hat g(x)|^2 \,\mathrm dx<\infty$ (see (\ref{further-study-7})), the above implies that
\begin{eqnarray}\label{TTJJin2.23}
  \int_{B_r(x_0)}  |u(x,T;u_k)|^2 \,\mathrm dx \rightarrow 0
  \;\;\mbox{as}\;\;  k\rightarrow\infty.
\end{eqnarray}

Meanwhile, from (\ref{further-study-8}) and (\ref{further-study-7}), we find that for each $k\in\mathbb N^+$,
\begin{eqnarray*}
  \int_{\mathbb R^n_x} e^{a|x|}  |u_k(x)|^2 \,\mathrm dx
  &=& \int_{\mathbb R^n_x} e^{a|x|}  |g(x)|^2 \,\mathrm dx<\infty
\end{eqnarray*}
and
\begin{eqnarray*}
 \int_{\mathbb R^n_x}   |u_k(x)|^2 \,\mathrm dx
  &=& \int_{\mathbb R^n_x}   |g(x)|^2 \,\mathrm dx =1 .
\end{eqnarray*}
From these and (\ref{TTJJin2.23}), we obtain  (\ref{0502-further-study-ii-1}) and (\ref{0502-further-study-ii-2}). This ends the proof the conclusion (iv).

\vskip 5pt
In summary, we finish the proof of this theorem.

\end{proof}

\begin{remark}\label{tianjinremark4.2}
$(a)$ From (i)  and (ii) of Theorem~\ref{theorem-further-study}, one can easily check that
for any ${x^\prime},\,{x^{\prime\prime}}\in\mathbb R^n$, $r_1,\,r_2>0$ and $T>S\geq0$, there is no  constant $C>0$ so that
any of the following inequalities holds:
$$
 \int_{\mathbb R^n}  |u_0(x)|^2 \,\mathrm dx
 \leq C \Big(\int_{B_{r_1}^c({x^\prime})}  |u(x,S;u_0)|^2 \,\mathrm dx
 + \int_{B_{r_2}({x^{\prime\prime}})}  |u(x,T;u_0)|^2 \,\mathrm dx\Big),\;\forall\;u_0\in L^2(\mathbb R^n;\mathbb C);
$$
$$
 \int_{\mathbb R^n}  |u_0(x)|^2 \,\mathrm dx
 \leq C \Big(\int_{B_{r_1}^c({x^\prime})}  |u(x,S;u_0)|^2 \,\mathrm dx
 + \int_0^T\int_{B_{r_2}({x^{\prime\prime}})}  |u(x,t;u_0)|^2 \,\mathrm dx \mathrm dt\Big),\;\forall\;u_0\in L^2(\mathbb R^n;\mathbb C).
$$
Hence, the terms on the right hand side of
(\ref{0229-sch-th1-0}) in
Theorem~\ref{theorem1} cannot be replaced by either
$$
 C\Big(\int_{B_{r_1}^c({x^\prime})}  |u(x,S;u_0)|^2 \,\mathrm dx
 + \int_{B_{r_2}({x^{\prime\prime}})}  |u(x,T;u_0)|^2 \,\mathrm dx\Big)
$$
or
$$
 C\Big(\int_{B_{r_1}^c({x^\prime})}  |u(x,S;u_0)|^2 \,\mathrm dx
 + \int_0^T\int_{B_{r_2}({x^{\prime\prime}})}  |u(x,t;u_0)|^2 \,\mathrm dx \mathrm dt\Big).
$$

$(b)$ From (iii)  of Theorem~\ref{theorem-further-study}, we see that in order to have (\ref{further-study-i-0}) (the observability at one point in time), it is necessary that $|A^c|=0$.
That is, in order to recover a solution by observing it at one point in time, we must observe it at one time point and over the whole $\mathbb{R}^n$.
From this, conclusions in $(a)$ of this remark and  Theorem~\ref{theorem1}, we see that the observability at two points in time is ``optimal".

$(c)$  From (iv)  of Theorem~\ref{theorem-further-study}, we find that for any $r>0$, $a>0$ and $T>0$, there is no $C>0$ or $\theta\in(0,1)$ so that
\begin{eqnarray*}
\int_{\mathbb R^n}|u_0(x)|^2 \,\mathrm dx
 \leq C \left(
 \int_{B_r(0)}|u(x,T;u_0)|^2 \,\mathrm dx
 \right)^{\theta}
 \left(\int_{\mathbb R^n}e^{ a|x|} |u_0(x)|^2 \,\mathrm dx\right)^{1-\theta}
\end{eqnarray*}
for all $u_0\in C_0^\infty(\mathbb R^n;\mathbb C)$.
Hence,  the inequality in $(i)$ of Theorem \ref{theorem3} will not be true if   $B_r^c(0)$ is replaced by $B_r(0)$.

\end{remark}

\bigskip
\section{Applications}

In this section, we consider the applications of Theorems \ref{theorem1}-\ref{theorem6} to the controllability for the Schr\"{o}dinger equation.
These theorems correspond to different kinds of controllability with a cost.

\subsection{A functional analysis framework}

This subsection presents an equivalence lemma (Lemma~\ref{lemma-0428-fn}) between some observability and some controllability in an abstract framework.
With the aid of it, we can use inequalities in Theorems \ref{theorem1}-\ref{theorem6} to study some controllability  for the Schr\"{o}dinger equation.

\begin{lemma}\label{lemma-0428-fn}
Let $\mathbb K$ be either $\mathbb R$ or $\mathbb C$. Let $X$, $Y$ and $Z$ be three Banach spaces over  $\mathbb K$, with their dual spaces $X^*$, $Y^*$ and $Z^*$. Let $R\in \mathcal L(Z,X)$ and $O\in \mathcal L(Z,Y)$. Then the following two propositions are equivalent:

\noindent (i) There exists $\widehat C_0>0$ and $\hat\varepsilon_0>0$ so that for each $z\in Z$,
\begin{eqnarray}\label{lemma-0428-fn-ii}
 \| R z \|^2_X  \leq   \widehat C_0 \|Oz\|^2_Y
  +  \hat\varepsilon_0 \|z\|_Z^2.
\end{eqnarray}

\noindent (ii) There exists  $C_0>0$ and $\varepsilon_0>0$ so that for each $x^*\in X^*$, there is
$y^*\in Y^*$ satisfying that
\begin{eqnarray}\label{lemma-0428-fn-i}
 \frac{1}{C_0} \|y^*\|^2_{Y^*} + \frac{1}{\varepsilon_0} \|R^*x^*-O^*y^*\|^2_{Z^*}
 \leq \|x^*\|^2_{X^*}.
\end{eqnarray}

Furthermore, when one of the above two propositions holds, the constant pairs $(C_0,\varepsilon_0)$ and $(\widehat C_0,\hat\varepsilon_0)$ can be chosen to be the same.

\end{lemma}

\begin{remark}
The part (i) of Lemma~\ref{lemma-0428-fn-ii} presents a non-standard observability. In this part, $Z$ is a state space,
$Y$ is an observation space, we call $X$ as a state transformation  space of $Z$. Further, $O$ is an observation operator, while we call $R$ as a state transformation operator.
The inequality (\ref{lemma-0428-fn-ii}) means that we can approximately recover the transferred state $Rz$ by observing $Oz$, the error is governed by $\sqrt{\hat \varepsilon_0}\|z\|_{Z}$.

The part (ii) of Lemma~\ref{lemma-0428-fn-ii} presents a non-standard controllability. In this part, $Y^*$ is a control space, $X^*$ is a state space,  and we call $Z^*$ as a
state transformation space of $X^*$. Furthermore, $O^*$ is a control operator, while we call $R^*$ as a state transformation operator.
The inequality (\ref{lemma-0428-fn-i}) can be understood as follows: For each state $x^*$, there is a control $y^*$ so that $O^*y^*$ is close to the target $R^*x^*$, with the distance less that $\sqrt{\varepsilon_0}\|x^*\|_{X^*}$. Moreover,
 the norm of  this control is governed by $\sqrt{C_0}\|x^*\|_{X^*}$.

\end{remark}

\begin{proof}[Proof of Lemma \ref{lemma-0428-fn}]
The proof is divided into the following several steps.

\vskip 5pt
\textit{Step 1. To show that (ii)$\Rightarrow$(i)}

\noindent Suppose that (ii) is true. Then, for each $x^*\in X^*$, there exists $y^*_{x^*}\in Y^*$ so that (\ref{lemma-0428-fn-i}), with $y^*=y^*_{x^*}$, is true. From this, it follows that for any  $x^*\in X^*$ and $z\in Z$,
\begin{eqnarray*}
 \langle Rz, x^* \rangle_{X,X^*}
 &=&  \langle z, R^*x^* \rangle_{Z,Z^*}
 =  \langle z, R^*x^* - O^* y^*_{x^*} \rangle_{Z,Z^*} + \langle z, O^* y^*_{x^*} \rangle_{Z,Z^*}
 \nonumber\\
 &=& \langle z, R^*x^* - O^* y^*_{x^*} \rangle_{Z,Z^*} + \langle O z,  y^*_{x^*} \rangle_{Y,Y^*}.
\end{eqnarray*}
By this and the Cauchy-Schwarz inequality, we deduce that for each  $x^*\in X^*$ and $z\in Z$,
\begin{eqnarray*}
 | \langle Rz, x^* \rangle_{X,X^*}  |
 &\leq& \big( \sqrt{C_0} \|z\|_Z \big)  \left( \frac{1}{\sqrt{C_0}} \| R^*x^* - O^* y^*_{x^*} \|_{Z^*} \right)
   + \big( \sqrt{\varepsilon_0} \| O z\|_Y \big)  \left(\frac{1}{\sqrt{\varepsilon_0}}\| y^*_{x^*} \|_{Y^*} \right)
 \nonumber\\
 &\leq& \big( C_0\|z\|_Z^2  +  \varepsilon_0\|O z\|_Y^2 \big)^{1/2}
 \left( \frac{1}{C_0}\| R^*x^* - O^* y^*_{x^*} \|_{Z^*}^2 + \frac{1}{\varepsilon_0}\| y^*_{x^*} \|_{Y^*}^2 \right)^{1/2}
 \nonumber\\
 &\leq& \big( C_0\|z\|_Z^2  +  \varepsilon_0\|O z\|_Y^2 \big)^{1/2}   \|x^*\|_{X^*}.
\end{eqnarray*}
Hence,  (\ref{lemma-0428-fn-ii}), with
$(\widehat C_0,\hat\varepsilon_0)$$=(C_0,\varepsilon_0)$, is true.

\vskip 5pt
\textit{Step 2. To show that (i)$\Rightarrow$(ii)}

\noindent Suppose that (i) is true.
Define a    subspace $E$ of $Y\times Z$ in the following manner:
\begin{eqnarray*}
 E\triangleq\left\{ \left(\sqrt{\widehat C_0} Oz,\sqrt{\hat\varepsilon_0} z \right) ~:~  z\in Z \right\}.
\end{eqnarray*}
The norm of $E$ is inherited form the following usual norm of $Y\times Z$:
\begin{eqnarray}\label{0426-5}
\|(f,g)\|_{Y\times Z} \triangleq
\left( \|f\|_{Y}^2 + \|g\|_Z^2 \right)^{1/2},\; (f,g)\in Y\times Z.
\end{eqnarray}
Arbitrarily fix   $x^*\in X^*$. Define  an operator $\mathcal T_{x^*}$ by
\begin{eqnarray}\label{0428-th1.2-control-4}
 \mathcal T_{x^*} ~:~   E &\rightarrow&  \mathbb K
 \nonumber\\
 \left(\sqrt{\widehat C_0} Oz,\sqrt{\hat\varepsilon_0} z \right) &\mapsto& \langle x^*, Rz \rangle_{X^*,X}.
\end{eqnarray}
By (\ref{lemma-0428-fn-ii}) and (\ref{0428-th1.2-control-4}), we can easily check that $\mathcal T_{x^*}$ is well defined and linear.
We now claim that
\begin{eqnarray}\label{0428-th1.2-control-5}
  \| \mathcal T_{x^*}\|_{\mathcal L(E,\mathbb K)}
 \leq   \|x^*\|_{X^*}.
\end{eqnarray}
Indeed, by the definition of $E$, we see that
  given $(f,g)\in E$, there is $z\in Z$ so that
$$
(f,g)= \left(\sqrt{\widehat C_0} Oz,\sqrt{\hat\varepsilon_0} z \right).
$$
Then by (\ref{0428-th1.2-control-4}), we find that
\begin{eqnarray*}
  |\mathcal T_{x^*} \big((f,g)\big)|
   = | \langle x^*, Rz \rangle_{X^*,X} |
  \leq \|x^*\|_{X^*} \|Rz\|_X.
\end{eqnarray*}
This, along with   (\ref{lemma-0428-fn-ii}), shows (\ref{0428-th1.2-control-5}).

Since $\mathcal T_{x^*}$ is a linear and bounded functional, we can apply  the Hahn-Banach extension theorem to find
 $\widetilde{ \mathcal T}_{x^*}$ in $(Y\times Z)^*$ so that
\begin{eqnarray}\label{0428-th1.2-control-7}
 \widetilde{ \mathcal T}_{x^*}\big((f,g)\big) = \mathcal T_{x^*}\big((f,g)\big)
 \;\;\mbox{for all}\;\;
 (f,g)\in E
\end{eqnarray}
and so that
\begin{eqnarray}\label{0428-th1.2-control-8}
 \|\widetilde{ \mathcal T}_{x^*}\|_{\mathcal L(Y\times Z,\mathbb K)}
 = \| \mathcal T_{x^*}\|_{\mathcal L(E,\mathbb K)}.
\end{eqnarray}
These, together with (\ref{0426-5}) and (\ref{0428-th1.2-control-5}), yield that
\begin{eqnarray*}
 |\widetilde{ \mathcal T}_{x^*}\big((f,0)\big)|
 &\leq&  \|x^*\|_{X^*} \|f\|_{Y}
 \;\;\mbox{for all}\;\;f\in Y,
 \nonumber\\
 |\widetilde{ \mathcal T}_{x^*}\big((0,g)\big)|
 &\leq&  \|x^*\|_{X^*} \|g\|_{Z}
 \;\;\mbox{for all}\;\;g\in Z.
\end{eqnarray*}
Thus,    there exists   $(y^*_{x^*},z^*_{x^*})\in Y^*\times Z^*$ so that
\begin{eqnarray*}
 \widetilde{ \mathcal T}_{x^*}\big((f,0)\big)
 &=&  \langle y^*_{x^*},f \rangle_{Y^*,Y}
 \;\;\mbox{for all}\;\;f\in Y,
 \nonumber\\
 \widetilde{ \mathcal T}_{x^*}\big((0,g)\big)
 &=&  \langle z^*_{x^*},g \rangle_{Z^*,Z}
\;\;\mbox{for all}\;\;g\in Z,
\end{eqnarray*}
from which,  it follows that
\begin{eqnarray}\label{0428-th1.2-control-9}
 \widetilde{ \mathcal T}_{x^*}\big( (f,g) \big)
 = \langle y^*_{x^*},f \rangle_{Y^*,Y}  +  \langle z^*_{x^*},g \rangle_{Z^*,Z}\;\;\mbox{for any}\;\;(f,g)\in Y\times Z.
\end{eqnarray}

 Two observations are given in order: The first one reads
  \begin{eqnarray}\label{0428-th1.2-control-11}
  \|y^*_{x^*}\|_{Y^*}^2 + \|z^*_{x^*}\|_{Z^*}^2  \leq \|x^*\|_{X^*}^2,
 \end{eqnarray}
 while the second one is as
 \begin{eqnarray}\label{tianjin5.10}
  R^*x^* - O^*(\sqrt{\widehat C_0}y^*_{x^*}) = \sqrt{\hat\varepsilon_0} z^*_{x^*}
  \;\;\mbox{in}\;\; Z^*.
\end{eqnarray}
When    (\ref{0428-th1.2-control-11}) and  (\ref{tianjin5.10})
are proved, the conclusion (ii)  (with $(C_0,\varepsilon_0)$$=(\widehat C_0,\hat\varepsilon_0)$) follows at once.

 To prove (\ref{0428-th1.2-control-11}), we see   from (\ref{0428-th1.2-control-9}), (\ref{0428-th1.2-control-8}) and (\ref{0426-5}) that
 for each $(f,g)\in Y\times Z$,
 \begin{eqnarray*}
  | \langle y^*_{x^*},f \rangle_{Y^*,Y}  +  \langle z^*_{x^*},g \rangle_{Z^*,Z} |
  \leq \|x^*\|_{X^*}  \left( \|f\|_{Y}^2 + \|g\|_Z^2 \right)^{1/2}.
 \end{eqnarray*}
Meanwhile, for each $\delta\in(0,1)$, we can choose $(f_\delta,g_\delta)\in Y\times Z$ so that
 \begin{eqnarray*}
  \langle y^*_{x^*},f_\delta \rangle_{Y^*,Y} &=& \|y^*_{x^*}\|^2_{Y^*}  + o_1(1), ~ \|f_\delta\|_Y=\|y^*_{x^*}\|_{Y^*},
\\
  \langle z^*_{x^*},g_\delta \rangle_{Z^*,Z} &=& \|z^*_{x^*}\|^2_{Z^*} + o_2(1), ~ \|g_\delta\|_Z=\|z^*_{x^*}\|_{Z^*},
  \end{eqnarray*}
 where $o_1(1)$ and $o_2(1)$ are so that
 $$
 \lim_{\delta\rightarrow0^+} o_1(1)=\lim_{\delta\rightarrow0^+} o_2(1)=0.
 $$
 From these,  it follows that
 \begin{eqnarray*}
  \|y^*_{x^*}\|_{Y^*}^2 + \|z^*_{x^*}\|_{Z^*}^2 - |o_1(1)| - |o_2(1)|  \leq \|x^*\|_{X^*} \left( \|y^*_{x^*}\|_{Y^*}^2 + \|z^*_{x^*}\|_{Z^*}^2 \right)^{1/2}.
 \end{eqnarray*}
 Sending $\delta\rightarrow 0^+$ in the above inequality leads to (\ref{0428-th1.2-control-11}).

To prove (\ref{tianjin5.10}), we find  from (\ref{0428-th1.2-control-4}), (\ref{0428-th1.2-control-7}) and (\ref{0428-th1.2-control-9}) that for all $z\in Z$,
\begin{eqnarray*}
 \langle x^*, Rz \rangle_{X^*,X}
 = \langle y^*_{x^*}, \sqrt{\widehat C_0} Oz \rangle_{Y^*,Y}
    + \langle z^*_{x^*}, \sqrt{\hat \varepsilon_0} z \rangle_{Z^*,Z},
\end{eqnarray*}
which yields that for all $z\in Z$,
\begin{eqnarray*}
 \langle R^*x^*, z \rangle_{Z^*,Z}
 = \langle O^*(\sqrt{\widehat C_0}y^*_{x^*}),  z \rangle_{Z^*,Z}
    + \langle \sqrt{\hat\varepsilon_0} z^*_{x^*},  z \rangle_{Z^*,Z}.
\end{eqnarray*}
This leads to (\ref{tianjin5.10}).

\vskip 5pt
\textit{Step 3. About the constant pairs $(C_0,\varepsilon_0)$ and $(\widehat C_0,\hat\varepsilon_0)$}

\noindent From the proofs in Step 1 and Step 2, we see that  when one of the propositions (i) and (ii) holds,
 $(C_0,\varepsilon_0)$ and $(\widehat C_0,\hat\varepsilon_0)$ can be chosen to be the same pair. This ends the proof of this lemma.

\end{proof}

We end this subsection with introducing the   following dual equation:
\begin{eqnarray}\label{0229-sch-adjoint-1}
\left\{\begin{array}{lll}
        i\partial_t \varphi(x,t) + \Delta \varphi(x,t) = 0, &(x,t)\in
         \mathbb R^n\times(0,T),\\
        \varphi(x,T)=z(x),    &x\in\mathbb R^n,
       \end{array}
\right.
\end{eqnarray}
where $T>0$ and  $z\in L^2(\mathbb{R}^n)$. Write $\varphi(\cdot,\cdot;T,z)$ for the solution to (\ref{0229-sch-adjoint-1}).
The equation (\ref{0229-sch-adjoint-1}) will play an important role in the studies of different controllability for  the Schr\"{o}dinger equation.

\subsection{Applications of Theorem \ref{theorem1}-Theorem~\ref{theorem4} to controllability}

First, we will use  Theorem~\ref{theorem1}, as well as  Lemma \ref{lemma-0428-fn}, to prove the exact controllability for the following impulse controlled Schr\"{o}dinger equation:
\begin{eqnarray}\label{0229-sch-control-1}
\left\{\begin{array}{lll}
        i\partial_t u(x,t) + \Delta u(x,t) = \delta_{\{t=\tau_1\}} \chi_{B_{r_1}^c({x^\prime})}(x) h_1(x) + \delta_{\{t=\tau_2\}} \chi_{B_{r_2}^c({x^{\prime\prime}})}(x) h_2(x), &(x,t)\in  \mathbb R^n\times (0,T),~\\
        u(0,x) =u_0(x),    &x\in\mathbb R^n,
       \end{array}
\right.
\end{eqnarray}
where ${x^\prime},\,{x^{\prime\prime}}\in\mathbb R^n$, $r_1,\,r_2>0$, $T$, $\tau_1$ and $\tau_2$ are three numbers with $0\leq\tau_1<\tau_2\leq T$,
 $u_0\in L^2(\mathbb{R}^n;\mathbb C)$, controls $h_1$ and $h_2$ are taken from the space $L^2(\mathbb R^n;\mathbb C)$.
  Write $u_1(\cdot,\cdot;u_0,h_1,h_2)$ for the solution to the equation (\ref{0229-sch-control-1}).

\begin{theorem}
Let ${x^\prime},\,{x^{\prime\prime}}\in\mathbb R^n$ and $r_1,\,r_2>0$. Let $T$, $\tau_1$ and $\tau_2$ be three numbers with $0\leq\tau_1<\tau_2\leq T$. Then  for each  $u_0\in L^2(\mathbb R^n;\mathbb C)$ and $u_T\in L^2(\mathbb R^n;\mathbb C)$, there is a pair of controls $(h_1,h_2)$
in  $L^2(\mathbb R^n;\mathbb C)\times L^2(\mathbb R^n;\mathbb C)$ so that
\begin{eqnarray}\label{0229-sch-th2-control-1}
 u_1(x,T;u_0,h_1,h_2)=u_T(x),~x\in \mathbb R^n
\end{eqnarray}
and so that
\begin{eqnarray}\label{0229-sch-th2-control-2}
\|h_1\|^2_{L^2(\mathbb R^n;\mathbb C)} + \|h_2\|^2_{L^2(\mathbb R^n;\mathbb C)}
\leq
%C\big(n,{x^\prime},{x^{\prime\prime}},r_1,r_2,\tau_1,\tau_2\big)
C e^{C r_1 r_2  \frac{1}{T-S}  }
\left\|u_T-e^{i\Delta T}u_0\right\|_{L^2(\mathbb R^n;\mathbb C)}^2,
\end{eqnarray}
where
%the constant
%$C\big(n,{x^\prime},{x^{\prime\prime}},r_1,r_2,\tau_1,\tau_2\big)>0$ is  given by (\ref{0313-constant-CnAB}).
the constant $C\triangleq C(n)$ is given by Theorem~\ref{theorem1}.

%Furthermore, when  $A^c$  (or $B^c$) is a bounded  open subset of $\mathbb R^n$, the exponent in (\ref{0229-sch-th2-control-2}) can be replaced by  $e^{C\min\big\{\frac{|A^c||B^c|}{(\tau_2-\tau_1)^n},\frac{|A^c|^{1/n}w(B^c)}{\tau_2-\tau_1}\big\}}$ (or $e^{C\min\big\{\frac{|A^c||B^c|}{(\tau_2-\tau_1)^n},\frac{|B^c|^{1/n}w(A^c)}{\tau_2-\tau_1}\big\}}$, respectively).
\end{theorem}

\begin{proof}
We organize the proof by the following two steps:

\vskip 5pt
  In Step 1,  we aim to prove that for each $z\in L^2(\mathbb R^n;\mathbb C)$,
\begin{eqnarray}\label{0229-sch-control-3}
 & &  \int_{\mathbb R^n}  |z(x)|^2 \,\mathrm dx
 \leq
 %C_1
 C e^{C r_1 r_2  \frac{1}{T-S}  }
 \Big(\int_{ B_{r_1}^c({x^\prime})}  |\varphi(x,\tau_1;T,z)|^2 \,\mathrm dx
 + \int_{ B_{r_2}^c({x^{\prime\prime}})}  |\varphi(x,\tau_2;T,z)|^2 \,\mathrm dx\Big),~~~
\end{eqnarray}
where
%$C_1\triangleq C\big(n,{x^\prime},{x^{\prime\prime}},r_1,r_2,\tau_1,\tau_2\big)>0$ is given by (\ref{0229-sch-th1-0}).
 $C\triangleq C(n)$ is given by Theorem~\ref{theorem1}.
 To this end, we set
 $$
 u_1(x) \triangleq \varphi(x,\tau_1;T,z),\;\;x\in\mathbb R^n.
  $$
  Then it follows from (\ref{0229-sch-1}) and (\ref{0229-sch-adjoint-1}) that for each $t\in[0,\tau_2-\tau_1]$,
\begin{eqnarray}\label{0229-sch-th1-control-2}
 u(x,t;u_1)= (e^{i\Delta t}u_1)(x)=(e^{i\Delta t} e^{i\Delta (\tau_1-T)}z )(x)=\varphi(x,t+\tau_1;T,z),~x\in\mathbb R^n.
\end{eqnarray}
 By Theorem~\ref{theorem1} (where $u_0=u_1$ and $T=\tau_2-\tau_1$), we find that
\begin{eqnarray*}
 \int_{\mathbb R^n}  |u_1(x)|^2 \,\mathrm dx
 &\leq&
 %C_1
 C e^{C r_1 r_2  \frac{1}{T-S}  }
   \Big(\int_{B_{r_1}^c({x^\prime})}  |u_1(x)|^2 \,\mathrm dx
 + \int_{B_{r_2}^c({x^{\prime\prime}})}  |u(x,\tau_2-\tau_1;u_1)|^2 \,\mathrm dx\Big).
\end{eqnarray*}
This, along with (\ref{0229-sch-th1-control-2}), implies that
\begin{eqnarray*}
 & &  \int_{\mathbb R^n}  |\varphi(x,\tau_1;T,z)|^2 \,\mathrm dx
 \leq
 %C_1
 C e^{C r_1 r_2  \frac{1}{T-S}  }
    \Big(\int_{ B_{r_1}^c({x^\prime})}  |\varphi(x,\tau_1;T,z)|^2 \,\mathrm dx
 + \int_{ B_{r_2}^c({x^{\prime\prime}})}  |\varphi(x,\tau_2;T,z)|^2 \,\mathrm dx\Big).
\end{eqnarray*}
 Because of  the conservation law
  of
   the Schr\"{o}dinger equation, the above inequality leads to   (\ref{0229-sch-control-3}).

\vskip 5pt

In Step 2, we aim to use Lemma \ref{lemma-0428-fn} and (\ref{0229-sch-control-3}) to prove (\ref{0229-sch-th2-control-1}) and (\ref{0229-sch-th2-control-2}).
For this purpose, we let
\begin{eqnarray}\label{tinajin5.17}
X\triangleq L^2(\mathbb R^n;\mathbb C)=X^*,~
Y\triangleq L^2(\mathbb R^n;\mathbb C)\times L^2(\mathbb R^n;\mathbb C)=Y^*
\;\;\mbox{and}\;\;
Z\triangleq L^2(\mathbb R^n;\mathbb C)=Z^*
\end{eqnarray}
 and define two operators $ \mathcal R: Z\rightarrow X$ and $ \mathcal O: Z\rightarrow Y$
  as follows:
\begin{eqnarray}\label{0428-th1.1-apply-1}
 \mathcal R z \triangleq z;\;\;\;
 \mathcal O z \triangleq \left( \chi_{B_{r_1}^c({x^\prime})}(\cdot)\varphi(\cdot,\tau_1;T,z),
 \chi_{B_{r_2}^c({x^{\prime\prime}})}(\cdot)\varphi(\cdot,\tau_2;T,z) \right)\;\;\mbox{for each}\;\; z\in Z.
\end{eqnarray}
By (\ref{0428-th1.1-apply-1}) and (\ref{tinajin5.17}), one can directly check, that
\begin{eqnarray}\label{0428-th1.1-apply-1-1}
 \mathcal R^* f = f,~\forall\, f\in L^2(\mathbb R^n;\mathbb C)
 ;\;\;\;
 \mathcal O^* (h_1,h_2) = u_1(\cdot,T;0,h_1,h_2),~\forall\,(h_1,h_2)\in L^2(\mathbb R^n;\mathbb C) \times L^2(\mathbb R^n;\mathbb C).
\end{eqnarray}

Arbitrarily fix $k\in\mathbb N^+$. From (\ref{0229-sch-control-3}) and (\ref{0428-th1.1-apply-1}), we find that for each $z\in L^2(\mathbb R^n;\mathbb C)$,
\begin{eqnarray}\label{0428-th1.1-apply-5}
 \|\mathcal R z\|_{X} ^2
 \leq
C e^{C r_1 r_2  \frac{1}{T-S}  }
  \| \mathcal O z\|^2_{Y}
 + \frac{1}{k} \|z\|_{Z}^2.
\end{eqnarray}
where
$C>0$
 is given by (\ref{0229-sch-control-3})
 and  $\|\cdot\|_{Y}$ denotes the usual norm of $L^2(\mathbb R^n;\mathbb C) \times L^2(\mathbb R^n;\mathbb C)$.

Arbitrarily fix $u_0$, $u_T\in L^2(\mathbb R^n;\mathbb C)$. Define a function over $\mathbb{R}^n$ in the following manner:
 \begin{eqnarray}\label{0428-th1.1-apply-4}
  f(x)  \triangleq u_T(x)-e^{i\Delta T} u_0(x),\;\;x\in \mathbb{R}^n.
 \end{eqnarray}
 By Lemma \ref{lemma-0428-fn} and (\ref{0428-th1.1-apply-5}), it follows that  there exists $(h_{1,k}^f, h_{2,k}^f)\in Y^*$ so that
\begin{eqnarray}\label{0428-th1.1-apply-2}
   C^{-1} e^{-C r_1 r_2  \frac{1}{T-S}}
 \|(h_{1,k}^f, h_{2,k}^f)\|^2_{Y^*}
 +  k  \| \mathcal R^* f  -  \mathcal O^* (h_{1,k}^f, h_{2,k}^f) \|_{Z^*}^2
 \leq \|f\|_{X^*}^2.
\end{eqnarray}
By (\ref{tinajin5.17}) and (\ref{0428-th1.1-apply-2}), one can easily find that there exits a subsequence $\{k_j\}_{j=1}^\infty$ of $\mathbb N^+$ and $(h_{1}^f, h_{2}^f)\in L^2(\mathbb R^n;\mathbb C) \times L^2(\mathbb R^n;\mathbb C)$ so that
\begin{eqnarray*}
 (h_{1,k_j}^f, h_{2,k_j}^f)  \rightarrow (h_{1}^f, h_{2}^f)
 \;\;\mbox{weakly in}\;\;  L^2(\mathbb R^n;\mathbb C) \times L^2(\mathbb R^n;\mathbb C),
 \;\;\mbox{as}\;\;  j\rightarrow\infty
\end{eqnarray*}
and so that
\begin{eqnarray*}
  \mathcal R^* f  -  \mathcal O^* (h_{1,k_j}^f, h_{2,k_j}^f)
    \rightarrow \mathcal R^* f  -  \mathcal O^* (h_{1}^f, h_{2}^f)
 \;\;\mbox{weakly in}\;\;  L^2(\mathbb R^n;\mathbb C),
 \;\;\mbox{as}\;\;  j\rightarrow\infty.
\end{eqnarray*}
(Here, we used the fact that the operator $O$ is linear and bounded. This fact follows from (\ref{0428-th1.1-apply-1}).)
These yield that
\begin{eqnarray*}
 \|(h_{1}^f, h_{2}^f)\|_{L^2(\mathbb R^n;\mathbb C) \times L^2(\mathbb R^n;\mathbb C)}
 \leq \liminf_{j\rightarrow \infty}
 \|(h_{1,k_j}^f, h_{2,k_j}^f)\|^2_{L^2(\mathbb R^n;\mathbb C) \times L^2(\mathbb R^n;\mathbb C)},
 \;\;\mbox{as}\;\;  j\rightarrow\infty
\end{eqnarray*}
and that
\begin{eqnarray*}
 \|\mathcal R^* f  -  \mathcal O^* (h_{1}^f, h_{2}^f)\|_{L^2(\mathbb R^n;\mathbb C)}
 \leq  \liminf_{j\rightarrow \infty}
 \|\mathcal R^* f  -  \mathcal O^* (h_{1,k_j}^f, h_{2,k_j}^f)\|_{L^2(\mathbb R^n;\mathbb C)},
 \;\;\mbox{as}\;\;  j\rightarrow\infty.
\end{eqnarray*}
From these and (\ref{0428-th1.1-apply-2}), it follows that
\begin{eqnarray}\label{0428-th1.1-apply-3}
\mathcal R^* f  =  \mathcal O^* (h_{1}^f, h_{2}^f)
\;\;\mbox{and}\;\;
 \|(h_{1}^f, h_{2}^f)\|^2_{L^2(\mathbb R^n;\mathbb C) \times L^2(\mathbb R^n;\mathbb C)}
 \leq C_1 \|f\|_{L^2(\mathbb R^n;\mathbb C)}^2.
\end{eqnarray}
 Now,  (\ref{0229-sch-th2-control-1}) and (\ref{0229-sch-th2-control-2}) follow from (\ref{0428-th1.1-apply-3}), (\ref{0428-th1.1-apply-1-1}) and (\ref{0428-th1.1-apply-4}) at once. This ends the proof of this theorem.

\end{proof}

\begin{remark}
The above theorem can be understood  as follows: For each $u_0$,\,$u_T\in L^2(\mathbb R^n;\mathbb C)$, there exists a  pair of controls (in $L^2(\mathbb R^n;\mathbb C) \times L^2(\mathbb R^n;\mathbb C)$)  steering  the solution of (\ref{0229-sch-control-1}) from  $u_0$ at time $0$ to  $u_T$ at time $T$. Moreover,
a bound of the norm of  this pair of controls is explicitly given.
 \end{remark}

Next,  we will use the inequality  (\ref{0229-two-points-1}) in (i) of Theorem~\ref{theorem3}, as well as Lemma \ref{lemma-0428-fn}, to get  some kind of approximate  controllability for the following impulse controlled Schr\"{o}dinger equation:
\begin{eqnarray}\label{0410-sch-th2-better-equation}
\left\{\begin{array}{ll}
        i\partial_t u(x,t) + \Delta u(x,t) = \delta_{\{t=\tau\}} \chi_{B_{r}^c(0)}(x) h(x,t), &(x,t)\in \mathbb R^n \times (0,T),\\
        u(x,0)=u_0,   &x\in\mathbb R^n,
       \end{array}
\right.
\end{eqnarray}
where  $T>\tau\geq 0$ and $r>0$, both the initial data $u_0$ and the control $h$ are taken from the space $L^2(\mathbb R^n;\mathbb C)$. Write $u_2(\cdot,\cdot; u_0, h)$ for the solution to the equation (\ref{0410-sch-th2-better-equation}). Define, for each $a>0$, a Banach space:
\begin{eqnarray}\label{0523-Xa}
X_a \triangleq \left\{f\in L^2(\mathbb R^n;\mathbb C)
~:~ \int_{\mathbb R^n} e^{a|x|} |f(x)|^2 \,\mathrm dx <\infty \right\},
\end{eqnarray}
endowed
with the norm:
\begin{eqnarray*}
 \|f\|_{X_a} \triangleq  \left( \int_{\mathbb R^n} e^{a|x|} |f(x)|^2 \,\mathrm dx \right)^{1/2},
 \;\;f\in X_a.
\end{eqnarray*}
One can directly check that for each $a>0$, the dual space  of $X_a$ reads
\begin{eqnarray}\label{0523-Xa-dual}
X_a^* = \overline{C_0^\infty(\mathbb R^n;\mathbb C)}^{\|\cdot\|_{X_a^*}},
\end{eqnarray}
with the norm $\|\cdot\|_{X_a^*}$ given by
\begin{eqnarray*}
 \|g\|_{X_a^*} \triangleq  \left( \int_{\mathbb R^n} e^{-a|x|} |g(x)|^2 \,\mathrm dx \right)^{1/2},
\;\;g\in X_a^*.
\end{eqnarray*}

\begin{theorem}\label{0410-sch-th2-better}
Let $r>0$, $a>0$ and $T>\tau\geq0$. Let $C>0$  and $\theta\in(0,1)$ be given by (i) of Theorem \ref{theorem3}. Write
\begin{eqnarray*}
 p \triangleq \theta^{1+\frac{r}{a(T-\tau)}}\in(0,1).
\end{eqnarray*}
 Then    for any $\varepsilon>0$, $u_0,\,u_T\in L^2(\mathbb R^n;\mathbb C)$, there is a control  $h\in L^2(\mathbb R^n;\mathbb C)$ so that
\begin{eqnarray}\label{0410-sch-th2-better-1}
& & \varepsilon^{\frac{1-p}{p}} \int_{\mathbb R^n} |h(x)|^2 \,\mathrm dx
+  \varepsilon^{-1} \| u_2(\cdot,T;u_0,h)-u_T(\cdot) \|^2_{X_a^*}
\nonumber\\
&\leq& C \left( 1+\frac{r^n}{(a(T-\tau))^n} \right)
 \int_{\mathbb R^n} |u_T(x) - e^{i\Delta T}u_0(x)|^2 \,\mathrm dx,
\end{eqnarray}

\end{theorem}

\begin{proof}
First of all, we claim that for each $z\in C_0^\infty(\mathbb R^n;\mathbb C)$ and each $\varepsilon>0$,
\begin{eqnarray}\label{0410-th1.2-control-3}
 & &  \int_{\mathbb R^n}  |z(x)|^2 \,\mathrm dx
 \\
 &\leq&   C \left( 1+\frac{r^n}{(a(T-\tau))^n} \right)
 \left(
 \varepsilon \int_{\mathbb R^n}e^{ a|x|} |z(x)|^2 \,\mathrm dx
 + \varepsilon^{-\frac{1-p}{p} }
 \int_{B_r^c(0)} |\varphi(x,\tau;T,z)|^2 \,\mathrm dx
 \right).
 \nonumber
\end{eqnarray}

 To this end, arbitrarily fix $z\in C_0^\infty(\mathbb R^n;\mathbb C)$. It follows from (\ref{0229-sch-1}) and (\ref{0229-sch-adjoint-1}) that
\begin{eqnarray}\label{0410-th1.2-control-2}
  \overline{u(x,t;\bar z)}= \varphi(x,T-t;T,z),~(x,t)\in\mathbb R^n\times [0,T].
\end{eqnarray}
 Then by (i) of Theorem~\ref{theorem3} (where $u_0$ and $T$ are replaced by $\bar z$ and $T-\tau$, respectively), we find that
\begin{eqnarray*}
& & \int_{\mathbb R^n} | \overline{z(x)}|^2 \, \mathrm dx
 \\
 &\leq&   C \left( 1+\frac{r^n}{(a(T-\tau))^n} \right)
  \left(\int_{B_{r}^c(0)} |\overline{ u(x,T-\tau;\bar z)}|^2 \, \mathrm dx\right)^{p}
 \left(\int_{\mathbb R^n}e^{ a|x|} |\overline{ z(x)}|^2 \,\mathrm dx\right)^{1-p},
\end{eqnarray*}
from which and  (\ref{0410-th1.2-control-2}), we find that
\begin{eqnarray*}
& & \int_{\mathbb R^n} |z(x)|^2 \, \mathrm dx
 \\
 &\leq&   C \left( 1+\frac{r^n}{(a(T-\tau))^n} \right)
  \left(\int_{B_{r}^c(0)} |\varphi(x,\tau;T,z)|^2 \, \mathrm dx\right)^{p}
 \left(\int_{\mathbb R^n}e^{ a|x|} |z(x)|^2 \,\mathrm dx\right)^{1-p}.
\end{eqnarray*}
This, along with  the Young inequality, yields (\ref{0410-th1.2-control-3}).

\vskip 5pt
Next,  we will use  Lemma \ref{lemma-0428-fn} and  (\ref{0410-th1.2-control-3}) to prove (\ref{0410-sch-th2-better-1}). For this purpose, we let
\begin{eqnarray}\label{tianjin5.28}
X\triangleq
L^2(\mathbb R^n;\mathbb C)=X^*,~
Y\triangleq
L^2(\mathbb R^n;\mathbb C)=Y^*
\;\;\mbox{and}\;\;
Z\triangleq
 X_a,
\end{eqnarray}
where the space $X_a$ is given by (\ref{0523-Xa}).
  %the completion of $C_0^\infty(\mathbb R^n;\mathbb C)$ in the norm $\|\cdot\|_{X_a}$ given by
%\begin{eqnarray}\label{0428-th1.2-apply-2}
% \|z\|_{X_a}  \triangleq  \int_{\mathbb R^n} |z(x)|^2 e^{a|x|} \,\mathrm dx\;\;\mbox{for each}\;\; z\in C_0^\infty(\mathbb R^n;\mathbb C).
%\end{eqnarray}
%One can check that the dual space $X_a^*$ of $X_a$ is equipped with the following norm
%\begin{eqnarray}\label{0428-th1.2-apply-2-1}
% \|g\|_{X_a^*} =  \int_{\mathbb R^n} |g(x)|^2 e^{-a|x|} \,\mathrm dx\;\;\mbox{for each}\;\; g\in X_a^*.
%\end{eqnarray}
Define two operators $ \mathcal R: Z\rightarrow X$ and $ \mathcal O: Z\rightarrow Y$ by
\begin{eqnarray}\label{0428-th1.2-apply-1}
 \mathcal R z  \triangleq  z\;\;\mbox{for each}\;\; z\in X_{a};\;\;\;
 \mathcal O z  \triangleq  \chi_{B_r^c(0)}(\cdot) \varphi(\cdot,\tau;T,z)\;\;\mbox{for each}\;\; z\in X_{a},
\end{eqnarray}
One can directly check that
\begin{eqnarray}\label{0428-th1.2-apply-1-1}
 \mathcal R^* f = f,~\forall\, f\in L^2(\mathbb R^n;\mathbb C)
 ;\;\;\;
 \mathcal O^* h = u_2(\cdot,T;0,h),~\forall\,h\in L^2(\mathbb R^n;\mathbb C).
\end{eqnarray}

 Arbitrarily fix $\varepsilon>0$.
By (\ref{0410-th1.2-control-3}), (\ref{0428-th1.2-apply-1}) and (\ref{0523-Xa}), we can use a standard density argument to verify that
\begin{eqnarray}\label{0428-th1.2-apply-3}
 \| \mathcal R z \|_{L^2(\mathbb R^n;\mathbb C)} ^2
 \leq C_2 \| \mathcal O z\|_{L^2(\mathbb R^n;\mathbb C)} ^2 + \varepsilon_2 \|z\|^2_{X_a}\;\;\mbox{for each}\;\;z\in X_a
\end{eqnarray}
where
\begin{eqnarray}\label{0428-th1.2-apply-4}
 C_2 \triangleq  C \left( 1+\frac{r^n}{(a(T-\tau))^n} \right)  \varepsilon^{-\frac{1-p}{p} }
 \;\;\mbox{and}\;\;
 \varepsilon_2 \triangleq C \left( 1+\frac{r^n}{(a(T-\tau))^n} \right) \varepsilon.
\end{eqnarray}
Arbitrarily fix  $u_0$ and $u_T$ in $L^2(\mathbb R^n;\mathbb C)$.
Define a function $f$ by
\begin{eqnarray}\label{0428-th1.2-apply-5}
 f \triangleq u_T- e^{i\Delta T} u_0\;\;\mbox{over}\;\;\mathbb{R}^n.
\end{eqnarray}
According to Lemma \ref{lemma-0428-fn} and (\ref{0428-th1.2-apply-3}), there exists $h^f\in L^2(\mathbb R^n;\mathbb C)$
(depending on $\varepsilon$, $u_0$ and $u_T$)
so that
\begin{eqnarray*}
 \frac{1}{C_2} \|h^f\|_{Y^*} ^2  +  \frac{1}{\varepsilon_2} \| \mathcal R^* f- \mathcal O^* h^f \|^2_{Z^*}
 \leq  \|f\|_{X^*} ^2.
\end{eqnarray*}
From this, (\ref{tianjin5.28}), (\ref{0428-th1.2-apply-1-1}),
 (\ref{0428-th1.2-apply-4}), (\ref{0428-th1.2-apply-5}) and (\ref{0523-Xa-dual}), we obtain  (\ref{0410-sch-th2-better-1}). This ends the proof of this theorem.

\end{proof}

\begin{remark}
The above theorem can be understood  follows: For each $u_0$,\,$u_T\in L^2(\mathbb R^n;\mathbb C)$ and $\varepsilon>0$, there exists a control (in $L^2(\mathbb R^n;\mathbb C)$)
 steering the solution of (\ref{0410-sch-th2-better-equation}) from  $u_0$ at time $0$ to the
  target $B_\varepsilon^{X_a^*}(u_T)$ at time $T$. (Here, $B^{X_a^*}_\varepsilon(u_T)$ denotes the closed ball in $X_a^*$,  centered at $u_T$ and of radius $\varepsilon$.) Moreover,
    a bound of the norm of this control is  explicitly given.

\end{remark}

Finally,  we will use the inequality  (\ref{0405-sch-th4-control}) in Theorem~\ref{theorem4},  as well as Lemma \ref{lemma-0428-fn},
 to get some kind of approximate null controllability for the following impulse controlled Schr\"{o}dinger equation:
\begin{eqnarray}\label{0410-sch-th4-better-equation}
\left\{\begin{array}{ll}
        i\partial_t u(x,t) + \Delta u(x,t) = \delta_{\{t=0\}} \chi_{B_{r_1}({x^\prime})}(x) h(x,t), &(x,t)\in \mathbb R^n \times (0,T),\\
        u(x,0)=u_0,   &x\in\mathbb R^n,
       \end{array}
\right.
\end{eqnarray}
where  $T>0$, ${x^\prime}\in\mathbb R^n$ and $r_1>0$, both the initial data $u_0$ and the control $h$ are taken from the space $L^2(\mathbb R^n;\mathbb C)$. Write $u_3(\cdot,\cdot; u_0, h)$ for the solution to the equation (\ref{0410-sch-th4-better-equation}). Define, for each $r_2>0$ and ${x^{\prime\prime}}\in \mathbb R^n$,  the following subspace:
\begin{eqnarray}\label{0523-initial-B}
 \widetilde{L}^2(B_{r_2}({x^{\prime\prime}});\mathbb C)
 \triangleq \{f\in L^2(\mathbb R^n;\mathbb C)
 ~:~  f=0 \;\;\mbox{over}\;\; B_{r_2}^c({x^{\prime\prime}})\}.
\end{eqnarray}

\begin{theorem}\label{0410-sch-th4-better}
Let ${x^\prime},\,{x^{\prime\prime}}\in\mathbb R^n$, $r_1,\,r_2>0$, $a>0$ and $T>0$. Let $C>0$  and $p>0$ be given by Theorem \ref{theorem4}.
Then    for each $\varepsilon>0$ and
$u_0\in \widetilde{L}^2(B_{r_2}({x^{\prime\prime}});\mathbb C)$, there is a control  $h\in L^2(\mathbb R^n;\mathbb C)$ so that
\begin{eqnarray}\label{0410-sch-th4-better-1}
 \varepsilon^{\frac{1-\theta^p}{\theta^p}} \int_{\mathbb R^n} |h(x)|^2 \,\mathrm dx
+  \varepsilon^{-1} \|u_3(\cdot,T;u_0,h)\|^2_{X_a^*}
\leq C r_2^n \big((aT)\wedge r_1\big)^{-n}  \int_{B_{r_2}({x^{\prime\prime}})} |u_0(x)|^2 \,\mathrm dx.
\end{eqnarray}

\end{theorem}

\begin{proof}
First of all, we claim that for each $z\in C_0^\infty(\mathbb R^n;\mathbb C)$ and each $\varepsilon>0$,
\begin{eqnarray}\label{0410-th1.4-control-3}
 & &  \int_{B_{r_2}({x^{\prime\prime}})}  |\varphi(x,0;T,z)|^2 \,\mathrm dx
 \\
 &\leq&  C r_2^n \big((aT)\wedge r_1\big)^{-n}  \Big( \varepsilon^{-\frac{1-\theta^p}{\theta^p}}  \int_{B_{r_1}({x^\prime})}  |\varphi(x,0;T,z)|^2 \,\mathrm dx
 + \varepsilon \int_{ \mathbb R^n}  e^{a|x|} |z(x)|^2 \,\mathrm dx\Big).
 \nonumber
\end{eqnarray}
 To this end, we arbitrarily fix $z\in C_0^\infty(\mathbb R^n;\mathbb C)$. It follows from (\ref{0229-sch-1}) and (\ref{0229-sch-adjoint-1}) that
\begin{eqnarray}\label{0410-th1.4-control-2}
  \overline{u(x,t;\bar z)}= \varphi(x,T-t;T,z),~(x,t)\in\mathbb R^n\times [0,T].
\end{eqnarray}
 Then by Theorem~\ref{theorem4} (where $u_0=\bar z$), we find that
\begin{eqnarray*}
& & \int_{B_{r_2}({x^{\prime\prime}})} |\overline{ u(x,T;\bar z)}|^2 \, \mathrm dx
 \\
 &\leq&   C r_2^n \big((aT)\wedge r_1\big)^{-n}
  \left(\int_{B_{r_1}({x^\prime})} | \overline{u(x,T;\bar z)}|^2 \, \mathrm dx\right)^{\theta^p}
 \left(\int_{\mathbb R^n}e^{ a|x|} | \overline{z(x)}|^2 \,\mathrm dx\right)^{1-\theta^p}.
\end{eqnarray*}
This, along with (\ref{0410-th1.4-control-2}), leads to that
\begin{eqnarray*}
& & \int_{B_{r_2}({x^{\prime\prime}})} |\varphi(x,0;T,z)|^2 \, \mathrm dx
 \\
 &\leq&   C r_2^n \big((aT)\wedge r_1\big)^{-n}
  \left(\int_{B_{r_1}({x^\prime})} |\varphi(x,0;T,z)|^2 \, \mathrm dx\right)^{\theta^p}
 \left(\int_{\mathbb R^n}e^{ a|x|} |z(x)|^2 \,\mathrm dx\right)^{1-\theta^p}.
\end{eqnarray*}
Now (\ref{0410-th1.4-control-3}) follows from the above inequality and the Young inequality at once.

\vskip 5pt
Next,  we will use Lemma \ref{lemma-0428-fn} and (\ref{0410-th1.4-control-3}) to prove (\ref{0410-sch-th4-better-1}). For this purpose, we let
\begin{eqnarray}\label{tianjin5.39}
X\triangleq
\widetilde{L}^2(B_{r_2}({x^{\prime\prime}});\mathbb C)=X^*,~
Y\triangleq
L^2(\mathbb R^n;\mathbb C)=Y^*
\;\;\mbox{and}\;\;
Z\triangleq
 X_a,
\end{eqnarray}
where the space $X_a$ is given by (\ref{0523-Xa}).
%is the completion of $C_0^\infty(\mathbb R^n;\mathbb C)$ in the norm $\|\cdot\|_{X_a}$ given by
%\begin{eqnarray}\label{0428-th1.4-apply-2}
% \|z\|_{X_a}  \triangleq  \int_{\mathbb R^n} |z(x)|^2 e^{a|x|} \,\mathrm dx\;\;\mbox{for each}\;\; z\in C_0^\infty(\mathbb R^n;\mathbb C).
%\end{eqnarray}
%One can check that the dual space $X_a^*$ of $X_a$ is equipped with the following norm
%\begin{eqnarray}\label{0428-th1.4-apply-2-1}
% \|g\|_{X_a^*} =  \int_{\mathbb R^n} |g(x)|^2 e^{-a|x|} \,\mathrm dx\;\;\mbox{for each}\;\; g\in X_a^*.
%\end{eqnarray}
Define two operators $ \mathcal R: Z\rightarrow X$ and $ \mathcal O: Z\rightarrow Y$ by
\begin{eqnarray}\label{0428-th1.4-apply-1}
 \mathcal R z  &\triangleq& \chi_{B_{r_2}^c({x^{\prime\prime}})}(\cdot) \varphi(\cdot,0;T,z)\;\;\mbox{for each}\;\;z\in X_{a};
 \nonumber\\
 \mathcal O z  &\triangleq&  \chi_{B_{r_1}^c({x^\prime})}(\cdot) \varphi(\cdot,0;T,z)\;\;\mbox{for each}\;\; z\in X_{a},
\end{eqnarray}
One can directly check that
\begin{eqnarray}\label{0428-th1.4-apply-1-1}
 \mathcal R^* f = u_3(\cdot,T;f,0),~\forall\, f\in \widetilde{L}^2(B_{r_2}({x^{\prime\prime}});\mathbb C)
 ;\;\;\;
 \mathcal O^* h = u_3(\cdot,T;0,h),~\forall\,h\in L^2(\mathbb R^n;\mathbb C).
\end{eqnarray}

Arbitrarily fix $\varepsilon>0$. By (\ref{0410-th1.4-control-3}), (\ref{0428-th1.4-apply-1}) and (\ref{0523-Xa}), we can use a standard density argument to verify that
\begin{eqnarray}\label{0428-th1.4-apply-3}
 \| \mathcal R z \|_{X} ^2
 \leq C_3 \| \mathcal O z\|_{Y} ^2 + \varepsilon_3 \|z\|^2_{Z}\;\;\mbox{for all}\;\; z\in Z,
\end{eqnarray}
where
\begin{eqnarray}\label{0428-th1.4-apply-4}
 C_3 \triangleq  C r_2^n \big((aT)\wedge r_1\big)^{-n}  \varepsilon^{-\frac{1-\theta^p}{\theta^p}}
 \;\;\mbox{and}\;\;
 \varepsilon_3 \triangleq C r_2^n \big((aT)\wedge r_1\big)^{-n} \varepsilon.
\end{eqnarray}
Arbitrarily fix $u_0\in \widetilde{L}^2(B_{r_2}({x^{\prime\prime}});\mathbb C)$ (given by (\ref{0523-initial-B})).
From Lemma \ref{lemma-0428-fn} and (\ref{0428-th1.4-apply-3}), we find that there exists $h^{u_0}$ (depending on $\varepsilon$ and $u_0$)
so that
\begin{eqnarray*}
 \frac{1}{C_3} \|h^{u_0}\|_{Y^*} ^2  +  \frac{1}{\varepsilon_3} \| \mathcal R^* u_0- \mathcal O^* h^{u_0} \|^2_{Z^*}
 \leq  \|u_0\|_{X^*} ^2.
\end{eqnarray*}
This, along with (\ref{tianjin5.39}), (\ref{0428-th1.4-apply-1-1}), (\ref{0428-th1.4-apply-4}) and (\ref{0523-Xa-dual}), yields  (\ref{0410-sch-th4-better-1}). This ends the proof of this theorem.

\end{proof}

\begin{remark}
The above theorem can be understood  as follows: For each $u_0\in \widetilde{L}^2(B_{r_2}({x^{\prime\prime}});\mathbb C)$ and $\varepsilon>0$, there exists a control (in $L^2(\mathbb R^n;\mathbb C)$)  steering the solution of (\ref{0410-sch-th4-better-equation}) from  $u_0$ at time $0$ to the target $B_\varepsilon^{X_a^*}(0)$ at time $T$.  Moreover, a bound of the norm of this control is explicitly given.

\end{remark}

\subsection{The applications of Theorem \ref{proposition3-1}-Theorem~\ref{theorem6} to controllability}

First, we will use the inequality (\ref{0229-sch-th4-control-3}) in Theorem~\ref{proposition3-1},
as well as Lemma \ref{lemma-0428-fn}, to get  some kind of exact controllability for the following impulse controlled Schr\"{o}dinger equation:
\begin{eqnarray}\label{0229-sch-th4-control-2-1}
\left\{\begin{array}{ll}
        i\partial_t u(x,t) + \Delta u(x,t) = \delta_{\{t=\tau\}} \chi_{B_r^c(0)}(x) h(x,t), &(x,t)\in \mathbb R^n \times (0,T),\\
        u(x,0)=u_0,   &x\in\mathbb R^n,
       \end{array}
\right.
\end{eqnarray}
where  $T$ and $\tau$  be two numbers with $0\leq\tau<T$, $r>0$, both the initial data $u_0$ and the control $h$ are taken from the space $L^2(\mathbb R^n;\mathbb C)$. Write $u_4(\cdot,\cdot; u_0, h)$ for the solution to the equation (\ref{0229-sch-th4-control-2-1}).

\begin{theorem}\label{0308-sch-th4}
 Let $0\leq\tau<T$, $r>0$ and $N>0$. Let  $C\triangleq C(n)>0$ be given by Theorem \ref{proposition3-1}.
   Then  for each  $u_0,\,u_T\in L^2(\mathbb R^n;\mathbb C)$, there is a control  $h\in L^2(\mathbb R^n;\mathbb C)$ so that
\begin{eqnarray}\label{0229-sch-th4-control-1}
 u_4(x,T;u_0,h)=u_T,~x\in B_N(0)
\end{eqnarray}
and so that
\begin{eqnarray}\label{0229-sch-th4-control-2}
\|h\|_{L^2(\mathbb R^n;\mathbb C)}
\leq e^{\frac{C}{2} \big(1+ \frac{N}{T-\tau}\big)}  \|u_T-e^{i\Delta T}u_0\|_{L^2(\mathbb R^n;\mathbb C)}.
\end{eqnarray}

\end{theorem}

\begin{proof}
First of all, we claim that for each $z\in \widetilde{L}^2(B_N(0);\mathbb C)$,
\begin{eqnarray}\label{0413-th1.3-control-3}
 \int_{\mathbb R^n} |z(x)|^2 \, \mathrm dx
 \leq  e^{C \big(1+ \frac{rN}{T-\tau}\big)}  \int_{B_r^c(0)} |\varphi(x,\tau;T,z)|^2 \, \mathrm dx.
\end{eqnarray}
 (Here, $\widetilde{L}^2(B_N(0);\mathbb C)$ is given by (\ref{0523-initial-B}), with $B_{r_2}({x^{\prime\prime}})$ being replaced by $B_N(0)$.)
 To this end, arbitrarily fix $z\in \widetilde{L}^2(B_N(0);\mathbb C)$. It follows from (\ref{0229-sch-1}) and (\ref{0229-sch-adjoint-1}) that
\begin{eqnarray}\label{0413-th1.3-control-2}
 \overline{ u(x,t;\bar z)}= \varphi(x,T-t;T,z),~(x,t)\in\mathbb R^n\times [0,T].
\end{eqnarray}
 Then by Theorem~\ref{proposition3-1} (where $u_0$ and $T$ are replaced by $\bar z$ and $T-\tau$, respectively), we find that
\begin{eqnarray*}
\int_{\mathbb R^n} |\overline{ z(x)}|^2 \, \mathrm dx
 \leq  e^{C \big(1+ \frac{rN}{T-\tau}\big)}  \int_{B_r^c(0)} |\overline{ u(x,T-\tau;\bar z)}|^2 \, \mathrm dx,
\end{eqnarray*}
where $C>0$ is given by Theorem \ref{proposition3-1}.
This, along with (\ref{0413-th1.3-control-2}), leads to  (\ref{0413-th1.3-control-3}).

\vskip 5pt

Next, we will use Lemma \ref{lemma-0428-fn} and (\ref{0413-th1.3-control-3}) to prove (\ref{0229-sch-th4-control-1}) and (\ref{0229-sch-th4-control-2}).\
Let
\begin{eqnarray}\label{tianjin5.49}
X\triangleq
L^2(\mathbb R^n;\mathbb C)=X^*,~
Y\triangleq
L^2(\mathbb R^n;\mathbb C)=Y^*
\;\;\mbox{and}\;\;
Z\triangleq
 \widetilde{L}^2(B_N(0);\mathbb C)=Z^*.
\end{eqnarray}
 Define two operators $ \mathcal R: Z\rightarrow X$ and $ \mathcal O: Z\rightarrow Y$ by
\begin{eqnarray}\label{0428-th1.3-apply-1}
 \mathcal R z &\triangleq& z\;\;\mbox{for each}\;\; z\in \widetilde{L}^2(B_N(0);\mathbb C);
 \nonumber\\
 \mathcal O z &\triangleq&  \chi_{B_{r}^c(0)}(\cdot)\varphi(\cdot,\tau;T,z)\;\;\mbox{for each}\;\; z\in \widetilde{L}^2(B_N(0);\mathbb C).
\end{eqnarray}
One can directly check that
\begin{eqnarray}\label{0428-th1.3-apply-1-1}
 \mathcal R^* f = \chi_{B_N(0)} f,~\forall\, f\in L^2(\mathbb R^n;\mathbb C)
 ;\;\;\;
 \mathcal O^* h =\chi_{B_N(0)} u_4(\cdot,T;0,h),~\forall\,h\in L^2(\mathbb R^n;\mathbb C).
\end{eqnarray}
%where the space $X_a$ is the completion of $C_0^\infty(\mathbb R^n;\mathbb C)$ in the norm $\|\cdot\|_{X_a}$ given by
%\begin{eqnarray}\label{0428-th1.3-apply-1-1}
% \|z\|_{X_a}  \triangleq  \int_{\mathbb R^n} |z(x)|^2 e^{a|x|} \,\mathrm dx,
% ~\forall\, z\in C_0^\infty(\mathbb R^n;\mathbb C).
%\end{eqnarray}
 From (\ref{0413-th1.3-control-3}) and (\ref{0428-th1.3-apply-1}), we find that
\begin{eqnarray}\label{0428-th1.3-apply-5}
 \|\mathcal R z\|_{X} ^2
 \leq e^{C \big(1+ \frac{rN}{T-\tau}\big)} \| \mathcal O \tilde z\|^2_{Y}
 + \frac{1}{k} \|z\|_{Z}^2\;\;\mbox{for all}\;\;k\in\mathbb N^+,\; z\in Z.
 \end{eqnarray}

 Arbitrarily fix $u_0,\,u_T\in L^2(\mathbb R^n;\mathbb C)$.
 Define a function $f$ by
 \begin{eqnarray}\label{0428-th1.3-apply-4}
  f  \triangleq u_T-e^{i\Delta T} u_0\;\;\mbox{over}\;\;\mathbb{R}^n.
 \end{eqnarray}
 By Lemma \ref{lemma-0428-fn} and (\ref{0428-th1.3-apply-5}), it follows that  there exists $h_{k}^f\in L^2(\mathbb R^n;\mathbb C) $ so that
\begin{eqnarray}\label{0428-th1.3-apply-2}
 e^{-C \big(1+ \frac{rN}{T-\tau}\big)} \|h_{k}^f\|^2_{Y^*}
 +  k  \| \mathcal R^* f  -  \mathcal O^* h_{k}^f \|_{Z^*}^2
 \leq \|f\|_{X^*}^2\;\;\mbox{for all}\;\;k\in\mathbb N^+.
\end{eqnarray}
 Since $\{h_{k}^f\}_{k=1}^\infty$ is bounded in $L^2(\mathbb R^n;\mathbb C)$ (see (\ref{0428-th1.3-apply-2}) and (\ref{tianjin5.49})), there exits a subsequence $\{k_j\}_{j=1}^\infty$ of $\mathbb N^+$ and $h^f\in L^2(\mathbb R^n;\mathbb C)$ so that
\begin{eqnarray*}
 h_{k_j}^f  \rightarrow h^f
 \;\;\mbox{weakly in}\;\;  L^2(\mathbb R^n;\mathbb C),
 \;\;\mbox{as}\;\;  j\rightarrow\infty
\end{eqnarray*}
and so that
\begin{eqnarray*}
  \mathcal R^* f  -  \mathcal O^* h_{k_j}^f
  \rightarrow  \mathcal R^* f  -  \mathcal O^* h^f
 \;\;\mbox{weakly in}\;\;  L^2(B_N(0);\mathbb C),
 \;\;\mbox{as}\;\;  j\rightarrow\infty.
\end{eqnarray*}
These yield that
\begin{eqnarray*}
 \|h^f\|_{L^2(\mathbb R^n;\mathbb C) }
 \leq \liminf_{j\rightarrow \infty}
 \|h_{k_j}^f\|^2_{L^2(\mathbb R^n;\mathbb C)};\;\; \|\mathcal R^* f  -  \mathcal O^* h^f\|_{L^2(B_N(0);\mathbb C)}
 \leq  \liminf_{j\rightarrow \infty}
 \|\mathcal R^* f  -  \mathcal O^* h_{k_j}^f\|_{L^2(B_N(0);\mathbb C)}.
 \end{eqnarray*}
From these and (\ref{0428-th1.3-apply-2}), it follows that
\begin{eqnarray}\label{0428-th1.3-apply-3}
\mathcal R^* f  =  \mathcal O^* h^f  \;\;\mbox{over}\;\;  B_N(0)
\;\;\mbox{and}\;\;
 \|h^f\|^2_{L^2(\mathbb R^n;\mathbb C)}
 \leq   e^{C \big(1+ \frac{rN}{T-\tau}\big)}   \|f\|_{L^2(\mathbb R^n;\mathbb C)}^2.
\end{eqnarray}
 Now,  (\ref{0229-sch-th2-control-1}) and (\ref{0229-sch-th2-control-2}) follow from (\ref{tianjin5.49}), (\ref{0428-th1.3-apply-3}), (\ref{0428-th1.3-apply-1-1}) and (\ref{0428-th1.3-apply-4}) at once. This ends the proof of this theorem.

\end{proof}

\begin{remark}
The above theorem can be understood  as follows: For each $u_0,\,u_T\in L^2(\mathbb R^n;\mathbb C)$ and $N>0$, there exists a control in $L^2(\mathbb R^n;\mathbb C)$
 steering the solution of (\ref{0229-sch-th4-control-2-1}) from $u_0$ at time $0$ to  $u_T$  at time $T$ over $B_N(0)$. Moreover, a bound of the
  norm of this control is  explicitly given.

\end{remark}

Next,  we will use the inequality (\ref{0405-sch-th5-control-1}) in Theorem~\ref{theorem5}, as well as Lemma \ref{lemma-0428-fn}, to
get some kind of approximate null controllability for the following impulse controlled Schr\"{o}dinger equation:
\begin{eqnarray}\label{0410-sch-th5-better-equation}
\left\{\begin{array}{ll}
        i\partial_t u(x,t) + \Delta u(x,t) = \delta_{\{t=0\}} \chi_{B_{r}(x_0)}(x) h(x,t), &(x,t)\in \mathbb R^n \times (0,T),\\
        u(x,0)=u_0,   &x\in\mathbb R^n,
       \end{array}
\right.
\end{eqnarray}
where  $T>0$, $x_0\in\mathbb R^n$ and $r>0$, both the initial data $u_0$ and the control $h$ are taken from the space $L^2(\mathbb R^n;\mathbb C)$. Write $u_5(\cdot,\cdot; u_0, h)$ for the solution to the equation (\ref{0410-sch-th5-better-equation}). Before state the main result,
we define, for each $b>0$ and ${x^\prime}\in \mathbb{R}^n$,
 the following space:
\begin{eqnarray*}\label{0523-X-b}
X_{b,{x^\prime}} \triangleq \left\{f\in L^2(\mathbb R^n;\mathbb C)
~:~ \int_{\mathbb R^n} e^{b|x-{x^\prime}|} |f(x)|^2 \,\mathrm dx<\infty \right\},
\end{eqnarray*}
with the norm $\|\cdot\|_{X_{b,{x^\prime}}}$ given by
\begin{eqnarray*}
 \|f\|_{X_{b,{x^\prime}}} \triangleq \left(\int_{\mathbb R^n} e^{b|x-{x^\prime}|} |f(x)|^2 \,\mathrm dx \right)^{1/2},\;\;f\in X_{b,{x^\prime}}.
\end{eqnarray*}
One can directly check that the dual space  of $X_{b,{x^\prime}}$  is as
\begin{eqnarray*}\label{0523-X-b-dual}
X_{b,{x^\prime}}^* = \overline{C_0^\infty(\mathbb R^n;\mathbb C)}^{\|\cdot\|_{X_{b,{x^\prime}}^*}},
\end{eqnarray*}
with the norm $\|\cdot\|_{X_{b,{x^\prime}}^*}$ given by
\begin{eqnarray*}
 \|g\|_{X_{b,{x^\prime}}^*} \triangleq \left(\int_{\mathbb R^n} e^{-b|x-{x^\prime}|} |g(x)|^2 \,\mathrm dx \right)^{1/2},\;\;
 g\in C_0^\infty(\mathbb R^n;\mathbb C).
\end{eqnarray*}

\begin{theorem}\label{0410-sch-th4-better}
Let $x_0, {x^\prime}\in\mathbb R^n$, $r>0$, $a>0$, $b>0$ and $T>0$. Let $C(x_0,{x^\prime},r,a,b,T)$  and $C$ be  given by Theorem \ref{theorem5}.
  Then    for each $\varepsilon\in(0,1)$ and  $u_0\in X_{b,{x^\prime}}$, there is a control  $h\in L^2(\mathbb R^n;\mathbb C)$ so that
\begin{eqnarray}\label{0410-sch-th5-better-1}
\frac{1}{\varepsilon} e^{-(\frac{1}{\varepsilon})^{1+\frac{1}{Cb ((aT)\wedge r)}}}
 \int_{\mathbb R^n} |h(x)|^2 \,\mathrm dx
+  \frac{1}{\varepsilon} \|u_5(\cdot,T;u_0,h)\|^2_{X_a^*}
\leq  C(x_0,{x^\prime},r,a,b,T)   \|u_0\|^2_{X_{b,{x^\prime}}}.
\end{eqnarray}
\end{theorem}

\begin{proof}
First of all, we claim that for each $z\in C_0^\infty(\mathbb R^n;\mathbb C)$ and each $\varepsilon\in(0,1)$,
\begin{eqnarray}\label{0410-th1.5-control-3}
 & &  \int_{\mathbb R^n} e^{-b|x-{x^\prime}|} |\varphi(x,0;T,z)|^2 \,\mathrm dx
 \\
 &\leq&  C(x_0,{x^\prime},r,a,b,T)   \Big( \varepsilon e^{\varepsilon^{-1-\frac{1}{Cb ((aT)\wedge r)}}}  \int_{B_{r}(x_0)}  |\varphi(x,0;T,z)|^2 \,\mathrm dx
 + \varepsilon \int_{ \mathbb R^n}  e^{a|x|} |z(x)|^2 \,\mathrm dx\Big).
 \nonumber
\end{eqnarray}
 To this end, we arbitrarily fix $z\in C_0^\infty(\mathbb R^n;\mathbb C)$. It follows from (\ref{0229-sch-1}) and (\ref{0229-sch-adjoint-1}) that
\begin{eqnarray}\label{0410-th1.5-control-2}
 \overline{ u(x,t;\bar z)}= \varphi(x,T-t;T,z),~(x,t)\in\mathbb R^n\times [0,T].
\end{eqnarray}
 Then by Theorem~\ref{theorem5} (where $u_0=\bar z$), we find that for each $\varepsilon\in(0,1)$,
\begin{eqnarray*}
& & \int_{\mathbb R^n} e^{-b|x-{x^\prime}|} |\overline{ u(x,T;\bar z)}|^2 \, \mathrm dx
 \\
 &\leq&   C(x_0,{x^\prime},r,a,b,T)   \Big( \varepsilon e^{\varepsilon^{-1-\frac{1}{Cb ((aT)\wedge r)}}}  \int_{B_{r}(x_0)}  | \overline{u(x,T;\bar z)}|^2 \,\mathrm dx
 + \varepsilon \int_{ \mathbb R^n}  e^{a|x|} |\overline{ z(x)}|^2 \,\mathrm dx\Big).
\end{eqnarray*}
This, along with (\ref{0410-th1.5-control-2}), leads to  (\ref{0410-th1.5-control-3}).

\vskip 5pt
Next,  we will use Lemma \ref{lemma-0428-fn} and (\ref{0410-th1.5-control-3}) to prove (\ref{0410-sch-th5-better-1}). For this purpose, we let
\begin{eqnarray*}
X\triangleq X_{b,{x^\prime}}^*,~
Y\triangleq L^2(\mathbb R^n;\mathbb C)=Y^*
\;\;\mbox{and}\;\;
Z\triangleq  X_a,
\end{eqnarray*}
where the space $X_a$ is given by (\ref{0523-Xa}).
 %is the completion of $C_0^\infty(\mathbb R^n;\mathbb C)$ in the norm $\|\cdot\|_{X_a}$ given by
%\begin{eqnarray}\label{0428-th1.5-apply-2}
% \|z\|_{X_a}  \triangleq  \int_{\mathbb R^n} |z(x)|^2 e^{a|x|} \,\mathrm dx\;\;\mbox{for all}\;\; z\in C_0^\infty(\mathbb R^n;\mathbb C).
%\end{eqnarray}
%One can check that the dual space $X_a^*$ of $X_a$ is equipped with the following norm
%\begin{eqnarray}\label{0428-th1.5-apply-2-1}
% \|g\|_{X_a^*} =  \int_{\mathbb R^n} |g(x)|^2 e^{-a|x|} \,\mathrm dx\;\;\mbox{for all}\;\; g\in X_a^*.
%\end{eqnarray}
Define two operators $ \mathcal R: Z\rightarrow X$ and $ \mathcal O: Z\rightarrow Y$ by
\begin{eqnarray}\label{0428-th1.5-apply-1}
 \mathcal R z \triangleq  \varphi(\cdot,0;T,z);\;\;
 \mathcal O z  \triangleq  \chi_{B_{r}^c(x_0)}(\cdot) \varphi(\cdot,0;T,z)\;\;\mbox{for all}\;\; z\in X_{a}.
\end{eqnarray}
One can directly check that
\begin{eqnarray}\label{0428-th1.5-apply-1-1}
 \mathcal R^* f = u_5(\cdot,T;f,0),~\forall\, f\in X_{b,{x^\prime}}
 ;\;\;\;
 \mathcal O^* h =u_5(\cdot,T;0,h),~\forall\,h\in L^2(\mathbb R^n;\mathbb C).
\end{eqnarray}

Arbitrarily fix $\varepsilon\in (0,1)$.
From (\ref{0410-th1.5-control-3}), (\ref{0428-th1.5-apply-1}) and (\ref{0523-Xa}), we can use a standard density argument to get that
\begin{eqnarray}\label{0428-th1.5-apply-3}
 \| \mathcal R z \|_{X} ^2
 \leq C_5 \| \mathcal O z\|_{Y} ^2 + \varepsilon_5 \|z\|^2_{Z}\;\;\mbox{for each}\;\;z\in Z,
\end{eqnarray}
where
\begin{eqnarray}\label{0428-th1.5-apply-4}
 C_5 \triangleq  C(x_0,{x^\prime},r,a,b,T) \varepsilon e^{\varepsilon^{-1-\frac{1}{Cb ((aT)\wedge r)}}}
 \;\;\mbox{and}\;\;
 \varepsilon_5 \triangleq C(x_0,{x^\prime},r,a,b,T) \varepsilon.
\end{eqnarray}
Arbitrarily fix  $u_0\in C_0^\infty(\mathbb R^n;\mathbb C)$.
Define a function $f$ by
\begin{eqnarray}\label{0428-th1.5-apply-5}
 f(x)  \triangleq  u_0(x),~x\in\mathbb R^n.
\end{eqnarray}
Then by Lemma \ref{lemma-0428-fn} and (\ref{0428-th1.5-apply-3}), there exists $h^{f}$  (depending on $\varepsilon$ and $u_0$)
so that
\begin{eqnarray*}
 \frac{1}{C_5} \|h^{f}\|_{Y^*} ^2  +  \frac{1}{\varepsilon_5} \| \mathcal R^* f- \mathcal O^* h^{f} \|^2_{Z^*}
 \leq  \|f\|_{X^*} ^2.
\end{eqnarray*}
This, along with (\ref{0428-th1.5-apply-1-1}), (\ref{0428-th1.5-apply-4}), (\ref{0428-th1.5-apply-5}) and (\ref{0523-Xa-dual}), yields that (\ref{0410-sch-th5-better-1}) holds. This ends the proof of this theorem.

\end{proof}

\begin{remark}
The above theorem can be understood as follows: For each $u_0\in X_{b,{x^\prime}}$ and $\varepsilon>0$, there exists a control (in $L^2(\mathbb R^n;\mathbb C)$) steering the solution of (\ref{0410-sch-th5-better-equation}) from $u_0$ at time $0$ to  the  target $B_\varepsilon^{X_a^*}(0)$ at time $T$. Moreover, a bound of the norm of this control is explicitly given.

\end{remark}

Finally,  we will use the inequality (\ref{0405-sch-th5-control-2}) in Theorem~\ref{theorem6},
as well as Lemma \ref{lemma-0428-fn},
  to get some kind of approximate controllability for the following impulse controlled Schr\"{o}dinger equation:
\begin{eqnarray}\label{0411-sch-th6-better-equation}
\left\{\begin{array}{ll}
        i\partial_t u(x,t) + \Delta u(x,t) = \delta_{\{t=\tau\}} \chi_{B_{r}(x_0)}(x) h(x,t), &(x,t)\in \mathbb R^n \times (0,T),\\
        u(x,0)=u_0,   &x\in\mathbb R^n,
       \end{array}
\right.
\end{eqnarray}
where  $T>\tau\geq 0$, $x_0\in\mathbb R^n$ and $r>0$, both the initial data $u_0$ and the control $h$ are taken from the space $L^2(\mathbb R^n;\mathbb C)$. Write $u_6(\cdot,\cdot; u_0, h)$ for the solution to the equation (\ref{0411-sch-th6-better-equation}).
For each $a>0$, we write $Q_a$ for the completion of $C_0^\infty(\mathbb R^n;\mathbb C)$ in the following norm:
\begin{eqnarray}\label{0413-norm-Qa}
 \|f\|_{Q_a}  \triangleq
 \left(
 \int_{ \mathbb R^n}  e^{a|x|} |f(x)|^2 \,\mathrm dx
 + \|f\|^2_{H^{n+3}(\mathbb R^n;\mathbb C)}
 \right)^{\frac{1}{2}},\;\; f\in C_0^\infty(\mathbb R^n;\mathbb C).
\end{eqnarray}
 One can easily check  that the space $Q_a$ is continuously imbedded to $L^2(\mathbb R^n;\mathbb C)$. Denote by $Q_a^*$ the dual space of $Q_a$ with respect to the pivot space $L^2(\mathbb R^n;\mathbb C)$.

\begin{theorem}\label{0411-sch-th6-better}
Let $x_0\in\mathbb R^n$, $r>0$, $a>0$ and $T>\tau\geq0$.
 Let ${C}(x_0,r,a,T-\tau)$  be  given by Theorem \ref{theorem6}, with $T$ being replaced by $T-\tau$.
Then    for each $\varepsilon\in(0,1)$ and  $u_0$,\,$u_T\in L^2(\mathbb R^n;\mathbb C)$, there is a control  $h\in L^2(\mathbb R^n;\mathbb C)$ so that
\begin{eqnarray}\label{0411-sch-th6-better-1}
& & \varepsilon^{-1} e^{-e^{\varepsilon^{-2}}} \int_{\mathbb R^n} |h(x)|^2 \,\mathrm dx
  +  \varepsilon^{-1} \| u_6(\cdot,T;u_0,h)-u_T(\cdot) \|_{Q_a^*}^2
  \nonumber\\
&\leq& {C}(x_0,r,a,T-\tau)
   \| u_T- e^{i\Delta T} u_0 \|_{L^2(\mathbb R^n;\mathbb C)}^2.
\end{eqnarray}

\end{theorem}

\begin{proof}
First of all, we claim that for each $z\in C_0^\infty(\mathbb R^n;\mathbb C)$ and each $\varepsilon\in(0,1)$,
\begin{eqnarray}\label{0411-th1.6-control-3}
   \int_{\mathbb R^n}  |z(x)|^2 \,\mathrm dx
 &\leq&  \overline{C}(x_0,r,a,T-\tau)   \Big(
 \varepsilon e^{e^{\varepsilon^{-2}}} \int_{B_{r}(x_0)}  |\varphi(x,\tau;T,z)|^2 \,\mathrm dx
 \nonumber\\
 & &  + \varepsilon   \big(
 \int_{ \mathbb R^n}  e^{a|x|} |z(x)|^2 \,\mathrm dx
 + \|z\|^2_{H^{n+3}(\mathbb R^n;\mathbb C)}
 \big)
 \Big).
\end{eqnarray}
 To this end, we arbitrarily fix $z\in C_0^\infty(\mathbb R^n;\mathbb C)$. It follows from (\ref{0229-sch-1}) and (\ref{0229-sch-adjoint-1}) that
\begin{eqnarray}\label{0411-th1.6-control-2}
 \overline{ u(x,t;\bar z)}= \varphi(x,T-t;T,z),~(x,t)\in\mathbb R^n\times [0,T].
\end{eqnarray}
 Then by Theorem~\ref{theorem6} (where $(u_0,T)$ is replaced by $(\bar z,T-\tau)$), we find that for each $\varepsilon\in(0,1)$,
\begin{eqnarray*}
& &  \int_{\mathbb R^n}  | \overline{z(x)}|^2 \, \mathrm dx
 \nonumber\\
 &\leq&  \overline{C}(x_0,r,a,T-\tau)   \left(
 \varepsilon \Big(\int_{\mathbb R^n} | \overline{z(x)}|^2 e^{a|x|} \mathrm dx
 +  \|\bar z\|^2_{H^{n+3}(\mathbb R^n;\mathbb C)}\Big)
  + \varepsilon  e^{e^{\varepsilon^{-2}}} \int_{B_{r}(x_0)} |\overline{ u(x,T-\tau;\bar z)}|^2 \, \mathrm dx
   \right),
\end{eqnarray*}
This, along with (\ref{0411-th1.6-control-2}), leads to  (\ref{0411-th1.6-control-3}).

\vskip 5pt

Next,  we will use Lemma \ref{lemma-0428-fn} and (\ref{0411-th1.6-control-3}) to prove (\ref{0411-sch-th6-better-1}).  Let
\begin{eqnarray*}\label{tianjin5.70}
X\triangleq L^2(\mathbb R^n;\mathbb C)=X^*,~
Y\triangleq L^2(\mathbb R^n;\mathbb C)=Y^*
\;\;\mbox{and}\;\;
Z\triangleq Q_a,
\end{eqnarray*}
where $Q_a$ is given by (\ref{0413-norm-Qa}).
Define two operators $ \mathcal R: Z\rightarrow X$ and $ \mathcal O: Z\rightarrow Y$ by
\begin{eqnarray}\label{0428-th1.6-apply-1}
 \mathcal R z  \triangleq z;\;\;
 \mathcal O z  \triangleq  \chi_{B_{r}^c(x_0)}(\cdot) \varphi(\cdot,\tau;T,z)\;\;\mbox{for each}\;\; z\in Z.
\end{eqnarray}
One can directly check that
\begin{eqnarray}\label{0428-th1.6-apply-1-1}
 \mathcal R^* f = f,~\forall\, f\in L^2(\mathbb R^n;\mathbb C)
 ;\;\;\;
 \mathcal O^* h =u_6(\cdot,T;0,h),~\forall\,h\in L^2(\mathbb R^n;\mathbb C).
\end{eqnarray}

%where the space $X_a$ is the completion of $C_0^\infty(\mathbb R^n;\mathbb C)$ in the norm $\|\cdot\|_{X_a}$ given by
%\begin{eqnarray}\label{0428-th1.6-apply-2}
% \|z\|_{X_a}  \triangleq  \int_{\mathbb R^n} |z(x)|^2 e^{a|x|} \,\mathrm dx,
% ~\forall\, z\in C_0^\infty(\mathbb R^n;\mathbb C).
%\end{eqnarray}
%One can check that the dual space $X_a^*$ of $X_a$ is equipped with the following norm
%\begin{eqnarray}\label{0428-th1.6-apply-2-1}
% \|g\|_{X_a^*} =  \int_{\mathbb R^n} |g(x)|^2 e^{-a|x|} \,\mathrm dx,
% ~\forall\, g\in X_a^*.
%\end{eqnarray}

 Arbitrarily fix $\varepsilon\in(0,1)$.
  From (\ref{0411-th1.6-control-3}), (\ref{0428-th1.6-apply-1}) and (\ref{0413-norm-Qa}), we can use a standard density argument to get that
\begin{eqnarray}\label{0428-th1.6-apply-3}
 \| \mathcal R z \|_{X} ^2
 \leq C_6 \| \mathcal O z\|_{Y} ^2 + \varepsilon_6 \|z\|^2_{Z}\;\;\mbox{for each}\;\;z\in Q_a
\end{eqnarray}
where
\begin{eqnarray}\label{0428-th1.6-apply-4}
 C_6 \triangleq  {C}(x_0,r,a,T-\tau)  \varepsilon  e^{e^{\varepsilon^{-2}}}
 \;\;\mbox{and}\;\;
 \varepsilon_6 \triangleq {C}(x_0,r,a,T-\tau)  \varepsilon.
\end{eqnarray}
 Arbitrarily fix $u_0$ and $u_T$ in $L^2(\mathbb R^n;\mathbb C)$. Define a function $f$ by
\begin{eqnarray}\label{0428-th1.6-apply-5}
 f  \triangleq u_T  -  e^{i\Delta T} u_0\;\;\mbox{over}\;\;\mathbb{R}^n.
\end{eqnarray}
Then by Lemma \ref{lemma-0428-fn} and (\ref{0428-th1.6-apply-3}), there exists $h^{f}$ (depending on $\varepsilon$, $u_0$ and $u_T$) so that
\begin{eqnarray*}
 \frac{1}{C_6} \|h^{f}\|_{Y^*} ^2  +  \frac{1}{\varepsilon_6} \| \mathcal R^* f- \mathcal O^* h^{f} \|^2_{Z^*}
 \leq  \|f\|_{X^*} ^2,
\end{eqnarray*}
which, along with (\ref{0428-th1.6-apply-1-1}), (\ref{0428-th1.6-apply-4}) and (\ref{0428-th1.6-apply-5}), leads to (\ref{0411-sch-th6-better-1}). This ends the proof of the theorem.

\end{proof}

\begin{remark}
The above theorem can be understood as follows: For each $u_0,\,u_T\in L^2(\mathbb R^n;\mathbb C)$ and $\varepsilon>0$, there exists a control (in
$L^2(\mathbb R^n;\mathbb C)$) steering the solution of (\ref{0411-sch-th6-better-equation}) from   $u_0$ at time $0$ to  the  target $B_\varepsilon^{Q_a^*}(u_T)$ at time $T$. Here, $B_\varepsilon^{Q_a^*}(u_T)$ denotes the closed ball in $Q_a^*$, centered at $u_T$ and of radius $\varepsilon$.  Moreover,
a bound of the norm of this  control is explicitly given.

\end{remark}

\bigskip
\footnotesize
\noindent\textit{Acknowledgments.}
The first and third authors were supported by the National Natural Science Foundation of China under grant 11571264. The second author was supported by the National Natural Science Foundation of China under grant 11426209, and the Fundamental Research Funds for the Central Universities, China University of Geosciences(Wuhan) under grant  G1323521638.

\end{document}